\documentclass[preprint,review,11pt,authoryear]{elsarticle}

\let\today\relax
\makeatletter
\def\ps@pprintTitle{%
    \let\@oddhead\@empty
    \let\@evenhead\@empty
    \def\@oddfoot{\footnotesize\itshape
         {Preprint submitted to arXiv} \hfill\today}
    \let\@evenfoot\@oddfoot
    }
\makeatother

\usepackage{amsmath,amssymb}
 \usepackage{amsthm}
\usepackage{bbm}       
\usepackage{bm}        
 \usepackage{graphicx}
\usepackage{color}
\definecolor{cornell-red}{RGB}{179,27,27}
\usepackage{subcaption}
\usepackage{mathtools}
\usepackage{dsfont}
\usepackage{color}
\usepackage{natbib}
    \usepackage{multirow}
\usepackage{pgfplots}
\usepackage{caption,subcaption}
\usepackage[left=1in,right=1in,top=1in,bottom=1in]{geometry}
\usepackage{indentfirst}
\usepackage[lined,boxed,commentsnumbered, ruled]{algorithm2e}   

\usepackage{enumitem}  
\usepackage{xcolor}
\usepackage[colorlinks = true,linkcolor = blue,urlcolor  = blue,citecolor = blue,anchorcolor = blue]{hyperref}

\usepackage{booktabs}

\usepackage{natbib}

\graphicspath{ {Figures/} {TikzPlots/} }

\usepackage{titlesec} 
\titlespacing*{\section}
{0pt}{1.0ex}{1.0ex}
\titlespacing*{\subsection}
{0pt}{1.0ex}{1.0ex}

\newtheorem{theorem}{Theorem}
\newtheorem{lemma}{Lemma}
\newtheorem{proposition}{Proposition}
\newtheorem{corollary}{Corollary}

\theoremstyle{definition}
\newtheorem{definition}{Definition}[section]

\newtheorem{example}{Example}
\newtheorem{assumption}{Assumption}
\newtheorem{remark}{\textbf{Remark}}
\newtheorem{axiom}{Axiom}[section]

\theoremstyle{remark}

\makeatletter
\newcommand{\setaxiomtag}[1]{
  \let\oldtheaxiom\theaxiom
  \renewcommand{\theaxiom}{#1}
  \g@addto@macro\endaxiom{
    \addtocounter{axiom}{-1}
    \global\let\theaxiom\oldtheaxiom}
  }
\makeatother

\newcommand{\ra}[1]{\renewcommand{\arraystretch}{#1}}  

\DeclareMathOperator*{\argmax}{arg\,max}
\DeclareMathOperator*{\argmin}{arg\,min}

\DeclareMathOperator{\Var}{Var}

\DeclareMathOperator{\Cov}{Cov}

\DeclareMathOperator{\diam}{diam}

\DeclareMathOperator{\conv}{conv} 
\DeclareMathOperator{\CVaR}{CVaR}
\DeclareMathOperator{\VaR}{VaR}

\mathchardef\mhyphen="2D 

\newcommand{\E}{\mathbb{E}}
\newcommand{\Prob}{\mathbb{P}}
\newcommand{\G}{\mathbb{G}}
\newcommand{\M}{\mathbb{M}}
\newcommand{\N}{\mathbb{N}}
\newcommand{\Q}{\mathbb{Q}}
\newcommand{\R}{\mathbb{R}}
\renewcommand{\S}{\mathbb{S}}

\newcommand{\calB}{\mathcal{B}}

\newcommand{\calF}{\mathcal{F}}

\newcommand{\calH}{\mathcal{H}}

\newcommand{\calK}{\mathcal{K}}

\newcommand{\calP}{\mathcal{P}}
\newcommand{\calS}{\mathcal{S}}
\newcommand{\calU}{\mathcal{U}}
\newcommand{\calX}{\mathcal{X}}
\newcommand{\calY}{\mathcal{Y}}

\newcommand{\bbmd}{\mathbbm{d}}
\newcommand{\bbmD}{\mathbbm{D}}
\newcommand{\bbmH}{\mathbbm{H}}

\newcommand{\sfd}{\mathsf{d}}

\newcommand{\tp}{\top}
\newcommand{\one}{\bm{1}} 
\newcommand{\norm}[1]{\left\lVert#1\right\rVert}
\newcommand{\norms}[1]{\lVert#1\rVert}

\newcommand{\xb}{\pmb{x}}
\newcommand{\yb}{\pmb{y}}

\newcommand{\pb}{\pmb{p}}

\newcommand{\mub}{\pmb{\mu}}
\newcommand{\Sigmab}{\pmb{\Sigma}}

\newcommand{\xbbar}{\overline{\pmb{x}}}
\newcommand{\xib}{\pmb{\xi}}
\newcommand{\xibh}{\widehat{\pmb{\xi}}}
\newcommand{\Probh}{\widehat{\Prob}}
\newcommand{\calXh}{\widehat{\calX}}
\newcommand{\calPh}{\widehat{\calP}}
\newcommand{\upsilonh}{\widehat{\upsilon}}
\newcommand{\rh}{\widehat{r}}
\newcommand{\xih}{\widehat{\xi}}
\newcommand{\Cbar}{\overline{C}}

\renewcommand{\hbar}{\overline{h}}

\newcommand{\muh}{\widehat{\mu}}
\newcommand{\sigmah}{\widehat{\sigma}}

\mathchardef\mhyphen="2D 

\usepackage{titlesec}

\begin{document}

\begin{frontmatter}

\title{On the Trade-Off Between Distributional Belief and Ambiguity: Conservatism, Finite-Sample Guarantees, and Asymptotic Properties}

\author{Man Yiu Tsang}
\cortext[cor1]{Corresponding author. }
\ead{mat420@lehigh.edu}
\author{Karmel S.~Shehadeh\corref{cor1}}
\ead{kas720@lehigh.edu}

\address{Department of Industrial and Systems Engineering, Lehigh University, Bethlehem, PA,  USA}

\begin{abstract}

\noindent We propose and analyze a new data-driven trade-off (TRO) approach for modeling uncertainty that serves as a middle ground between the optimistic approach, which adopts a distributional belief, and the pessimistic distributionally robust optimization approach, which hedges against distributional ambiguity. We equip the TRO model with a TRO ambiguity set characterized by a size parameter controlling the level of optimism and a shape parameter representing distributional ambiguity. We first show that constructing the TRO ambiguity set using a general star-shaped shape parameter with the empirical distribution as its star center is necessary and sufficient to guarantee the hierarchical structure of the sequence of TRO ambiguity sets. Then, we analyze the properties of the TRO model, including quantifying conservatism, quantifying bias and generalization error, and establishing asymptotic properties. Specifically, we show that the TRO model could generate a spectrum of decisions, ranging from optimistic to conservative decisions. Additionally, we show that it could produce an unbiased estimator of the true optimal value. Furthermore, we establish the almost-sure convergence of the optimal value and the set of optimal solutions of the TRO model to their true counterparts. We exemplify our theoretical results using an inventory control problem and a portfolio optimization problem.


\end{abstract}
\end{frontmatter}

\section{Introduction} \label{sec:introduction}

Consider stochastic optimization problems of the following form:
\begin{equation} \label{prob:SO}
    \upsilon^\star = \inf_{\xb\in\calX} \E_{\Prob^\star} \big[f(\xb,\xib)\big].
\end{equation}
In \eqref{prob:SO}, $\xb$ is a vector of decision variables; $\calX \subseteq \R^n$ is a non-empty set of deterministic constraints on $\xb$; $\xib:\Omega\rightarrow\Xi$ is a random vector defined on a measurable space $(\Omega,\calF)$ with support $\Xi\subseteq\R^\ell$; and $f:\R^n\times\R^\ell \rightarrow\R$ is a continuous function in $\xb$ for a given $\xib\in\Xi$ and is measurable in $\xib$ for given $\xb\in\calX$. Function $f$ measures the performance of the system of interest, and $\E_{\Prob^\star}[\cdot]$ denotes the expectation with respect to (w.r.t.) probability measure $\Prob^\star$ defined on $(\Omega,\calF)$. We assume that \eqref{prob:SO} has a finite optimal value with a non-empty set of optimal solutions and $\E_{\Prob^\star}|f(\xb,\xib)|<\infty$ for any $\xb\in\calX$.

In many real-life settings, the true distribution $\Prob^\star$ of $\xib$ is fundamentally unknown. Suppose instead that we have a (potentially small) set of historical observations of $\{\xibh_1,\dots,\xibh_N\}$ of $\xib$. Moreover, suppose we \textit{believe} this data reflects or well approximates the true distribution $\Prob^\star$. Then,  we can estimate $\E [f(\xb,\xib)]$ for any $\xb \in \calX$ by averaging values $f(\xb,\xib_j)$, for $j \in [N]:=\{1,\dots,N\}$. This leads to the following sample average approximation (SAA) 
\begin{equation} \label{prob:SAA}
    \upsilonh_N^\text{SAA} = \inf_{\xb\in\calX} \E_{\Probh_N}\big[f(\xb,\xib)\big] =\inf_{\xb\in\calX} \frac{1}{N} \sum_{i=1}^N f(\xb,\xibh_i) 
\end{equation} 
of the true problem \eqref{prob:SO}, where $\Probh_N=N^{-1}\sum_{i=1}^N\delta_{\xibh_i}$ is the empirical distribution of $\xibh$ and $\delta_a$ is the Dirac measure on $a$. Optimal solutions to \eqref{prob:SAA} are known to be sensitive to the input data used in the model, i.e., they are a function of a particular \textit{distributional belief} \citep{Birge_Louveaux:2011,Shapiro_et_al:2014}. Moreover, as pointed out by \cite{Kuhn_et_al:2019}, even if one employs sophisticated statistical techniques to estimate the uncertainty distribution using historical data, the estimated distribution $\Probh_N$ may differ from the actual distribution $\Prob^\star$. This is concerning because optimal solutions to \eqref{prob:SAA} obtained using an estimated distribution may inherit estimation errors and bias, leading to poor performance and disappointments under unseen data. This phenomenon is known as the \textit{optimizer's curse} \citep{Mohajerin-Esfahani_Kuhn:2018, Smith_Winkler:2006, Van-Parys_et_al:2021}.

One popular approach to address distributional ambiguity is distributionally robust optimization (DRO). In DRO, instead of faithfully adopting a specific distribution (such as the empirical distribution), we consider an ambiguity set encompassing various potential distributions of $\xib$ against which we aim to safeguard. Let $\calP(\Xi)$ be the set of probability measures defined on the measurable space $(\R^\ell,\calB)$ induced by $\xib$ with support $\Xi$, where $\calB=\calB(\R^\ell)$ is the Borel $\sigma$-field. The DRO counterpart of problem \eqref{prob:SO} is as follows: 
\begin{equation} \label{prob:DRO}
    \upsilonh_N^\text{DRO} = \inf_{\xb\in\calX} \sup_{\Prob\in\calP_N}\E_{\Prob}\big[f(\xb,\xib)\big], 
\end{equation} 
where $\calP_N\subseteq\calP(\Xi)$ is the ambiguity set. The versatility and power of DRO stem from its ability to account for one’s incomplete knowledge of the distribution $\Prob$ of $\xib$ by specifying an ambiguity set $\calP_N$. There are two common approaches in the literature for constructing the ambiguity set. The first utilizes partial distributional information such as the moments and support  \citep{Delage_Ye:2010, Goh_Sim:2010, Nie_et_al:2023, Postek_et_al:2018, Postek_et_al:2019, Sun_et_al:2022, Wiesemann_et_al:2014, Xu_et_al:2018} and marginal distribution \citep{Chen_et_al:2023, Kakouris_Rustem:2014}.  This approach often leads to tractable DRO formulations. However, ambiguity sets constructed using this approach do not fully capture what is known about $\xib$ and often lead to conservative decisions. Moreover, as pointed out by \cite{Lam:2021}, DRO models based on such sets do not have large-sample asymptotic convergence properties. The second approach considers distributions that are close to a reference distribution (e.g., empirical distribution) in the sense of a chosen statistical distance. Popular choices of the statistical distance are $\phi$-divergence \citep{Bayraksan_Love:2015, Ben-Tal_et_al:2013} and Wasserstein distance \citep{Blanchet_et_al:2021, Gao_Kleywegt:2022, Mohajerin-Esfahani_Kuhn:2018, Xie_et_al:2021}.  The distance-based approach also leads to tractable reformulations under mild conditions \citep{Bayraksan_Love:2015, Mohajerin-Esfahani_Kuhn:2018, Zhao_Guan:2015}. In addition, some distance-based DRO (e.g., Wasserstein DRO) model enjoys out-of-sample and asymptotic consistency guarantees \citep{Duchi_et_al:2021, Mohajerin-Esfahani_Kuhn:2018, Kuhn_et_al:2019}.

Since its inception, the DRO approach has received significant attention in operations research, economics, finance, and other fields due to its desirable theoretical properties and ability to produce decisions that maintain robust performance under various distributions. Nevertheless, there are also criticisms concerning the conservatism of the DRO approach. These concerns arise because solutions to \eqref{prob:DRO} hedge against the worst-case distribution within a possibly diverse set of distributions $\calP_N$ and hence could be overly conservative. In particular, the worst-case distributions that attain $\sup_{\Prob\in\calP_N}\E_{\Prob}\big[f(\xb,\xib)]$ under many celebrated (moment- and distance-based) ambiguity sets are shown to be discrete with few support points even though the true distribution $\Prob^*$ could be continuous or have a large number of other support points \citep{Bayraksan_Love:2015, Chen_et_al:2011,Das_et_al:2021,  de-Klerk_et_al:2020, Gao_Kleywegt:2022, Long_et_al:2021}. For example, the worst-case distribution of the DRO model for a newsvendor problem that employs a mean-variance ambiguity set is a two-point distribution \citep{Scarf:1959}. Similarly, the results of \cite{Long_et_al:2021} indicate that if $f(\xb,\xib)$ is supermodular in $\xib$ for any $\xb\in\calX$, then under the mean-absolute-deviation ambiguity set, the worst-case distribution is supported on $(2\ell+1)$ points, where $\ell$ is the dimension of $\xib$. Hence, the DRO approach often hedges against pessimistic scenarios that are less likely to be observed in practice and thus result in conservative decisions; that is, the realized objective value obtained by DRO solutions would often be better than the optimal value of problem \eqref{prob:DRO}. 

Recently, several authors proposed approaches to mitigate the conservatism of DRO models. These include globalized DRO \citep{Ding_et_al:2020, Li_Xing:2022, Liu_et_al:2023} and techniques to reduce the size of the ambiguity set, such as imposing additional structural properties on distributions in moment-based ambiguity set \citep{de-Klerk_et_al:2020, Hanasusanto_et_al:2015, Lam_et_al:2021, Li_et_al:2019, Van-Parys_et_al:2016}, restricting distributions in the ambiguity set to specific parametric family  \citep{Iyengar_et_al:2022, Michel_et_al:2021, Michel_et_al:2022}, and developing modified Wasserstein metrics for distance-based ambiguity sets \citep{Liu_et_al:2022b, Wang_et_al:2021}. These pioneering approaches often change the structure of the DRO problem and, as a result, require tailored reformulation and solution techniques to solve the modified problem effectively. Thus, most of these studies mainly focused on reformulating and solving their proposed models without analyzing the models' statistical properties. We refer to \cite{Kuhn_et_al:2024} and \cite{Rahimian_Mehrotra:2022} for a comprehensive review of the theoretical and computational developments in the DRO literature.

Inevitably, there are trade-offs among different approaches to modeling uncertainty. On the one hand, by following a blind distributional belief, one may obtain decisions based on an over-optimistic viewpoint that might lead to disappointing performance in practice. On the other hand, by focusing the optimization on the worst-case distribution, one may obtain pessimistic (conservative) solutions. In this paper, our goal is to introduce and analyze an alternative trade-off approach for modeling uncertainty that serves as a middle ground between the optimistic approach, which adopts a distributional belief, and the pessimistic approach, which protects against distributional ambiguity. We formulate the trade-off (TRO)  problem as follows: 
\begin{equation} \label{model:trade-off_model}
    \upsilonh_N(\theta)= \inf_{\xb\in\calX} \Big \{ (1-\theta) \cdot \E_{\Probh_N}[f(\xb,\xib)]+\theta \cdot \sup_{\Q\in\calP_N}\E_{\Q}[f(\xb,\xib)]  \Big\}=\inf_{\xb\in\calX}\, \sup_{\Prob\in\calP'_{N,\theta}} \E_{\Prob}[f(\xb,\xib)], 
\end{equation}
where $\calP'_{N,\theta}$ is the \textit{TRO ambiguity set} defined as 
\begin{equation} \label{eqn:trade-off_ambig_set}
    \calP'_{N,\theta} = \big\{ (1-\theta)\Probh_N + \theta\Q \mid \Q\in\calP_N\big\}
\end{equation}
for some $\theta\in[0,1]$. The TRO model  \eqref{model:trade-off_model} can be viewed as a new data-driven DRO model equipped with a TRO ambiguity set $\calP'_{N,\theta}$. The TRO ambiguity set $\calP'_{N,\theta}$ in \eqref{eqn:trade-off_ambig_set} is characterized by two parameters: the \textit{shape} parameter $\calP_N$ and the \textit{size} parameter $\theta$. The shape parameter $\calP_N$ represents distributional ambiguity and could be any (data-driven) ambiguity set satisfying some mild assumptions to be made precise later. The size parameter $\theta\in[0,1]$ controls the level of optimism, i.e., it controls the trade-off between solving the problem under a distributional belief and solving it under ambiguity. When $\theta=0$, problem \eqref{model:trade-off_model} reduces to problem \eqref{prob:SAA}. In contrast, when $\theta=1$, problem \eqref{model:trade-off_model} reduces to problem \eqref{prob:DRO}. Between the two extremes, $\theta\in(0,1)$ indicates a trade-off between optimistic and pessimistic perceptions of the objective. Thus, as we later show, by changing the value of $\theta$ in \eqref{model:trade-off_model}, one could obtain a spectrum of optimal solutions, ranging from optimistic to conservative solutions. 

We emphasize that one can construct the TRO ambiguity set $\calP'_{N,\theta}$ using any shape parameter $\calP_N$, including general moment- and distance-based ambiguity sets. This flexibility allows us to derive general results outlined in Section~\ref{subsec:contributions} on the TRO model's conservatism, finite-sample properties, and asymptotic convergence. Moreover, one can adopt the same techniques developed for reformulating and solving classical DRO problems with specific ambiguity sets $\calP_N$ to reformulate and solve the TRO problem with $\calP'_{N,\theta}$  constructed using $\calP_N$ as the shape parameter. To the best of our knowledge and according to recent surveys \citep{Kuhn_et_al:2024, Rahimian_Mehrotra:2022}, our paper is the first to formally introduce the TRO approach and conduct a thorough theoretical investigation of its properties. Notably, \cite{Shehadeh_Tucker:2022} is the first to explore the use of a convex combination of SAA and DRO (with mean-support ambiguity set) to model uncertainty in the context of disaster relief. Their numerical investigations show that the trade-off approach could lead to a less conservative disaster preparation plan than the DRO approach, with better post-disaster response performance than the SAA-based plan. However, that work was problem-specific and did not analyze the theoretical properties of the model. Recently, \cite{Wang_et_al:2023} analyzed the trade-off between the robustness to unseen data and the specificity of the training data in the context of machine learning. As we show in \ref{apdx:Wang_et_al}, \cite{Wang_et_al:2023}'s Bayesian distributionally robust (BDR) model is a special case of our TRO model. In particular, by choosing the shape parameter $\calP_N$ as the distance-based ambiguity set  $B_{\epsilon_N}(\Probh_N)=\{\Prob\in\calP(\Xi)\mid \Delta(\Prob,\Probh_N)\leq \epsilon\}$ (where $\calP(\Xi)$ is the set of probability measures on the support $\Xi$ and $\Delta$ is a statistical distance), our TRO model reduces to  \cite{Wang_et_al:2023}'s BDR model. While our theoretical investigations are valid for any shape parameter (including moment-based ambiguity sets), \cite{Wang_et_al:2023}'s analyses are limited to the special case where $\calP_N$ is a distance-based ambiguity set. In addition, different from \cite{Wang_et_al:2023}, we introduce and study the hierarchical properties of the TRO ambiguity $\calP'_{N,\theta}$, analyze the conservatism of our TRO model, and provide more comprehensive analyses of the finite-sample guarantees and asymptotic properties of the TRO model; see \ref{apdx:Wang_et_al} for detailed discussions. Finally, we note that the form of our TRO ambiguity set resembles those considered in robust statistics, particularly in Huber contamination models \citep{Huber:1964}.  In \ref{apdx:add_discuss:Huber}, we discuss the differences between the Huber contamination model and our TRO model.

\subsection{Contributions}\label{subsec:contributions}
\noindent  We highlight the following main contributions of our theoretical investigations:
\begin{itemize}[itemsep=0mm,topsep=0mm,leftmargin=5mm]

    \item \textit{Quantifying the conservatism of the TRO model}. We investigate the conservatism of the TRO model from two perspectives: the size and properties of the TRO ambiguity set and the characteristics of the model’s optimal value and solutions.  First, we establish that constructing the TRO ambiguity set  $\calP'_{N,\theta}$ using a general star-shaped shape parameter $\calP_N$ with a star center $\Probh_N$ is both necessary and sufficient for the sequence   $\{\calP'_{N,\theta}\mid\theta\in[0,1]\}$ to be non-decreasing (Theorem~\ref{thm:calP_nondecreasing}, part (i)). If, in addition, $\calP_N\ne\{\Probh_N\}$,  the sequence $\{\calP'_{N,\theta}\mid\theta\in[0,1]\}$ increases with $\theta$, i.e., $\calP'_{N,\theta}$ contains more distributions with a larger $\theta$ (Theorem~\ref{thm:calP_nondecreasing}, part (ii)). These results indicate that the TRO model becomes more conservative as $\theta$ increases. Second, we conduct quantitative stability analyses to quantify the difference in the optimal value $\upsilonh_N(\theta)$ and the set of optimal solutions $\calXh_N(\theta)$ and hence the conservatism incurred by perturbation in $\calP_{N,\theta}'$. Specifically, we show the Lipschitz continuity of the optimal value function $\upsilonh_N(\theta)$ and the  H\"{o}lder continuity of $\calXh_N(\theta)$ (Theorem~\ref{thm:sensitivity_in_theta}). In addition, we quantify the difference between the optimal value (resp. set of optimal solutions) and the convex combination of optimal values (resp. sets of optimal solutions) to the SAA and DRO problems resulting from solving each separately (Theorem~\ref{thm:conservatism}). Together, the results of Theorems~\ref{thm:calP_nondecreasing}--\ref{thm:conservatism} show that by solving the TRO model with different $\theta\in[0,1]$, one can obtain a spectrum of decisions that span $[\upsilonh_N(0),\upsilonh_N(1)]$, representing decisions with different levels of conservatism.

    \item \textit{Finite-sample properties}. The optimal value of our TRO model $\widehat{\upsilon}_N(\theta)$ is an estimator of the optimal value $ \upsilon^\star$ of problem \eqref{prob:SO}. We show that under some mild assumptions, there exists $\theta^\text{u}_N\in[0,1]$ such that $\upsilonh_N(\theta^\text{u}_N)$ is an unbiased estimator (Theorem~\ref{thm:debias_theta}). In addition, we show that our TRO model could produce estimators with a smaller bias than the SAA estimator (Corollary~\ref{cor:debias_theta}) and derive the asymptotic convergence rate of $\theta^\text{u}_N$ as $N\rightarrow\infty$ (Theorems~\ref{thm:rate_of_theta_LIL} and \ref{thm:rate_of_theta_AN}). Moreover, our TRO estimator may not have a significantly larger variance than the SAA estimator, especially when $\theta$ is small (Theorem~\ref{thm:TRO_estimator_variance}). These analytical results hold for TRO models with TRO ambiguity sets constructed using general shape parameters, including moment-based ambiguity sets. Moreover, we derive a bound on the generalization error of the TRO model (Theorem~\ref{thm:generalization_error}). Our results suggest that the TRO model might have a tighter generalization bound than the SAA and DRO models. We also show that the generalization error has an exponentially decaying tail for specific choices of the shape parameter.
   
    \item \textit{Asymptotic properties}. We show that as the number of data points $N \rightarrow \infty $, the optimal value $\upsilonh_N(\theta_N)$  and the set of optimal solutions $\calXh_N(\theta_N)$ of the TRO problem \eqref{model:trade-off_model} converge respectively to the true optimal value $\xb^*$ and set of optimal solutions $\calX^*$ of problem \eqref{prob:SO} almost surely (Theorem~\ref{thm:asymptotic_convergence}). In addition, we derive the asymptotic distribution of $\upsilonh_N(\theta_N)$ (Theorem~\ref{thm:asy_dist}). These asymptotic properties hold for TRO models with TRO ambiguity sets constructed using general shape parameters, such as moment-based ambiguity sets. This differs from the existing convergence results established for data-driven DRO models, which mainly employ distance-based ambiguity sets. For the special case when the shape parameter is chosen as a distance-based ambiguity set, we can recover the asymptotics of the optimal value of classical distance-based DRO models (see, e.g., \citealp{Blanchet_Shapiro:2023}). Specifically, we show that $\upsilonh_N(\theta_N)$ converges to different distributions depending on the convergence rates of the size parameter $\theta_N$ and the radius $r_N$ in the shape parameter (Theorem~\ref{thm:asy_dist_distance_based}).
\end{itemize}

\subsection{Structure of the paper}
The remainder of the paper is organized as follows. In Section~\ref{sec:conservatism_TRO}, we investigate the conservatism of the TRO model. In Section~\ref{sec:finite_sample_property}, we analyze the bias of the optimal value of our TRO model $\upsilonh_N(\theta)$ as an estimator of the true optimal value $\upsilon^\star$ and the generalization error of our TRO model. In Section~\ref{sec:asymptotic}, we derive the asymptotic properties of our TRO model. Finally, in Section~\ref{sec:numerics}, we exemplify our theoretical results using an inventory control problem and a mean-risk portfolio optimization problem.

\subsection{Notation}
For a set $S$ in a general convex space $X$, the convex hull of $S$ is defined as $\conv(S)=\big\{\sum_{i=1}^k \alpha_i x_i\mid k\in\N,\, \sum_{i=1}^k\alpha_i = 1,\, \alpha_i\geq 0,\, x_i\in S,\, \forall i\in\{1,\dots,k\}\big\}$. We use $\norms{\cdot}$ to denote a general norm defined on a vector space. For any $a\in\R$, we define $(a)_+=\max\{a,0\}$. In the Euclidean space $\R^n$, we use boldface letter such as $\yb$ to denote a column vector in $\R^n$ and $\norms{\yb}_p$ to denote the $p$-norm of the vector $\yb$.  For any functions $g(N)$ and $h(N)$ defined on $\N$, we write $h(N)=O(g(N))$ (as $N\rightarrow\infty$) if there exists $M>0$ and $N_0\in\N$ such that $h(N)\leq Mg(N)$ for all $N\geq N_0$, and we write $h(N)=o(g(N))$ (as $N\rightarrow\infty$) if $\lim_{N\rightarrow\infty}h(N)/g(N)=0$. Two probability measures $\Prob_1$ and $\Prob_2$ on a measurable space $(\Omega,\calF)$ are equal, i.e., $\Prob_1=\Prob_2$, if $\Prob_1(B)=\Prob_2(B)$ for all $B\in\calF$. For a probability measure $\Prob$, the variance of a random variable $Y$ under $\Prob$ is denoted as $\Var_{\Prob}(Y)$, and the covariance of two random variables $Y_1$ and $Y_2$ is denoted as $\Cov_{\Prob}(Y_1,Y_2)$. For a sequence of random variables $\{Y_n\}_{n\in\N}$ and a sequence of positive real numbers $\{a_n\}_{n\in\N}$, we write $Y_n=o_{\Prob}(a_n)$ if $|Y_n|/a_n\rightarrow0$ in probability.

\section{Conservatism of the TRO Model} \label{sec:conservatism_TRO}

In this section, we investigate the conservatism of the TRO model from two angles:  by analyzing the size and properties of the TRO ambiguity set  (Section~\ref{sec:TRO_ambig_set_property}) and properties of the TRO model’s optimal value and solutions (Section~\ref{sec:conservatism}). We recognize that there is no universally accepted definition or method for quantifying conservatism in the literature. Different studies have approached conservatism in DRO models from varying perspectives. One common interpretation links conservatism to the size of the ambiguity set, with a larger set encompassing a wider range of distributions---potentially including extreme or pathological ones---which could lead to overly conservative decisions (e.g., \citealp{Ding_et_al:2020, Mohajerin-Esfahani_Kuhn:2018, Van-Parys_et_al:2016}). In Section~\ref{sec:TRO_ambig_set_property}, we follow this view and analyze properties of the TRO ambiguity set $\calP'_{N,\theta}$ and the sequence of TRO ambiguity sets $\{\calP'_{N,\theta}\mid \theta\in[0,1]\}$ as a function of the size parameter $\theta$. Another perspective on conservatism interprets the optimal value of the DRO model as a proxy  (e.g., \citealp{Gorissen_et_al:2015, Li_Xing:2022, Van-Parys_et_al:2021}). In minimization (or maximization) problems, a higher (or lower) optimal value reflects more conservative or cautious decisions. From this perspective, the DRO model's optimal value provides an indirect measure of conservatism, indicating how conservatively the model accounts for the worst-case distribution within the defined ambiguity set. In Section~\ref{sec:conservatism}, we adopt this perspective to study the properties of the TRO model's optimal value and its corresponding optimal solution.

\subsection{Hierarchical Property of the TRO Ambiguity Set} \label{sec:TRO_ambig_set_property}

In this section, we analyze properties of $\calP'_{N,\theta}$ and the sequence of TRO ambiguity sets $\{\calP'_{N,\theta}\mid \theta\in[0,1]\}$. To facilitate the discussion, we introduce the notion of star-shapedness (equivalently, star convexity) of the set of distributions $\calP_N$ in Definition~\ref{def:starshaped} and hierarchical properties of $\{\calP'_{N,\theta}\mid \theta\in[0,1]\}$ in Definition~\ref{def:Hierarchical}.

\begin{definition}[Star-Shaped Set] \label{def:starshaped}
The set of distributions $\calP_N$ is called \textit{star-shaped} if there exists $\M\in\calP_N$ such that 
\begin{equation} \label{eqn:starshaped}
(1-\alpha) \M + \alpha \Prob \in \calP_N,\quad\forall \alpha\in[0,1],\, \Prob\in\calP_N. 
\end{equation}
Any $\M\in\calP_N$ satisfying condition~\eqref{eqn:starshaped} is called a \textit{star center} of $\calP_N$.
\end{definition}

\begin{definition}[Hierarchical Properties]\label{def:Hierarchical}
The sequence $\{\calP'_{N,\theta}\mid \theta\in[0,1]\}$ satisfies the \textit{hierarchical property} if $\calP'_{N,\theta}$ is non-decreasing in $\theta$, i.e., $\calP'_{N,\theta_1}\subseteq\calP'_{N,\theta_2}$ for any $0\leq\theta_1<\theta_2\leq1$. The sequence $\{\calP'_{N,\theta}\mid \theta\in[0,1]\}$ satisfies the \textit{strict hierarchical property} if  $\calP'_{N,\theta}$ is increasing in $\theta$, i.e., $\calP'_{N,\theta_1}\subset\calP'_{N,\theta_2}$ for any $0\leq\theta_1<\theta_2\leq1$.
\end{definition}

The hierarchical properties indicate that the size of the TRO ambiguity set $\calP'_{N,\theta}$ increases with $\theta$, i.e., $\calP'_{N,\theta}$ contains more distributions with a larger $\theta$. In other words, with a larger $\theta$, our TRO model hedges against a larger set of distributions and thus could lead to more conservative decisions. In Theorem~\ref{thm:calP_nondecreasing}, we provide necessary and sufficient conditions for the sequence of TRO ambiguity sets $\{\calP'_{N,\theta}\mid \theta\in[0,1]\}$ to satisfy these properties.

\begin{theorem}\label{thm:calP_nondecreasing}
The following assertions hold.
\begin{enumerate}[topsep=1mm,itemsep=0mm]
    \item [(i)] The sequence of TRO ambiguity sets $\{\calP'_{N,\theta}\mid \theta\in[0,1]\}$ satisfies the hierarchical property if and only if $\calP_N$ is star-shaped with a star center $\Probh_N\in\calP_N$.
    \item [(ii)] The sequence of TRO ambiguity sets $\{\calP'_{N,\theta}\mid \theta\in[0,1]\}$ satisfies the strict hierarchical property if and only if $\calP_N$ is star-shaped with a star center $\Probh_N\in\calP_N$ and $\calP_N\ne\{\Probh_N\}$.
\end{enumerate}
\end{theorem}

Theorem~\ref{thm:calP_nondecreasing}  establishes that constructing the TRO ambiguity set $\calP'_{N,\theta}$ using a star-shaped parameter $\calP_N$ with a star center $\Probh_N$ is necessary and sufficient for the sequence of TRO ambiguity sets $\{\calP'_{N,\theta}\mid \theta\in[0,1]\}$ to satisfy the hierarchical property. Part (i) shows that for a general star-shaped shape parameter $\calP_N$, the TRO ambiguity set $\calP'_{N,\theta}$ is non-decreasing in $\theta$, i.e., $\calP'_{N,\theta_1}\subseteq \calP'_{N,\theta_2}$ whenever $\theta_1\leq\theta_2$, indicating that the objective function of the trade-off model \eqref{model:trade-off_model} is non-decreasing in $\theta$. Figure \ref{fig:trade_off_ambig_set_1} illustrates the relationship between the sets $\{\Probh_N\}$, $\calP'_{N,\theta_1}$, $\calP'_{N,\theta_2}$, and $\calP_N$ with $0<\theta_1<\theta_2<1$ as suggested by part (ii) of Theorem~\ref{thm:calP_nondecreasing}. Specifically, this figure shows how the TRO ambiguity set  $\calP'_{N,\theta}$ enlarges with $\theta$. This, in turn, implies that the TRO model is more conservative when we pick a larger $\theta$. Note that if $\calP'_{N,\theta}$ is constructed using a star-shaped $\calP_N$ with a star center $\M \neq \Probh$, then $\{\calP'_{N,\theta}\mid\theta\in[0,1]\}$ does not satisfy the hierarchical properties; see \ref{apdx:example_star_center} for an example.

\begin{figure}[t!]
    \centering
    \includegraphics[scale=0.85]{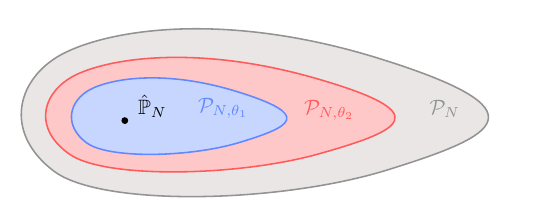}
    \caption{Illustration of the strict hierarchical property of the sequence of TRO ambiguity sets $\{\calP'_{N,\theta}\mid 0<\theta_1<\theta_2<1\}$.}
    \label{fig:trade_off_ambig_set_1}
\end{figure}

Theorem~\ref{thm:calP_nondecreasing} requires the shape parameter $\calP_N$ to be star-shaped with a star center $\Probh_N$.  As is well known, if $\calP_N$ is convex and $\Probh_N\in\calP_N$, then $\calP_N$ is star-shaped with a star center $\Probh_N$ (see Lemma~\ref{lem:star_shaped_convex} in \ref{apdx:lem:star_shaped_convex}). Many celebrated (data-driven) ambiguity sets $\calP_N$ are convex and contain the empirical distribution; see Examples \ref{eg:moment_ambig_set}--\ref{eg:RO_ambig_set} below. Thus, they are star-shaped with a star center $\Probh_N$. On the other hand, a star-shaped set is not necessarily convex; hence, the hierarchical properties hold for general shape parameters.  For example, if $\calP_N$ is the union of a collection of convex ambiguity sets $\{\calP_{N,\gamma}\}_{\gamma\in\Gamma}$ and $\Probh_N\in \calP_{N,\gamma}$ for all $\gamma\in\Gamma$, then $\calP_N$ is star-shaped with a star center $\Probh_N$ but not generally convex. Note that in some settings, one may want to consider the union of multiple (convex) ambiguity sets as the shape parameters. For example, the available data may not always follow a structured pattern and, thus, a known shape parameter. Also, the decision-makers may be unable to articulate their preference for the set of distributions to protect against. In such settings, considering a union of two or more ambiguity sets may better represent or capture the true distribution of the available data. In Propositions~\ref{prop:star_shape_moment} and~\ref{prop:star_shape_distance}, we derive necessary and sufficient conditions for general moment- and distance-based ambiguity sets, respectively, to be star-shaped with a star center $\Probh_N$.

\begin{proposition} \label{prop:star_shape_moment}
Consider the moment-based ambiguity set $\calP_N=\big\{\Prob\in\calP(\Xi)\mid \E_{\Prob}[\Phi_i(\xib)]\in\calK_i,\,i\in\{1,\dots,p\}\big\}$, where $\Phi_i:\Xi\rightarrow\R^{d_i\times d_i}$ is a matrix-valued function and $\calK_i\subseteq\R^{d_i\times d_i}$ is a set of matrices for all $i\in\{1,\dots,p\}$. Then, $\calP_N$ is star-shaped with a star center $\Probh_N$ if and only if $\calK_i$ is star-shaped on $\calS_i:=\{\E_{\Q}[\Phi_i(\xib)]\mid \Q\in\calP_N\}\subseteq\R^{d_i\times d_i}$ with a star center $\E_{\Probh_N}[\Phi_i(\xib)]\in\calS_i$, i.e., $(1-\alpha)\E_{\Probh_N}[\Phi_i(\xib)] + \alpha \Psi\in\calK_i$ for any $\alpha\in[0,1]$ and $\Psi\in\calS_i$, for all $i\in\{1,\dots,p\}$.
\end{proposition}

\begin{example}  \label{eg:moment_ambig_set}
Consider the following moment ambiguity set 
\begin{equation} \label{eqn:moment_ambig_set}
    \calP_N=\Big\{ \Prob\in\calP(\Xi) \,\,\Big|\,\, \E_{\Prob}[\Phi_i(\xib)] = \mub^N_i,\,\, \forall i\in\{1,\dots,p\},\,\, \E_{\Prob}[\Phi_i(\xib)] \preceq \mub^N_i,\,\, \forall i\in\{p+1,\dots,q\}\Big\},  
\end{equation}
where $\Phi_i:\Xi\rightarrow\R^{d_i\times d_i}$ for $i\in\{1,\dots,q\}$ is a symmetric matrix (or a scalar when the dimension $d_i$ is one) with measurable entries and $\E_{\Prob}[\Phi_i(\xib)] \preceq \mub^N_i$ means that $\E_{\Prob}[\Phi_i(\xib)] - \mub^N_i$ is negative semidefinite \citep{Sun_Xu:2016, Xu_et_al:2018}. Here, $\mub_i^N$ for $i\in\{1,\dots,q\}$ is the sample estimate of $\E_{\Prob^\star}[\Phi_i(\xib)]$ based on the given data, i.e., $\mub^N_i=(1/N)\sum_{i=1}^N \Phi(\xibh_i)$. The ambiguity set \eqref{eqn:moment_ambig_set} includes several popular ambiguity sets adopted in the literature \citep{Delage_Ye:2010, So:2011}. It is straightforward to verify that $\calK_i:=\{\Psi_i\in\R^{d_i\times d_i}\mid \Psi_i=\mub^N_i\}$ for $i\in\{1,\dots,p\}$ and $\calK_i:=\{\Psi_i\in\R^{d_i\times d_i}\mid \Psi_i-\mub^N_i \preceq 0\}$ for $i\in\{p+1,\dots,q\}$ are convex with $\E_{\Probh_N}[\Phi_i(\xib)]\in\calK_i$. Thus, $\calK_i$ is star-shaped with a star center $\E_{\Probh_N}[\Phi_i(\xib)]$ for all $i\in\{1,\dots,q\}$.  It follows from  Proposition~\ref{prop:star_shape_moment} that $\calP_N$ is star-shaped with a star center $\Probh_N$.
\end{example}

\begin{proposition} \label{prop:star_shape_distance}
Let $\sfd:\calP(\Xi)\times\calP(\Xi)\rightarrow\R_+$ be a statistical distance satisfying $\sfd(\Prob_1,\Prob_2)=0$ if and only if $\Prob_1=\Prob_2$. Consider the distance-based ambiguity set of the form $\calP_N(\varepsilon)=\{\Prob\in\calP(\Xi)\mid \sfd(\Prob,\Probh_N) \leq\varepsilon\}$. Then, $\calP_N(\varepsilon)$ is star-shaped with a star center $\Probh_N$ for all $\varepsilon\geq 0$ if and only if the map $\Q \mapsto \sfd(\Q, \Probh_N)$ is quasi-convex about $\Probh_N$, i.e.,
$$\sfd\Big((1-\alpha)\Probh_N+\alpha\Q,\Probh_N\Big) \leq \max\Big\{\sfd(\Probh_N,\Probh_N),\,\sfd(\Q,\Probh_N)\Big\} = \sfd(\Q,\Probh_N)$$
for any $\alpha\in[0,1]$ and $\Q\in\calP(\Xi)$.
\end{proposition}

\begin{example}\label{eg:phi_div_ambig_set}  
Consider the following $\phi$-divergence ambiguity set 
\begin{equation} \label{eqn:phi_div_ambig_set}
  \calP_N=\bigg\{ \Prob = \sum_{i=1}^N p_i \delta_{\xibh_i} \,\bigg|\, \pb=(p_1,\dots,p_N)\in\R^N_+,\,\pb^\tp\one = 1,\, \frac{1}{N}\sum_{i=1}^N \phi(Np_i) \leq \varepsilon \bigg\}   
\end{equation}
for some  $\varepsilon>0$ and convex function $\phi:\R_+\rightarrow\R_+$ with $\phi(1)=0$. This contains all the distribution with support $\{\xibh_1,\dots,\xibh_N\}$ such that the $\phi$-divergence defined by $\sfd(\Prob,\Probh_N)=(1/N)\sum_{i=1}^N \phi(Np_i)$ is no greater than the radius $\varepsilon$. It is easy to verify that $\sfd$  is convex in $\Prob$ (identified with $\pb\in\R^N$). Thus, it follows from Proposition~\ref{prop:star_shape_distance} that  ambiguity set \eqref{eqn:phi_div_ambig_set} is star-shaped with a star center $\Probh_N$.
\end{example}

\begin{example}\label{eg:Wass_ambig_set} 
Consider the following $p$-Wasserstein ambiguity set  
\begin{equation} \label{eqn:Wass_ambig_set}
  \calP_N=\big\{ \Prob\in\calP(\Xi)\mid W_p(\Prob,\Probh_N) \leq \varepsilon\big\}     
\end{equation}
for some $p\in[1,\infty)$ and $\varepsilon>0$, where $W_p(\Prob,\Probh_N)$ is the $p$-Wasserstein distance between probability measures $\Prob$ and $\Probh_N$ (see \citealp{Villani:2009, Mohajerin-Esfahani_Kuhn:2018}). Since $\sfd=W_p$ is $p$-convex in $\Prob$ (see Lemma~2.10 of \citealp{Pflug_Pichler:2014}), we have 
$$W_p\Big((1-\alpha)\Probh_N+\alpha\Q,\Probh_N\Big)\leq \bigg\{ (1-\alpha)W^p_p(\Probh_N,\Probh_N)+\alpha W^p_p(\Q,\Probh_N) \bigg\}^{\frac{1}{p}} \leq \alpha^\frac{1}{p} W_p(\Q,\Probh_N) \leq W_p(\Q,\Probh_N) $$
for any $\alpha\in[0,1]$ and $\Q\in\calP_N$. It follows from Proposition~\ref{prop:star_shape_distance} that ambiguity set \eqref{eqn:Wass_ambig_set} is star-shaped with a star center $\Probh_N$.
\end{example}

Finally, in the following example, we show that when the shape parameter $\calP_N$ consists only of Dirac measures (as in the classical robust optimization approach), the sequence of TRO ambiguity sets $\{\calP'_{N,\theta}\mid\theta\in[0,1]\}$ satisfies the hierarchical property under some mild assumptions.

\begin{example} \label{eg:RO_ambig_set} 
The robust optimization (RO) counterpart of problem \eqref{prob:SO}, defined as $\inf_{\xb\in\calX} \sup_{\xib\in\calU} f(\xb,\xib)$ for some uncertainty set $\calU\subseteq\Xi$, can be recast as a DRO model with ambiguity set $\calP_N=\big\{ \delta_{\xib}\mid \xib\in\calU\big\}$. It is obvious that $\calP_N$ contains only Dirac measure and thus, does not contain $\Probh_N$. However, if for any $\xb\in\calX$, there exists $\xib\in\calU$ such that $f(\xb,\xib)\geq f(\xb,\xibh_i)$ for all $i\in\{1,\dots,N\}$, then the RO problem can be written as a DRO problem with ambiguity set  
\begin{equation} \label{eqn:RO_ambig_set}
  \calP_N=\conv\Big( \Probh_N\cup\big\{ \delta_{\xib}\mid \xib\in\calU\big\} \Big)  
\end{equation}
(see \ref{apdx:RO_ambig_set} for a proof). By construction, the ambiguity set in \eqref{eqn:RO_ambig_set} is convex with $\Probh_N\in\calP_N$. Hence, it is also star-shaped with a star center $\Probh_N$, and thus the hierarchical property of $\{\calP'_{N,\theta}\mid\theta\in[0,1]\}$ follows from Theorem~\ref{thm:calP_nondecreasing}.
\end{example}

\subsection{Properties of the TRO Model's Optimal Value and Solutions}\label{sec:conservatism}

In this section, we analyze properties of the optimal value $\upsilonh_N(\theta)$ and the set of optimal solutions $\calXh_N(\theta)$ of the TRO model for a fixed sample through the lens of quantitative stability analysis. Let us first introduce some additional notation to lay the foundation for subsequent discussions. We refer to \ref{apdx:known_QSA_DRO} for background results relevant to our analysis. We define the distance between a point $\yb\in\R^n$ and a set $\calY\subseteq\R^n$ as $d(\yb,\calY)=\inf_{\yb'\in\calY} \norm{\yb-\yb'}$ and the distance between two sets $\calY_1\subseteq\R^n$ and $\calY_2\subseteq\R^n$ as $D(\calY_1,\calY_2)=\sup_{\yb_1\in\calY_1} \inf_{\yb_2\in\calY_2} \norms{\yb_1-\yb_2}$. For two probability distributions $\Prob_1$ and $\Prob_2$, we define the pseudometric $\bbmd(\Prob_1,\Prob_2)$ as 
\begin{equation} \label{eqn:dist_btw_prob_measures}
    \bbmd(\Prob_1,\Prob_2) = \sup_{x\in\calX} \Big| \E_{\Prob_1}[f(\xb,\xib)] - \E_{\Prob_2}[f(\xb,\xib)]\Big| 
\end{equation}
\citep{Liu_Xu:2013, Romisch:2003, Sun_Xu:2016}. Using the pseudometric $\bbmd$, we define the distance between a single probability distribution $\Prob$ and a set $\calP$ of distributions as $\bbmD(\Prob,\calP) = \inf_{\Q\in\calP} \bbmd(\Prob,\Q)$. Finally, we define the Hausdorff distance between two sets of probability distributions, $\calP_1$ and $\calP_2$, as 
\begin{equation} \label{def:Hausdorff_distance}
    \bbmH(\calP_1,\calP_2)=\max\bigg\{ \sup_{\Q\in\calP_1} \bbmD(\Q,\calP_2),\, \sup_{\Q\in\calP_2} \bbmD(\Q,\calP_1) \bigg\}. 
\end{equation}

We make the following technical assumptions to ensure that the TRO model is well-defined.

\begin{assumption} \label{assumption:loss_and_ambig_set_boundedness}
The feasible set $\calX$ and the function $f$ satisfy the following conditions:
\begin{enumerate}[itemsep=0mm,topsep=1mm]
    \item [(a)] the feasible set $\calX$ is compact;
    \item [(b)] given an ambiguity set $\calP_N$, (i) $f(\cdot,\xib)$ is Lipschitz continuous on $\calX$ for any fixed $\xib\in\Xi$ with Lipschitz modulus bounded by $\kappa(\xib)$, where $\sup_{\Prob\in\calP_N} \E_{\Prob}[\kappa(\xib)]<\infty$; and (ii)  there exists $\xb_0\in\calX$ such that $\max\big\{ \E_{\Probh_N} \big|f(\xb_0,\xib)\big|,\, \sup_{\Prob\in\calP_N} \E_{\Prob}\big| f(\xb_0,\xib) \big| \big\}<\infty$.
\end{enumerate}

\end{assumption}

Assumption~\ref{assumption:loss_and_ambig_set_boundedness}(a) is a standard assumption in the stochastic optimization literature \citep{Duchi_et_al:2021, Shapiro_et_al:2014, Van-Parys_et_al:2021, Zhang_et_al:2016}. Assumption \ref{assumption:loss_and_ambig_set_boundedness}(b) ensures the smoothness of the objective function of \eqref{model:trade-off_model} and that the objective is finite at some point $\xb_0\in\calX$ \citep{Gao:2022, Pichler_Xu:2022, Sun_Xu:2016}. Essentially, Assumption \ref{assumption:loss_and_ambig_set_boundedness} ensures that
\begin{equation} \label{eqn:def_of_CN}
    C_N:=\max\Bigg\{\sup_{\xb\in\calX} \Big| \E_{\Probh_N}[f(\xb,\xib)] \Big|,\,  \sup_{\xb\in\calX}\sup_{\Prob\in\calP_N}\Big|\E_{\Prob}[f(\xb,\xib)]\Big|\Bigg\}<\infty,
\end{equation}
implying that our trade-off model \eqref{model:trade-off_model} is well-defined, i.e., has a finite optimal value and optimal solution for any $\theta\in[0,1]$. These assumptions hold valid in many applications (e.g., scheduling, inventory control, facility location, etc.).

First, in Theorem \ref{thm:sensitivity_in_theta}, we quantify the impact of a perturbation of the trade-off parameter $\theta$ on the optimal value and the set of optimal solutions to our TRO model.  Specifically, we show the Lipschitz continuity of the optimal value function $\upsilonh_N(\theta)$ and the stability of the set of optimal solutions $\calXh_N(\theta)$ to our TRO model.

\begin{theorem} \label{thm:sensitivity_in_theta}
Under Assumption \ref{assumption:loss_and_ambig_set_boundedness}, the following assertions hold for any $\{\theta_1,\theta_2\}\subset[0,1]$. 
\begin{enumerate}[itemsep=0mm,topsep=1mm]
    \item [(i)] $|\upsilonh_N(\theta_1)-\upsilonh_N(\theta_2)| \leq 2C_N |\theta_1-\theta_2|$.
    \item [(ii)] If, in addition, $\sup_{\Prob\in\calP'_{N,\theta_1}} \E_{\Prob}[f(\xb,\xib)]$ satisfies the second-order growth condition at  $\calXh_N(\theta_1)$, i.e., there exists $\tau>0$ such that
    \begin{equation} \label{eqn:second_order_growth}
        \sup_{\Prob\in\calP'_{N,\theta_1}} \E_{\Prob}[f(\xb,\xib)] \geq \upsilonh_N(\theta_1)+\tau \Big[d\Big(\xb,\calXh_N(\theta_1)\Big)\Big]^2
    \end{equation}
    for all $\xb\in\calX$, then $D\big(\calXh_N(\theta_2),\calXh_N(\theta_1)\big)\leq \sqrt{6C_N\tau^{-1}\, |\theta_1-\theta_2|}.$
\end{enumerate}
\end{theorem}

Theorem \ref{thm:sensitivity_in_theta} establishes mechanisms to quantify the difference in  $\upsilonh_N(\theta)$ and  $\calXh_N(\theta)$ (and hence conservatism) incurred by perturbation in $\theta$. In particular, it shows that both the optimal value and the set of optimal solutions change gradually with $\theta\in[0,1]$; $\upsilonh_N(\theta)$ is Lipschitz continuous in $\theta$ and $\calXh_N(\theta)$ is H\"{o}lder continuous with H\"{o}lder exponent $1/2$ under distance $D$. Moreover, if $\theta$ is sufficiently close to zero (resp. one), the optimal value and the set of optimal solutions to our TRO model are close to the SAA (resp. DRO) counterparts. Therefore, in practice, to generate a spectrum of optimal solutions with various extents of conservatism, it suffices to consider multiple disjoint values of $\theta$ given the continuity of $\upsilonh_N(\theta)$ and $\calXh_N(\theta)$. We remark that the second-order growth condition assumption in \eqref{eqn:second_order_growth} is standard in stochastic optimization literature \citep{Liu_et_al:2019, Pichler_Xu:2022, Shapiro:1994}. For example, as discussed in \cite{Pichler_Xu:2022},  the second-order growth condition \eqref{eqn:second_order_growth} holds when the function $f(\xb,\xib)$ is $\mu(\xib)$-strongly convex in $\xb$ for each $\xib\in\Xi$ and $\inf_{\Prob\in\calP}\E_{\Prob}[\mu(\xib)]>0$.

Note that directly combining SAA and DRO optimal solutions, e.g., via a convex combination after solving each separately, may not yield a feasible solution to the TRO problem \eqref{model:trade-off_model}.  This is particularly true in many practical applications where $\calX$ is not convex, for example, in problems that involve integer variables such as facility location and scheduling problems. Thus, one needs to solve the TRO model to obtain decisions with different levels of conservatism. In Theorem \ref{thm:conservatism}, we leverage the results in Theorem \ref{thm:sensitivity_in_theta} to quantify the difference between our TRO model's optimal value (set of optimal solutions) and the convex combination of optimal values (sets of optimal solutions) to the SAA ($\theta=0$) and DRO ($\theta=1$) problems resulting from solving each separately.

\begin{theorem} \label{thm:conservatism}
Under Assumption \ref{assumption:loss_and_ambig_set_boundedness}, the following assertion holds.
\begin{enumerate}[itemsep=0mm,topsep=1mm]
    \item [(i)] $\upsilonh_N:[0,1]\rightarrow\R$ is concave with 
    \begin{equation} \label{eqn:expansion_in_theta}
        \upsilonh_N(\theta) = (1-\theta)\cdot\upsilonh_N(0)+\theta\cdot\upsilonh_N(1)+\rh_N(\theta), 
    \end{equation}
    where $\rh_N(\theta)\in[0,4C_N\theta(1-\theta)]$.
    
    \item [(ii)] If, in addition, the second order growth rate condition \eqref{eqn:second_order_growth} holds, then 
    $$D\Big(\calXh_N(\theta),(1-\theta)\calXh_N(0)+\theta\calXh_N(1)\Big) \leq \sqrt{6C_N\tau^{-1}\theta(1-\theta)}\Big( \sqrt{\theta} + \sqrt{1-\theta}\Big), $$
    %
    where the set $(1-\theta)\calXh_N(0)+\theta\calXh_N(1):=\{(1-\theta)\xb_1+\theta\xb_2\mid \xb_1\in\calXh_N(0),\,\xb_2\in\calXh_N(1)\}$ contains convex combinations of the SAA and DRO optimal solutions.
\end{enumerate}
\end{theorem}

Part (i) of Theorem \ref{thm:conservatism} establishes that the optimal value $\upsilonh_N(\theta)$ to our TRO model is not less than the convex combination $(1-\theta)\upsilonh_N(0)+\theta\upsilonh_N(1)$ of the SAA and DRO optimal values. In addition, if $\Probh_N\in\calP_N$, then $\upsilonh_N(\theta)$ is the minimum of the non-decreasing functions $\big\{g(\theta;\xb):=\E_{\Probh_N}[f(\xb,\xib)]+\theta\big[\sup_{\Prob\in\calP_N}\E_{\Prob}[f(\xb,\xib)]-\E_{\Probh_N}[f(\xb,\xib)]\big]\,\big|\,\xb\in\calX\big\}$ and thus, is non-decreasing in $\theta$ as illustrated in Figure~\ref{fig:conservatism_plot}. Therefore, by solving our TRO model with different $\theta\in[0,1]$, we can obtain a spectrum of decisions that spans $[\upsilonh_N(0),\upsilonh_N(1)]$, representing decisions with different levels of conservatism. Moreover, since $\upsilonh_N$ is concave, the rate of change in conservatism  (i.e., slope of $\upsilonh_N$) is larger for smaller values of $\theta$. In other words, the increase in conservatism $\upsilonh_N(\theta+\Delta)-\upsilonh_N(\theta)$ is more significant when $\theta$ is small, where $\Delta>0$ is the perturbation on $\theta$. To obtain a decision with a specific target level of conservatism, say $\overline{\upsilon}=(1-\lambda)\upsilonh_N(0) + \lambda\upsilonh_N(1)$ for some $\lambda\in(0,1)$, one should pick a value of $\theta$ less than $\lambda$ (see Figure \ref{fig:conservatism_plot}). Part (ii) of Theorem \ref{thm:conservatism} indicates that the set of optimal solutions $\calXh_N(\theta)$ to our TRO model can be approximated by $\overline{\calX}_N(\theta):=(1-\theta)\calXh_N(0)+\theta\calXh_N(1)$ only when $\theta$ is close to zero or one; however, the difference could be huge for intermediate values of $\theta \in (0, 1)$. 
\begin{figure}
    \centering
    \includegraphics[scale=0.9]{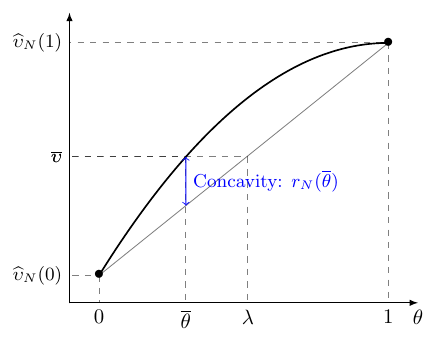}
    \caption{Illustration of $\upsilonh_N(\theta)$ when $\calP_N$ is convex and $\Probh_N\in\calP_N$. For a given target level of conservatism $\overline{\upsilon}=(1-\lambda)\upsilonh_N(0) + \lambda\upsilonh_N(1)$ with $\lambda\in(0,1)$, one should pick $\overline{\theta}<\lambda$ in the TRO model.}   \label{fig:conservatism_plot}
\end{figure}

\section{Finite-Sample Properties} \label{sec:finite_sample_property}

In this section, we investigate the finite-sample properties of the TRO model. First, in Section~\ref{subsec:bias_analysis}, we analyze the bias of the optimal value of our TRO model $\upsilonh_N(\theta)$ as an estimator of the true optimal value $\upsilon^\star$. Then, in Section~\ref{subsec:generalization}, we derive the generalization bound for our TRO model.

\subsection{Bias Analysis} \label{subsec:bias_analysis}

The optimal value of our TRO model $\widehat{\upsilon}_N(\theta)$ represents an estimator of the optimal value $ \upsilon^\star$ of problem \eqref{prob:SO}. In this section, we analyze the bias of $\upsilonh_N(\theta)$, i.e., $\E_{\Prob^N}[\upsilonh_N(\theta)]-\upsilon^\star$, where $\Prob^N$ is the joint distribution of $\{\xibh_1,\dots,\xibh_N\}$. Recall that when $\theta=0$, our TRO estimator reduces to the SAA estimator $\upsilonh_N(0)$, which suffers from downward bias \citep{Mak_et_al:1999, Norkin_et_al:1998, Shapiro_et_al:2014}, i.e.,
\begin{equation} \label{eqn:SAA_downward_bias}
    \E_{\Prob^N}\big[\upsilonh_N(0)\big]=\E_{\Prob^N}\Bigg[\min_{\xb\in\calX} \frac{1}{N} \sum_{i=1}^N f(\xb,\xibh_i)\Bigg] \leq \min_{\xb\in\calX} \E_{\Prob^N}\Bigg[\frac{1}{N} \sum_{i=1}^N f(\xb,\xibh_i)\Bigg] =\min_{\xb\in\calX} \E_{\Prob^\star}[f(\xb,\xib)] = \upsilon^\star.
\end{equation}
Although the bias $\upsilon^\star-\E_{\Prob^N}\big[\upsilonh_N(0)\big]$ decreases monotonically with $N$ and goes to zero as $N\rightarrow\infty$ \citep{Homem-de-Mello_Bayraksan:2014}, it may diminish slowly and could remain significant even for large $N$ \citep{Dentcheva_Lin:2022}.  Thus, constructing an accurate estimator of $\upsilon^\star$, and ideally, an unbiased estimator of $\upsilon^\star$ is of practical interest. In this section, we show that our TRO model could produce estimators with a smaller bias than the SAA estimator. Moreover, under mild assumptions, we show that there exists $\theta^\text{u}_N\in[0,1]$ such that $\upsilonh_N(\theta^\text{u}_N)$ is an unbiased estimator.  We also derive the asymptotic convergence rate of $\theta^\text{u}_N$ as $N\rightarrow\infty$. Additionally, we derive an upper bound on the difference between the TRO and SAA estimators’ standard errors, showing that the TRO estimator may not have a significantly larger variance than the SAA estimator.

In what follows, we say Assumption~\ref{assumption:loss_and_ambig_set_boundedness} holds if Assumption~\ref{assumption:loss_and_ambig_set_boundedness}(a) holds and Assumption~\ref{assumption:loss_and_ambig_set_boundedness}(b) holds for almost every ambiguity set $\calP_N$.

First, in Proposition~\ref{prop:bias_non_neg_UB}, we derive an upper bound on the bias of the TRO estimator $\upsilonh_N(\theta)$. 

\begin{proposition} \label{prop:bias_non_neg_UB}
Under Assumption \ref{assumption:loss_and_ambig_set_boundedness}, we have
\begin{equation} \label{eqn:bias_non_neg_UB}
    \E_{\Prob^N}[\upsilonh_N(\theta)]-\upsilon^\star \leq \theta \Big\{\E_{\Prob^N}\big[\upsilonh_N(1)\big]-\E_{\Prob^N}\big[\upsilonh_N(0)\big]\Big\} + R_N(\theta),
\end{equation}
where $R_N(\theta)\in[0,4\Cbar_N\theta(1-\theta)]$ and $\Cbar_N:=\E_{\Prob^\star}(C_N)<\infty$, where $C_N$ is as defined in \eqref{eqn:def_of_CN}. If $\Probh_N\in\calP_N$ almost surely, then the upper bound in \eqref{eqn:bias_non_neg_UB} is non-negative.
\end{proposition}

Proposition~\ref{prop:bias_non_neg_UB} shows that the bias of the TRO estimator $\upsilonh_N(\theta)$ is upper bounded by two terms in \eqref{eqn:bias_non_neg_UB}. The first term $\theta \{\E_{\Prob^N}\big[\upsilonh_N(1)\big]-\E_{\Prob^N}\big[\upsilonh_N(0)\big]\}$ is due to the DRO objective component in the TRO model, which is increasing with $\theta$ and is zero when $\theta=0$. The second term $R_N(\theta)$ is a consequence of the concavity of $\upsilonh_N$ (see Theorem \ref{thm:conservatism}). Thanks to the concavity, $R(\theta)$ is non-negative when $\theta\in(0,1)$. Thus, Proposition \ref{prop:bias_non_neg_UB} suggests that the bias of $\upsilonh_N(\theta)$ may not be a downward bias as that of the SAA estimator. Indeed, in the next theorem, we show that $\upsilonh_N(\theta)$ is an unbiased estimator of $\upsilon^\star$ for some choice of $\theta$ under a mild assumption.

\begin{theorem} \label{thm:debias_theta}
Suppose that Assumption \ref{assumption:loss_and_ambig_set_boundedness} holds and $\E_{\Prob^N}[\upsilonh_N(1)] \geq \upsilon^\star$. Then, there exists $\theta^\textup{u}_N\in[0,1]$ such that  $\E_{\Prob^N}[\upsilonh_N(\theta^\textup{u}_N)] = \upsilon^\star$, i.e., $\upsilonh_N(\theta^\textup{u}_N)$ is an unbiased estimator of $\upsilon^\star$.
\end{theorem}

\begin{corollary} \label{cor:debias_theta}
Under the same assumptions as in Theorem \ref{thm:debias_theta}, there exists $\theta_N^\textup{u}\in[0,1]$ such that $|\E_{\Prob^N}[\upsilonh_N(\theta)]-\upsilon^\star| \leq |\E_{\Prob^N}[\upsilonh_N(0)]-\upsilon^\star|$ for all $\theta \in[0,\theta_N^\textup{u}]$, i.e., the bias of $\upsilonh_N(\theta)$ is not greater than the bias of the SAA estimator $\upsilonh_N(0)$.
\end{corollary}

We illustrate the results of Theorem \ref{thm:debias_theta} and Corollary \ref{cor:debias_theta} in Figure \ref{fig:bias_reduction}. First, in line with Theorem~\ref{thm:debias_theta}, this figure shows that when  $\upsilon^\star$ is between $\E_{\Prob^N}[\upsilonh_N(0)]$ and $\E_{\Prob^N}[\upsilonh_N(1)]$, one can find $\theta_N^\text{u} \in [0,1]$ for which $\E_{\Prob^N}[\upsilonh_N(\theta_N^\text{u})]= \upsilon^\star $. Here, $\theta_N^\text{u}$ is a constant that depends only on the sample size $N$. Second, any choice of $\theta\in[0,\theta_N^\text{u}]$ leads to an estimator $\upsilonh_N(\theta)$ that has a smaller bias than the SAA estimator as indicated by Corollary \ref{cor:debias_theta}. Finally,  this figure shows that $\E_{\Prob^N}[\upsilonh_N(\theta)]$  is the sum of three terms (see Proposition \ref{prop:bias_non_neg_UB}) (a) the expected value of the SAA estimator $\E_{\Prob^N}[\upsilonh_N(0)]$; (b) the DRO effect $\theta\big\{\E_{\Prob^N}\big[\upsilonh_N(1)\big]-\E_{\Prob^N}\big[\upsilonh_N(0)\big]\big\}$; and (c)  the concavity effect $R_N(\theta)$. Specifically, the DRO effect adjusts $\E_{\Prob^N}[\upsilonh_N(0)]$ up to the linear combination of the two endpoints of $\E_{\Prob^N}[\upsilonh_N(\theta)]$, while the concavity effect accounts for the remaining difference. In practice, finding $\theta_N^\text{u}$ that leads to the unbiased estimator $\upsilonh_N(\theta_N^\text{u})$ as shown in Theorem \ref{thm:debias_theta} is difficult. However, Corollary \ref{cor:debias_theta} suggests that the bias of $\upsilonh_N(\theta_N^\text{u})$ is smaller than that of the SAA estimator for sufficiently small $\theta$. Therefore, we could always choose a small $\theta$ in our TRO model to produce an estimator with a smaller bias.

\begin{figure}
    \centering
    \includegraphics[scale=0.9]{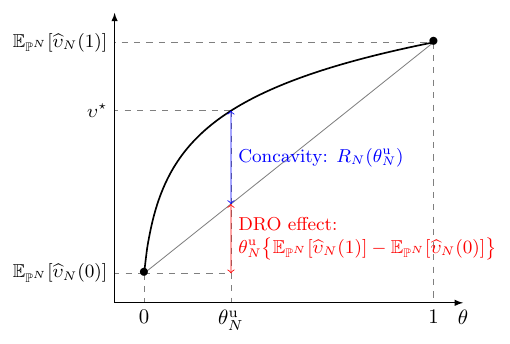}
    \caption{Illustration of the SAA bias reduction effect. The solid curve represents the concave function $\E_{\Prob^N}[\upsilonh_N(\theta)]$.}
    \label{fig:bias_reduction}
\end{figure}

\begin{remark} \label{rem:assumption_in_debias_theta_thm}

In Theorem \ref{thm:debias_theta}, we assume that $\E_{\Prob^N}[\upsilonh_N(1)] \geq \upsilon^\star$. This assumption will likely hold when the ambiguity set $\calP_N$ is rich enough such that it contains a wide range of distributions. For example, if the (data-driven) ambiguity set $\calP_{N,\alpha}$ satisfies $\Prob^N(\Prob^\star\in\calP_{N,\alpha})\geq 1-\alpha$ for some $\alpha\in(0,1)$ (see, e.g., \citealp{Delage_Ye:2010, Mohajerin-Esfahani_Kuhn:2018}), the assumption would probably hold by choosing $\calP_{N,\alpha}$ with $\alpha$ close to zero; see \ref{apdx:add_discuss:remark} for a detailed discussion.
\end{remark}

Note that under Assumption~\ref{assumption:loss_and_ambig_set_boundedness}, if $\E_{\Prob^N}[\upsilonh_N(1)]>\upsilon^\star$, then $\theta_N^\text{u}$ is unique by concavity of $\E_{\Prob^N}[\upsilonh_N(\theta)]$. Therefore, without loss of generality, we assume that $\theta_N^\text{u}$ is unique in the following discussions. In the case when $\Theta_N:=\big\{\theta\in[0,1]\mid \E_{\Prob^N}[\upsilonh_N(\theta)]=\upsilon^\star\big\}$ is not a singleton, we consider $\theta_N^\text{u}$ as one of the elements in $\Theta_N$. Next, we analyze the asymptotic behavior of $\theta_N^\text{u}$. In Theorem~\ref{thm:rate_of_theta_LIL}, we prove the convergence of $\theta_N^\text{u}$ as $N\rightarrow\infty$ and derive the rate of convergence under mild assumptions.

\begin{theorem} \label{thm:rate_of_theta_LIL}
Suppose that Assumption \ref{assumption:loss_and_ambig_set_boundedness} and the following hold.
\begin{enumerate}[itemsep=0mm,topsep=1mm]
    \item [(a)] The samples $\{\xibh_i\}_{i=1}^N$ are i.i.d. following the true distribution $\Prob^\star$.
    \item [(b)] There exists $\xb_0\in\calX$ such that $\E_{\Prob^\star}[f(\xb_0,\xib)^2]<\infty$.
    \item [(c)] There exists a measurable function $\kappa:\xib\rightarrow\R_+$ such that $\E_{\Prob^\star}[\kappa(\xib)^2]<\infty$ and 
    $$\Big| f(\xb_1,\xib)- f(\xb_2,\xib) \Big| \leq\kappa(\xib)\norms{\xb_1-\xb_2} $$
    for any $\{\xb_1,\xb_2\}\subseteq\calX$ and $\Prob^\star$-almost everywhere $\xib$.
\end{enumerate}
If $\E_{\Prob^N}[\upsilonh_N(1)] \geq \upsilon^\star$ for all $N\in\N$ and $\inf_{N\in\N} \big\{\E_{\Prob^N}[\upsilonh_N(1)]-\E_{\Prob^N}[\upsilonh_N(0)]\big\}>0$, then $\theta_N^\text{u}=o(\sqrt{\log\log N}/\sqrt{N})$.
\end{theorem}

Assumptions (a)--(c) are widely adopted in the stochastic optimization literature for deriving asymptotic distributions (see, e.g., \citealp{Guigues_et_al:2018, Lam:2019, Shapiro_et_al:2014}). The condition $\inf_{N\in\N}\big\{\E_{\Prob^N}[\upsilonh_N(1)]-\E_{\Prob^N}[\upsilonh_N(0)]\}>0$ typically holds when $\Probh_N\in\calP_N$ and $\calP_N$ contains a wide range of distributions such that $\upsilonh_N(1)\geq \upsilonh_N(0)+\Phi_N$ for some $\Phi_N\geq 0$ with $\inf_{N\in\N}\E_{\Prob^N}(\Phi_N)>0$. For example, suppose we choose $\calP_N$ as the $1$-Wasserstein ambiguity set with radius $\varepsilon>0$ (see Example~\ref{eg:Wass_ambig_set}), $\Xi=\R^\ell$, and $f(\xb,\cdot)$ is convex and piecewise linear for any $\xb\in\calX$, then $\sup_{\Prob\in\calP_N}\E_{\Prob}[f(\xb,\xib)]=\E_{\Probh_N}[f(\xb,\xib)]+ \kappa\varepsilon$ for any $\xb\in\calX$ with $\kappa\geq 0$ \citep{Mohajerin-Esfahani_Kuhn:2018}. Therefore, the condition $\inf_{N\in\N}\big\{\E_{\Prob^N}[\upsilonh_N(1)]-\E_{\Prob^N}[\upsilonh_N(0)]\}>0$ holds when $\kappa>0$. The assumption that $\E_{\Prob^N}[\upsilonh_N(1)]\geq\upsilon^\star$ for all $N\in\N$ is a technical assumption to ensure the existence of $\theta^\text{u}_N$ as shown in Theorem~\ref{thm:debias_theta} (see Remark~\ref{rem:assumption_in_debias_theta_thm}). 

Theorem~\ref{thm:rate_of_theta_LIL} shows that $\theta_N^\text{u}$ converges to zero at a rate of $o(\sqrt{\log\log N}/\sqrt{N})$. This is not surprising since when the size parameter $\theta$ converges to zero, the objective function of our TRO model reduces to the SAA objective, and the bias of the SAA estimator $|\E_{\Prob^N}[\upsilonh(0)]-\upsilon^\star|$ is of order $o(\sqrt{\log\log N}/\sqrt{N})$ \citep{Banholzer_et_al:2022}. Thus, to construct an unbiased estimator $\upsilonh_N(\theta_N)$ of $\upsilon^\star$, one could use the convergence order as a criterion for choosing the size parameter $\theta_N$. Next, under a stronger assumption widely adopted in the stochastic optimization literature (see, e.g., \citealp{Homem-de-Mello_Bayraksan:2014, Kim_et_al:2015, Shapiro_et_al:2014}), we can derive an improved convergence rate for $\theta_N^\text{u}$ in Theorem \ref{thm:rate_of_theta_AN}.

\begin{theorem} \label{thm:rate_of_theta_AN}
In addition to the assumptions in Theorem \ref{thm:rate_of_theta_LIL}, suppose that the sequence $\{X_N:=\sqrt{N}(\upsilonh_N(0)-\upsilon^\star)\}_{N\in\N}$ is asymptotic uniformly integrable, i.e., $\lim_{M\rightarrow\infty} \limsup_{N\rightarrow\infty} \E\big(|X_N|\one (|X_N|>M)\big) = 0$. If $\E_{\Prob^N}[\upsilonh_N(1)] \geq \upsilon^\star$ for all $N\in\N$ and $\inf_{N\in\N} \big\{\E_{\Prob^N}[\upsilonh_N(1)]-\E_{\Prob^N}[\upsilonh_N(0)]\big\}>0$, then
\begin{enumerate}
    \item $\theta_N^\text{u}=o(1/\sqrt{N})$ when the optimal solution to \eqref{prob:SO} is unique;
    \item $\theta_N^\text{u}=O(1/\sqrt{N})$ when there are multiple optimal solutions to \eqref{prob:SO}.
\end{enumerate}
\end{theorem}

Note that while our TRO estimator has a smaller bias than the classical SAA estimator, its variance may be smaller or larger. Quantifying the variance (or the standard error) of the TRO estimator $\upsilonh_N(\theta)$ (i.e., the optimal value of the TRO model) is challenging, in part because $\upsilonh_N(\theta)$ does not admit a closed-form expression. According to recent surveys on DRO \citep{Kuhn_et_al:2024, Rahimian_Mehrotra:2022}, analyzing the variance of the optimal value of data-driven DRO problems is an open question. In the following theorem, we establish an upper bound on the difference between the TRO and SAA estimators' standard errors in the following theorem.

\begin{theorem} \label{thm:TRO_estimator_variance}
In addition to Assumption~\ref{assumption:loss_and_ambig_set_boundedness}, suppose that (a) $\Xi$ is compact, (b) $f(\xb,\cdot)$ is uniformly Lipschitz in $\xib$, i.e., there exists $L>0$ such that $|f(\xb,\xib_1)-f(\xb,\xib_2)|\leq L\norms{\xib_1-\xib_2}$ for all $\xb\in\calX$, (c) there exists $\xib_0$ such that $\sup_{\xb\in\calX}|f(\xb,\xib_0)|\leq M<\infty$, and (d) $\Probh_N\in\calP_N$ almost surely. Then, we have
$$\sqrt{\Var_{\Prob^N}\big(\upsilonh_N(\theta)\big)}\leq\sqrt{\Var_{\Prob^N}\big(\upsilonh_N(0)\big)}+2\theta(3-2\theta)\big(L\cdot\diam(\Xi)+M\big),$$
where $\diam(\Xi)$ is the diameter of $\Xi$.
\end{theorem}

Theorem~\ref{thm:TRO_estimator_variance} suggests that our TRO estimator may not have a significantly larger variance than the SAA estimator, especially when $\theta$ is small. Our numerical results in Section~\ref{sec:numerics} show that the variance of the TRO estimator will not increase substantially and could be even smaller than the SAA estimator.

We close this section by highlighting that some recent studies in the stochastic optimization literature have also proposed approaches to construct estimators of the true optimal value $\upsilon^\star$ with smaller bias than the SAA estimator, see, e.g.,  \cite{Blanchet_et_al:2019a, Dentcheva_Lin:2022, Norton_Royset:2023}. Other studies focused on obtaining confidence intervals for  $\upsilon^\star$ (see, e.g., \citealp{Bertsimas_et_al:2018, Duchi_et_al:2021, Lam:2022}). Our proposed TRO model and results in this section add to the first stream of literature on alternative approaches to constructing estimators with smaller biases.

\subsection{Generalization Bound} \label{subsec:generalization}

In this section, we analyze the generalization error (also known as the out-of-sample disappointment) of the TRO estimator. This error quantifies how much the model generalizes in the (unseen) out-of-sample setting. Specifically, let $\xb_N(\theta)\in\calXh_N(\theta)$ be an optimal solution to the TRO model. The generalization error (GE) of the TRO model is given by $\E_{\Prob^\star}[f(\xb_N(\theta),\xib)]-\sup_{\Prob\in\calP'_{N,\theta}}\E_{\Prob}[f(\xb_N(\theta),\xib)]$, i.e., the difference between the actual expected cost associated with implementing $\xb_N(\theta)$ under the true unknown distribution $\Prob^\star$ and the cost estimated by the TRO model. A  positive GE indicates that the TRO model is overly optimistic. As in the literature, we quantify the GE of our TRO model via the following probability:
\begin{equation} \label{eqn:generalization_error}
    \Prob^N\Bigg(\sup_{\xb\in\calX}\bigg\{ \E_{\Prob^\star}[f(\xb,\xib)]- \sup_{\Prob\in\calP'_{N,\theta}}\E_{\Prob}[f(\xb,\xib)] \bigg\}>\delta\Bigg)
\end{equation}
for some $\delta>0$ \citep{Duchi_Namkoong:2021, Gao:2022, Liu_et_al:2023b}. Note that \eqref{eqn:generalization_error} provides an upper bound on the probability that the GE is greater than $\delta$. Thus, our goal is to derive an upper bound on \eqref{eqn:generalization_error} that decays rapidly (e.g., exponentially) when $N$ increases. To this end,  in Theorem~\ref{thm:generalization_error}, we first derive an upper bound on \eqref{eqn:generalization_error}.

\begin{theorem} \label{thm:generalization_error}
Suppose that the generalization error of the SAA model satisfies
\begin{equation} \label{eqn:generalization_error_bound_SAA}
    \Prob^N\Bigg(\sup_{\xb\in\calX}\bigg\{ \E_{\Prob^\star}[f(\xb,\xib)]- \E_{\Probh_N}[f(\xb,\xib)] \bigg\}>\delta\Bigg) \leq \alpha_{N,1}(\delta)
\end{equation}
and the generalization error of the DRO model with ambiguity set $\calP_N$ satisfies
\begin{equation} \label{eqn:generalization_error_bound_DRO}
    \Prob^N\Bigg(\sup_{\xb\in\calX}\bigg\{ \E_{\Prob^\star}[f(\xb,\xib)]- \sup_{\Prob\in\calP_N}\E_{\Prob}[f(\xb,\xib)] \bigg\}>\delta\Bigg) \leq \alpha_{N,2}(\delta)
\end{equation}
for some constants $\alpha_{N,1}(\delta)>0$ and $\alpha_{N,2}(\delta)>0$ (which may depend on the complexity of the function class $\calH=\{f(\xb,\cdot)\mid\xb\in\calX\}$). Then, for any $\theta\in(0,1)$, we have
\begin{equation} \label{eqn:generalization_error_bound}
    \Prob^N\Bigg(\sup_{\xb\in\calX}\bigg\{ \E_{\Prob^\star}[f(\xb,\xib)]- \sup_{\Prob\in\calP'_{N,\theta}}\E_{\Prob}[f(\xb,\xib)] \bigg\}>\delta\Bigg)\leq \inf_{\gamma\in[0,\delta]}\bigg\{\alpha_{N,1}\bigg(\frac{\gamma}{1-\theta}\bigg)+\alpha_{N,2}\bigg(\frac{\delta-\gamma}{\theta}\bigg)\bigg\}.
\end{equation}
\end{theorem}

Theorem~\ref{thm:generalization_error} has the following important implications.  First, note that the generalization bound~\eqref{eqn:generalization_error_bound_SAA} for the SAA model is well-known in the literature. Specifically, under some mild assumptions, there exist constants $D_1>0$ and $\beta_1(\delta)>0$ such that \eqref{eqn:generalization_error_bound_SAA} holds with $\alpha_{N,1}(\delta)=D_1\exp\{-N\beta_1(\delta)\}$ (see, e.g., Theorem~7.73 of \citealp{Shapiro_et_al:2014}). Similarly, the generalization bound~\eqref{eqn:generalization_error_bound_DRO} for the DRO model has been widely studied in the literature under various choices of the shape parameter $\calP_N$. For example, for some moment-based ambiguity sets \citep{Delage_Ye:2010} and distance-based ambiguity sets \citep{Duchi_Namkoong:2021, Liu_et_al:2023b, Mohajerin-Esfahani_Kuhn:2018, Van-Parys_et_al:2021}, there exist constants $D_2>0$ and $\beta_2(\delta)>0$ such that \eqref{eqn:generalization_error_bound_DRO} holds with $\alpha_{N,2}(\delta)=D_2\exp\{-N\beta_2(\delta)\}$. Hence, suppose we construct the TRO ambiguity set using one of these shape parameters. We can then apply Theorem~\ref{thm:generalization_error} to obtain the following exponentially decaying bound on the probability
\begin{align*}
&\quad\,\,\Prob^N\Bigg(\sup_{\xb\in\calX}\bigg\{ \E_{\Prob^\star}[f(\xb,\xib)]- \sup_{\Prob\in\calP'_{N,\theta}}\E_{\Prob}[f(\xb,\xib)] \bigg\}>\delta\Bigg) \\
    &\leq \inf_{\gamma\in[0,\delta]} \Bigg\{D_1\exp\bigg\{-N\beta_1\bigg(\frac{\gamma}{1-\theta}\bigg)\bigg\}+D_2\exp\bigg\{-N\beta_2\bigg(\frac{\delta-\gamma}{\theta}\bigg)\bigg\}\Bigg\}\\
    &\leq D_1\exp\{-N\beta_1(\delta)\}+ D_2\exp\{-N\beta_2(\delta)\} \\ 
    &\leq D \exp\{-N\beta(\delta)\},
\end{align*}
where the second inequality follows from choosing $\gamma=\delta(1-\theta)$, and the third inequality follows from defining $D=D_1+D_2$ and $\beta(\delta)=\min\{\beta_1(\delta),\beta_2(\delta)\}$. This shows that when the TRO ambiguity set is constructed using some specific choice of the shape parameter, the generalization error of our TRO model has an exponentially decaying tail, which is a desirable property for data-driven stochastic optimization models (see, e.g., \citealp{Liu_et_al:2023b, Van-Parys_et_al:2021}).

Second, the GE bound in \eqref{eqn:generalization_error_bound} involves an optimization over $\gamma$ that effectively balances between the two terms $\alpha_{N,1} (\cdot)$ (from the SAA component) and  $\alpha_{N,2} (\cdot)$ (from the DRO component), where each of these terms depends on the choice of $\theta$. By adjusting $\theta$ and optimizing over $\gamma$, the TRO model might have a tighter GE bound than the bounds of the SAA or DRO models. The exact improvement depends on the behavior of the functions $\alpha_{N,1} (\cdot)$ and  $\alpha_{N,2}(\cdot)$, as well as the choice of $\theta$.   Since the probabilities \eqref{eqn:generalization_error_bound_SAA} and \eqref{eqn:generalization_error_bound_DRO} are non-increasing in $\delta$, both $\alpha_{N,1}(\delta)$ and $\alpha_{N,2}(\delta)$ are non-increasing functions of $\delta$. As $\theta\rightarrow 1$, if $\alpha_{N,1}(\delta)\rightarrow 0$ as $\delta\rightarrow\infty$ (see, e.g., \citealp{Liu_et_al:2023, Wainwright:2019}),  then the upper bound in~\eqref{eqn:generalization_error_bound} reduces to $\alpha_{N,2}(\delta)$, i.e., the GE of the DRO model. In contrast, when $\theta\rightarrow0$, (2) recovers the GE of the SAA model. Since  $\alpha_{N,1}\big(\gamma/(1-\theta)\big)$ in \eqref{eqn:generalization_error_bound} is non-increasing in $\theta$ and  $\alpha_{N,2}\big((\delta-\gamma)/\theta\big)$ is non-decreasing in $\theta$, there exists a balance between the two terms that can be optimized for a carefully chosen $\theta$. Optimization over $\gamma$ (controls the relative contribution of the two terms to the bound) allows for finding a point where this balance is best achieved, resulting in the smallest possible upper bound on the generalization error. Hence, the TRO model can potentially achieve a tighter GE bound than both the SAA and DRO models, by leveraging the faster decay of $\alpha_{N,1}$ compared with the growth of $\alpha_{N,2}$. We leave the analysis related to finding the value of $\theta$ under which the TRO model can achieve a tighter GE for future work. This is in part because this investigation requires analyzing the generalization bounds of DRO models with  $\delta>0$, which is an open question. Prior studies have primarily focused on the generalization bound \eqref{eqn:generalization_error_bound_DRO} of DRO with $\delta=0$; see \cite{Kuhn_et_al:2024} for detailed discussions. Analyzing the generalization bounds of DRO models with  $\delta>0$ is beyond the scope of our paper.

\section{Asymptotic Properties} \label{sec:asymptotic}

In this section, we analyze the asymptotic properties of our TRO model. Specifically, in Section~\ref{subsec:asy_convergence}, we show the almost sure convergence of the optimal value $\upsilonh_N(\theta_N)$ and the set of optimal solutions $\calXh_N(\theta_N)$ of the TRO problem \eqref{model:trade-off_model} to their true counterparts when $N\rightarrow\infty$. In Section \ref{subsec:asy_dist}, we derive the asymptotic distribution of $\upsilonh_N(\theta_N)$ when $N\rightarrow\infty$.

\subsection{Asymptotic Convergence} \label{subsec:asy_convergence}

In this section, we show the almost sure convergence of the optimal value $\upsilonh_N(\theta_N)$  and the set of optimal solutions $\calXh_N(\theta_N)$ of our TRO model respectively to the true optimal value $v^\star$ and set of optimal solutions $\calX^*$ of problem \eqref{prob:SO} when $N\rightarrow\infty$ (Theorem \ref{thm:asymptotic_convergence}). Our proof idea of these convergence properties leverages the fact that the objective function of our TRO model has two components:  $\E_{\Probh_N}[f(\xb,\xib)]$ and  $\sup_{\Prob\in\calP_N}\E_{\Prob}[f(\xb,\xib)]$. Hence, to establish the desired convergence results, we want to ensure that $\E_{\Probh_N}[f(\xb,\xib)]$ converges to $\E_{\Prob^\star}[f(\xb,\xib)]$  and $\sup_{\Prob\in\calP_N}\E_{\Prob}[f(\xb,\xib)]$ does not diverge (to infinity) as $N\rightarrow\infty$ (Lemma \ref{lem:DRO_component_asym_boundedness}). First, we adopt the following standard assumption that guarantees the uniform convergence of  $\E_{\Probh_N}[f(\xb,\xib)]$ over $\xb\in\calX$. In Example \ref{eg:GC_class}, we provide two sufficient conditions commonly employed in the literature for this assumption to hold.

\begin{assumption}\label{assumption:GC-class}
    The function class $\calH:=\{f(\xb,\cdot):\Xi\rightarrow\R\mid \xb\in\calX\}$ is $\Prob^\star$-Glivenko-Cantelli, i.e., 
    $$\norms{\Probh_N - \Prob^\star}_{\calH} :=\sup_{h\in\calH} \Bigg| \frac{1}{N} \sum_{i=1}^N h(\xibh_i) - \E_{\Prob^\star}[h(\xib)] \Bigg| \rightarrow 0 $$
    almost surely as $N\rightarrow\infty$.
\end{assumption}

\begin{example} \label{eg:GC_class}
    Suppose the samples $\{\xibh_i\}_{i=1}^N$ are i.i.d. generated from the true distribution $\Prob^\star$. The following are two sufficient conditions for $\calH$ to be $\Prob^\star$-Glivenko-Cantelli.
    \begin{enumerate}[itemsep=0mm,topsep=1mm]
        \item [(a)] $\calX$ is compact; $f(\cdot,\xib)$ is continuous for $\Prob^\star$-almost every $\xib\in\Xi$; $f(\xb,\xib)$ is dominated by an integrable function for all $\xb\in\calX$ (see Theorem 7.53 of \citealp{Shapiro_et_al:2014}).
        \item[(b)] $\calX$ is compact; $f(\cdot,\xib)$ is equicontinuous, i.e., for any $\varepsilon>0$, there exists $\delta>0$ such that $|f(\xb,\xib)-f(\xb',\xib)|\leq\varepsilon$ for all $\xib\in\Xi$ and $\xb'\in\calX$ with $\norms{\xb-\xb'}\leq \delta$ (see Lemma EC.5 of \citealp{Bertsimas_Kallus:2020}).
    \end{enumerate}
\end{example}

The second step toward {deriving the desired convergence results is to establish a finite upper bound on the DRO component $\sup_{\Prob\in\calP_N}\E_{\Prob}[f(\xb,\xib)]$. Note that under Assumption \ref{assumption:loss_and_ambig_set_boundedness}, we have 
\begin{align} 
    \sup_{\Prob\in\calP_N} \E_{\Prob}|f(\xb,\xib)| &\leq \sup_{\Prob\in\calP_N} \E_{\Prob}|f(\xb,\xib) - f(\xb_0,\xib)| +  \sup_{\Prob\in\calP_N}\E_{\Prob}|f(\xb_0,\xib)| \nonumber \\
    &\leq \sup_{\Prob\in\calP_N} \E_{\Prob}[\kappa(\xib)] \cdot \diam(\calX) + \sup_{\Prob\in\calP_N} \E_{\Prob}|f(\xb_0,\xib)|, \label{eqn:DRO_exp_UB_w_data}
\end{align}
where $\diam(\calX)=\sup_{\xb\in\calX,\,\xb'\in\calX} \norms{\xb-\xb'}<\infty$. From \eqref{eqn:DRO_exp_UB_w_data}, we observe that the choice of ambiguity set $\calP_N$ plays a critical role in the finiteness of $\sup_{\Prob\in\calP_N}\E_{\Prob}[f(\xb,\xib)]$. 
In particular,  $\calP_N$ should be carefully chosen to ensure that $\sup_{\Prob\in\calP_N} \E_{\Prob}[\kappa(\xib)]$ and $\sup_{\Prob\in\calP_N} \E_{\Prob}|f(\xb_0,\xib)|$ in \eqref{eqn:DRO_exp_UB_w_data} are finite for all $N\in\N$. To this end, we make the following assumptions on the objective function and/or the sequence of ambiguity sets $\{\calP_N\}_{N\in\N}$.

\begin{assumption}\label{assumption:ambig_set_regularity}
Either the objective function is uniformly bounded, i.e., $\sup_{\xb\in\calX}\sup_{\xib\in\Xi} |f(\xb,\xib)|<\infty$, or the sequence of ambiguity sets $\{\calP_N\}_{N\in\N}$ satisfies one of the following conditions.
\begin{enumerate}[itemsep=0mm,topsep=1mm]
    \item [(a)] There exists $\calPh\subseteq\calP(\Xi)$ (independent of the data) such that $\calP_N\subset\calPh$ almost surely for sufficiently large $N$ with  $\sup_{\Prob\in\calPh}\E_{\Prob}|f(\xb_0,\xib)|<\infty$ and $\sup_{\Prob\in\calPh}\E_{\Prob}|\kappa(\xib)|<\infty$. 
    \item [(b)] There exists $\calPh\subseteq\calP(\Xi)$ (independent of the data) such that $\bbmH(\calP_N,\calPh)\rightarrow 0$ almost surely as $N\rightarrow\infty$ with  $\sup_{\Prob\in\calPh}\E_{\Prob}|f(\xb_0,\xib)|<\infty$ and $\sup_{\Prob\in\calPh}\E_{\Prob}|\kappa(\xib)|<\infty$.
    \item [(c)] For ambiguity sets $\calP_N$ containing only distributions with support on the i.i.d. sample $\{\xibh_i\}_{i=1}^N$, which can be identified as a vector $\pb\in\R^N$, we have $\limsup_{N\rightarrow\infty} \sup_{\Prob\in\calP_N}\norms{N\pb-\one}_\infty <\infty$ almost surely with $\E_{\Prob^\star}|f(\xb_0,\xib)|<\infty$ and $\E_{\Prob^\star}|\kappa(\xib)|<\infty$.
\end{enumerate}
\end{assumption}

The uniform boundedness assumption on $f$ has been widely adopted in the literature and holds valid in various real applications, for example, when $f(\xb,\xib)$ represents the cost of action $\xb$ under scenario $\xib$, and both the sets $\calX$ and $\Xi$ are compact. Next, in Examples~\ref{eg:ambig_set_seq_1}--\ref{eg:ambig_set_seq_5}, we provide sequences $\{\calP_N\}_{N\in\N}$ that satisfy Assumptions~\ref{assumption:ambig_set_regularity}(a)--(c).

\begin{example} \label{eg:ambig_set_seq_1}
If $\calP_N=\calPh$ for some fixed $\calPh\subseteq\calP(\Xi)$, then Assumption~\ref{assumption:ambig_set_regularity}(a) holds immediately. For example,  in robust optimization models, $\calPh$ may contain Dirac measures on a set of scenarios $\Xi'\subseteq \Xi$, i.e., $\calPh=\{\delta_{\xib'}\mid \xib'\in\Xi'\}$. 
\end{example}

\begin{example} \label{eg:ambig_set_seq_2}
Consider the moment-based ambiguity set of the form \eqref{eqn:moment_ambig_set} with dimension $d=1$. If $\mu_i^N$ for $i\in\{1,\dots,q\}$ are the sample estimates, then  $\bbmH(\calP_N,\calPh)\rightarrow 0$ almost surely under mild conditions, where $\calPh$ is the moment-based ambiguity set with true moments $\mu_i$ (see, e.g., \citealp{Sun_Xu:2016}); thus, Assumption~\ref{assumption:ambig_set_regularity}(b) holds.
\end{example}

\begin{example} \label{eg:ambig_set_seq_3}
Consider the distance-based ambiguity set $\calP_N=\{\Prob\in\calP(\Xi)\mid \sfd(\Prob,\Probh_N)\leq r\}$ for some radius $r>0$. Here, $\sfd$ is a statistical distance satisfying (i) $\sfd(\Prob,\Prob)=0$ for any $\Prob\in\calP(\Xi)$, (ii) $\sfd(\Prob_1,\Prob_2)=\sfd(\Prob_2,\Prob_1)$ for any $\{\Prob_1,\Prob_2\}\subseteq\calP(\Xi)$, (iii) $\sfd(\Prob_1,\Prob_2)\leq\sfd(\Prob_1,\Prob_3)+\sfd(\Prob_3,\Prob_2)$ for any $\{\Prob_1,\Prob_2,\Prob_3\}\subseteq\calP(\Xi)$, and (iv) $\sfd$ is convex in the first argument. These properties hold, for example, if $\sfd$ is the Wasserstein metric. Let $\calPh=\{\Prob\in\calP(\Xi)\mid \sfd(\Prob,\Prob^\star)\leq r\}$. If $\sup_{\Prob_1\in\calP(\Xi),\,\Prob_2\in\calP(\Xi)} \bbmd(\Prob_1,\Prob_2)<\infty$ and $\sfd(\Probh_N,\Prob^\star)\rightarrow0$ almost surely, then $\bbmH(\calP_N,\calPh)\rightarrow 0$ almost surely as $N\rightarrow\infty$ (see \ref{apdx:distance_ambig_set_convergence} for a proof); thus, Assumption~\ref{assumption:ambig_set_regularity}(b) holds.
\end{example}

\begin{example} \label{eg:ambig_set_seq_4}
Consider the $\phi$-divergence ambiguity set $\calP_N$ in \eqref{eqn:phi_div_ambig_set} with radius $r/N$, where we note that $\calP_N$ only consists of distributions with support $\{\xibh_i\}_{i=1}^N$. Under some smoothness conditions on $\phi$ such as differentiability, Lemma 13 in \cite{Duchi_et_al:2021} shows that $\sup_{N\in\N}\sup_{\Prob\in\calP_N}\norms{N\pb-\one}_2 \leq \sqrt{rC_\phi}$ almost surely for some constant $C_\phi$ that depends on $\phi$ only. Since $\norms{\cdot}_\infty\leq\norms{\cdot}_2$, our Assumption~\ref{assumption:ambig_set_regularity}(c) holds.
\end{example}

\begin{example} \label{eg:ambig_set_seq_5}
 Consider the total variation ambiguity set $\calP_N$ by setting $\phi(t)$ as the the non-differentiable function $|t-1|$ in \eqref{eqn:phi_div_ambig_set} with radius $r/N$, i.e., $\calP_N=\big\{\Prob=\sum_{i=1}^N p_i \delta_{\xibh_i}\mid \pb\in\R_+^N,\, \pb^\tp\one=1,\, \norms{\pb-N^{-1}\one}_1\leq r/N \big\}$.  From the definition of $\calP_N$, we immediately have $\sup_{\Prob\in\calP_N} \norms{N\pb-\one}_1\leq r$ almost surely for all $N\in\N$, implying that Assumption~\ref{assumption:ambig_set_regularity}(c) holds.
\end{example}

In Lemma \ref{lem:DRO_component_asym_boundedness}, we show that $\sup_{\Prob\in\calP_N}\E_{\Prob}[f(\xb,\xib)]$ is upper bounded for sufficiently large $N$ under our Assumptions \ref{assumption:loss_and_ambig_set_boundedness} and \ref{assumption:ambig_set_regularity}.

\begin{lemma} \label{lem:DRO_component_asym_boundedness}
Under Assumptions \ref{assumption:loss_and_ambig_set_boundedness} and \ref{assumption:ambig_set_regularity}, we have $\limsup_{N\rightarrow\infty} \sup_{\xb\in\calX} \sup_{\Prob\in\calP_N} \E_{\Prob}|f(\xb,\xib)|\leq M$ almost surely for some constant $M$.
\end{lemma}

With Lemma \ref{lem:DRO_component_asym_boundedness}, we are ready to prove the asymptotic convergence of our TRO model. Specifically, in Theorem \ref{thm:asymptotic_convergence}, we show that, for any sequence $\{\theta_N\}_{N\in\N}$ converging to zero, the optimal value $\upsilonh_N(\theta_N)$ and the set of optimal solutions $\calXh_N(\theta_N)$ of our TRO model converges almost surely to the true optimal value $\upsilon^\star$ and the set of optimal solutions $\calX^\star$ to \eqref{prob:SO}, respectively.

\begin{theorem}  \label{thm:asymptotic_convergence}
In addition to Assumptions \ref{assumption:loss_and_ambig_set_boundedness}--\ref{assumption:ambig_set_regularity}, suppose that $\theta_N=o(1)$. Then, almost surely, as $N\rightarrow\infty$, we have (i) $\upsilonh_N(\theta_N)\rightarrow\upsilon^\star$, and (ii) $D\big(\calXh_N(\theta_N),\calX^\star\big)\rightarrow 0$.
\end{theorem}

\subsection{Asymptotic Distribution} \label{subsec:asy_dist}

We now derive the asymptotic distribution of $\upsilonh_N(\theta_N)$. We use the following additional notation in our derivations. Let $L^2(\Prob^\star)$ be the set of $\Prob^\star$-square-integrable functions equipped with the norm $\norms{h}_{L^2(\Prob^\star)}=\sqrt{\E_{\Prob^\star}[h(\xib)^2]}$ for $h\in L^2(\Prob^\star)$. For a class of functions $\calH\subseteq L^2(\Prob^\star)$, let $\ell^\infty(\calH)=\{g:\calH\rightarrow\R\mid \norms{g}_\calH <\infty \}$ be the space of bounded functions, where $\norms{g}_\calH:=\sup_{h\in\calH} |g(h)|$ is the sup-norm of $g\in\ell^\infty(\calH)$. Following the convention in empirical process theory, we use the shorthand notation $\Prob (h):= \E_{\Prob}(h)$ for $h\in\calH$. We adopt the following standard assumption on the complexity of the objective function class (see, e.g., \citealp{Eichhorn_Romisch:2007, Lam:2019, Lam:2022}).

\begin{assumption}\label{assumption:Donsker}
The function class $\calH:=\{f(\xb,\cdot):\Xi\rightarrow\R\mid \xb\in\calX\}$ is $\Prob^\star$-Donsker, i.e., $\sqrt{N} (\Probh_N-\Prob^\star) \Rightarrow \G$ in $\ell^\infty(\calH)$, where ``$\Rightarrow$'' denotes weak convergence and $\G$ is a tight Gaussian process indexed by $\calH$ with mean zero and covariance function $\Cov(\G(h_1),\G(h_2))=\Cov_{\Prob^\star}(h_1(\xib),h_2(\xib))$.
\end{assumption}

In Assumption~\ref{assumption:Donsker}, the measures $\Probh_N$ and $\Prob^\star$ are considered as elements in $\ell^\infty(\calH)$. Example \ref{eg:Donsker_Lipschitz_suff} provides a sufficient condition for Assumption~\ref{assumption:Donsker} to hold.

\begin{example} \label{eg:Donsker_Lipschitz_suff}
Suppose that $\{\xibh_1,\xibh_2,\dots\}$ is i.i.d. following $\Prob^\star$, $\calX$ is compact, $\E_{\Prob^\star}[f(\xb,\xib)]<\infty$, and $\Var_{\Prob^\star}(f(\xb,\xib))<\infty$ for all $\xb\in\calX$. If there exists a measurable function $\kappa:\xib\rightarrow\R_+$ such that $f(\cdot,\xib)$ is Lipschitz with modulus $\kappa(\xib)$ almost surely with  $\E_{\Prob^\star}[\kappa(\xib)^2]<\infty$, then $\calH$ is $\Prob^\star$-Donsker \citep{Lam:2019}. This sufficient condition is widely adopted in the literature to derive asymptotic distribution in stochastic optimization problems \citep{Guigues_et_al:2018, Shapiro_et_al:2014}.
\end{example}

A key step toward deriving the asymptotic distribution of $\upsilonh_N(\theta_N)$ is to show that $\S_N:=\sqrt{N}\big[(1-\theta_N)\Probh_N+\theta_N\Prob_N-\Prob^\star\big]$ converges weakly to some tight Gaussian process $\G'$ in $\ell^\infty(\calH)$ for any $\Prob_N\in\calP_N$. Here,  $(1-\theta_N)\Probh_N+\theta_N\Prob_N$ is a probability measure in the TRO ambiguity set.  Lemma~\ref{lem:asy_tight_tradeoff_ambig_set} establishes the desired convergence under some mild assumptions.

\begin{lemma} \label{lem:asy_tight_tradeoff_ambig_set}
Let $\{\Prob_N\}_{N\in\N}$ be any sequence of probability measure $\Prob_N\in\calP_N$. In addition to Assumptions~\ref{assumption:loss_and_ambig_set_boundedness} and \ref{assumption:ambig_set_regularity}, suppose that
\begin{enumerate}[itemsep=0mm,topsep=1mm]
    \item [(a)] there exists a square-integrable envelope $H$ of the function class $\calH$, i.e., $h(\xib)\leq H(\xib)$ for all $h\in\calH$ with $\E_{\Prob^\star}[H(\xib)^2]<\infty$, and
    \item [(b)] $\theta_N=o(N^{-1/2})$.
\end{enumerate}
Then, as $N\rightarrow\infty$, the process $\S_N:=\sqrt{N}[(1-\theta_N)\Probh_N + \theta_N\Prob_N-\Prob^\star]\Rightarrow\G'$ in $\ell^\infty(\calH)$, where $\G'$ is a tight Gaussian process indexed by $\calH$ with mean zero and covariance function $\Cov(\G'(h_1),\G'(h_2))=\Cov_{\Prob^\star}(h_1(\xib),h_2(\xib))$.
\end{lemma}

With Lemma \ref{lem:asy_tight_tradeoff_ambig_set}, we are ready to derive the asymptotic distribution of $\upsilonh_N(\theta_N)$. Specifically, in Theorem~\ref{thm:asy_dist}, we show that $\sqrt{N}(\upsilonh_N(\theta_N)-\upsilon^\star)$ converges to the infimum of some Gaussian process indexed by $\calX$. For notational simplicity, we write $\upsilonh_N=\upsilonh_N(\theta_N)$.

\begin{theorem} \label{thm:asy_dist}
In addition to assumptions in Lemma \ref{lem:asy_tight_tradeoff_ambig_set}, suppose that Assumption~\ref{assumption:Donsker} holds and there exists a worst-case distribution $\Prob_N^\star \in\argmax_{\Prob\in\calP_N} \E_{\Prob}[f(\xb,\xib)]$ such that $\Prob^\star_N\in\calP_N$ for any $\xb\in\calX$ and $N\in\N$. Then, as $N\rightarrow\infty$,
\begin{enumerate}[itemsep=0mm,topsep=1mm]
    \item [(i)] $\sqrt{N}(\upsilonh_N-\upsilon^\star) \Rightarrow \inf_{\xb\in \calX^\star} \G(\xb)$, 
    where $\G$ is a tight Gaussian process indexed by $\calX$ with mean zero and covariance function $\Cov(\G(\xb_1)),\G(\xb_2))=\Cov_{\Prob^\star}(f(\xb_1,\xib),f(\xb_2,\xib))$;
    \item[(ii)] $\upsilonh_N-\upsilon^\star = \inf_{\xb\in\calX^\star}  \Big\{ (1-\theta_N)\E_{\Probh_N}[f(\xb,\xib)] + \theta_N \sup_{\Prob\in\calP_N}\E_{\Prob}[f(\xb,\xib)] - \E_{\Prob^\star}[f(\xb,\xib)] \Big\} + o_{\Prob^\star}(N^{-1/2})$.
\end{enumerate}
In particular, $\upsilonh_N= \inf_{\xb\in\calX^\star}\Big\{ (1-\theta_N)\E_{\Probh_N}[f(\xb,\xib)] + \theta_N \sup_{\Prob\in\calP_N}\E_{\Prob}[f(\xb,\xib)]\Big\} +o_{\Prob^\star}(N^{-1/2})$.
\end{theorem}

It follows from Theorem~\ref{thm:asy_dist} that when set of optimal solutions to \eqref{prob:SO} is a singleton, say $\calX^\star=\{\xb^\star\}$, $\sqrt{N}(\upsilonh_N-\upsilon^\star) \Rightarrow N\big(0,\Var_{\Prob^\star}(f(\xb^\star,\xib))\big)$, i.e., a normal distribution with mean zero and variance $\Var_{\Prob^\star}(f(\xb^\star,\xib))$. 

Our asymptotic convergence results hold for TRO models with TRO ambiguity sets constructed using general shape parameters $\calP_N$, such as moment- and distance-based ambiguity sets. This differs from results in the existing literature focusing on a specific ambiguity set. In the special case where the shape parameter is chosen as a distance-based ambiguity set $\calP_{N,r_N}=\{\Prob\in\calP(\Xi)\mid \sfd(\Prob,\Probh_N)\leq r_N\}$, we can recover the asymptotics of the optimal value of classical distance-based DRO models (see, e.g., \citealp{Blanchet_et_al:2021, Blanchet_Shapiro:2023}). Specifically, in Theorem~\ref{thm:asy_dist_distance_based}, we derive the asymptotic distribution of $\upsilonh_N$ under three different convergence rates of the size parameter $\theta_N$ and the radius $r_N$ in the shape parameter $\calP_{N,r_N}$.

\begin{theorem} \label{thm:asy_dist_distance_based}
Let $\G$ be a tight Gaussian process indexed by $\calX$ with mean zero and covariance function $\Cov(\G(\xb_1)),\G(\xb_2))=\Cov_{\Prob^\star}(f(\xb_1,\xib),f(\xb_2,\xib))$. Also, let $\varepsilon_N(\xb)$ represent any term that converges to zero in probability uniformly over $\xb\in\calX$, i.e., $\sup_{\xb\in\calX} |\varepsilon_N(\xb)|=o_{\Prob^\star}(1)$. In addition to Assumptions~\ref{assumption:loss_and_ambig_set_boundedness} and \ref{assumption:Donsker}, suppose we construct the TRO ambiguity set using a distance-based shape parameter $\calP_{N,r_N}=\{\Prob\in\calP(\Xi)\mid \sfd(\Prob,\Probh_N)\leq r_N\}$ and that $\sup_{\Prob\in\calP_{N,r_N}} \E_{\Prob}[f(\xb,\xib)]$ exhibits the following expansion
\begin{equation} \label{eqn:DRO_expansion}
    \sup_{\Prob\in\calP_{N,r_N}} \E_{\Prob}[f(\xb,\xib)] = \E_{\Probh_N}[f(\xb,\xib)] + r^\gamma_N g_N(\xb)+r^\gamma_N \varepsilon_N(\xb) 
\end{equation}
for some $\gamma>0$, where $g_N(\xb)$ satisfies $g_N(\xb)=h(\xb)+\varepsilon_N(\xb)$ for some continuous deterministic process $h(\xb)$. Then, the following assertions hold.
\begin{itemize}[itemsep=0mm,topsep=1mm]
    \item [(i)] If $\theta_N r^\gamma_N=o(N^{-1/2})$, then $\sqrt{N}(\upsilonh_N-\upsilon^\star)\Rightarrow \inf_{\xb\in\calX^\star} \G(\xb)$. 
    \item[(ii)] If $\theta_N r^\gamma_N= N^{-1/2}$, then $\sqrt{N}(\upsilonh_N-\upsilon^\star)\Rightarrow \inf_{\xb\in\calX^\star}\{\G(\xb)+h(\xb)\}$.
    \item[(iii)] If $o(\theta_N r^\gamma_N)=N^{-1/2}$, then $\theta_N^{-1} r^{-\gamma}_N (\upsilonh_N-\upsilon^\star)\Rightarrow  \inf_{\xb\in\calX^\star} h(\xb)$.
\end{itemize}
\end{theorem}
\color{black}

Theorem \ref{thm:asy_dist_distance_based}  demonstrates that when the TRO ambiguity set is constructed using a distance-based shape parameter $\calP_{N,r}$ such that $\sup_{\Prob\in\calP_{N,r_N}} \E_{\Prob}[f(\xb,\xib)]$ satisfies \eqref{eqn:DRO_expansion}, the asymptotic distribution of $\upsilonh_N$ depends on the convergence rate of $\theta_N r^\gamma_N$. In particular, when $r_N=r$ is fixed, Theorem~\ref{thm:asy_dist_distance_based} implies that  $\upsilonh_N$ converges to three different distributions depending on the convergence rate of the size parameter $\theta_N$. This resembles the asymptotics of distance-based DRO models, where different convergence rates of the radius $r_N$ lead to different asymptotic distributions of the DRO optimal value (see, e.g., \citealp{Blanchet_et_al:2021, Blanchet_Shapiro:2023}). These results indicate that the size parameter, $\theta_N$, of the TRO model with TRO ambiguity set constructed using the shape parameter $\calP_{N,r}$ and the radius $r_N$ in classical distance-based DRO models play a similar role in controlling the asymptotic distribution of the model's optimal value. Finally, we highlight that the expansion~\eqref{eqn:DRO_expansion} is commonly used in the DRO literature to derive asymptotics of the DRO optimal value and is satisfied by popular distance-based ambiguity sets such as $\phi$-divergence and Wasserstein ambiguity sets \citep{Blanchet_Shapiro:2023}.

\section{Numerical Examples} \label{sec:numerics}

In this section, we illustrate our theoretical results in the context of two stylized optimization problems:  an inventory control problem (Section~\ref{subsec:inventory_control_problem}) and a portfolio optimization problem (Section~\ref{subsec:portfolio_optimization_problem}).

\subsection{Inventory Control} \label{subsec:inventory_control_problem}

Consider the classical inventory control problem, where the decision-maker needs to decide the order quantity $x$ of a single item before observing the random demand $\xi\in\R$ \citep{Gallego_Moon:1993}. If $x$ units are ordered, then $\min\{x,\xi\}$ units are sold and  $(x-\xi)_+$ units are salvaged. Let  $c$ be the per-unit ordering cost, $p$ be the per-unit selling price, and $h$ be the per-unit salvage value. The goal is to minimize the expected ordering cost minus the expected profit and salvage value: 
$$\underset{x\geq 0}{\text{minimize}}\quad  \E_{\Prob^\star}[f(x,\xi)]= \E_{\Prob^\star}[cx - p\min\{x,\xi\} - h(x-\xi)_+], $$
where $\Prob^\star$ is the true distribution of $\xi$. Writing $f(x,\xi)=(c-h)x+(p-h)\big[(\xi-x)_+-\xi\big]$ (see \citealp{Gallego_Moon:1993}), we formulate the following TRO model for this problem 
\begin{equation} \label{eqn:inventory_contol_trade_off}
    \underset{x\geq 0}{\text{minimize}}\quad (c-h)x + (p-h) \bigg\{ (1-\theta) \frac{1}{N}\sum_{i=1}^N \big[(\xih_i-x)_+ -\xih_i\big] + \theta \sup_{\Prob\in\calP_N} \E_{\Prob}\big[(\xi-x)_+-\xi\big] \bigg\}, 
\end{equation}
where $\{\xih_i\}_{i=1}^N$ is the set of samples. Note that $f(x,\xi)$ is Lipschitz in $x$ with Lipschitz constant $c+p-2h$ (independent of $\xi$). Moreover, one can impose a practical upper bound on $x$ so that the feasible set $\calX$ is compact, and thus, Assumption \ref{assumption:loss_and_ambig_set_boundedness} holds.

We consider the following shape parameters (ambiguity sets) $\calP_N$ in our experiment: (a) the mean-variance ambiguity set \citep{Gallego_Moon:1993}, (b) the 1-Wasserstein ambiguity set \citep{Mohajerin-Esfahani_Kuhn:2018}, (c) the empirical Burg-entropy divergence ball \citep{Lam:2019}, (d) a set of Dirac measures on points that lie in the confidence interval of $\xi$ \citep{Fabozzi_et_al:2007}. Table \ref{table:inventory_control_reformulations} summarizes these sets and the parameter settings we used for each.  It is easy to verify that ambiguity sets (a)--(d) satisfy Assumption~\ref{assumption:ambig_set_regularity} (see \ref{apdx:num_expt_IC_check_assumption} for a proof for ambiguity set (d)). In \ref{apdx:num_expt_IC_reform}, we provide tractable reformulations of the TRO model \eqref{eqn:inventory_contol_trade_off} under each set. We adopt similar parameter settings as in \cite{Gotoh_et_al:2021}. Specifically, we set $p =30$, $c=2$, and $h=1$. For illustrative purposes, we use the exponential distribution with mean $50$ as the true distribution of the demand $\xi$ for generating the data. Under this distribution, we can compute the true optimal value $\upsilon^\star=-1,232$ \citep{Gallego_Moon:1993}.
\begin{table}[t]\centering\small 
\footnotesize
\caption{Shape parameters for the inventory control problem. \textit{Notation:} $\muh_N$ is the sample mean; $\sigmah_N^2$ is the sample variance; $\Delta_N\subseteq\R^N$ is the probability simplex; $t_{\nu,\alpha/2}$ is the upper $(1-\alpha)/2$-th quantile of a Student's $t$-distribution with degree of freedom $\nu$.} \label{table:inventory_control_reformulations}
\ra{0.8}  
\begin{tabular}{@{}l|c|c@{}} \toprule
& Ambiguity Set $\calP_N$ & Parameters\\ \midrule
(a) Mean-Variance & $\calP_N=\{\Prob\in\calP(\R_+)\mid \E_{\Prob}(\xi)=\mu,\, \Var_{\Prob}(\xi)=\sigma^2\}$ & $(\mu,\sigma^2)=(\muh_N,\sigmah_N^2)$  \\ [.5ex]
(b) 1-Wasserstein & $\calP_N=\{\Prob\in\calP(\R_+)\mid W_1(\Prob,\Probh_N) \leq r\}$ & $r=100$ \\ [.5ex]
(c) Burg-Divergence & $\calP_N=\{\pb\in\Delta_N \mid -\frac{1}{N}\sum_{i=1}^N \log(Np_i) \leq \frac{r}{N}\}$ & $r=10$ \\ [.5ex]
(d) Confidence Interval & $\calP_N=\{\delta_{\xi}\mid \xi\geq 0,\, |\xi-\muh_N| \leq t_{N-1,\alpha/2}\sigmah_N/\sqrt{N}\}$ & $\alpha = 0.95$\\
\bottomrule
\end{tabular}
\end{table}

First, to demonstrate the effect of $\theta$ on the optimal solution and value, we solve our TRO model with different values of $\theta\in\{0,0.01,0.02,\dots,1\}$ and $N=100$. Figure \ref{fig:expt_IC_opt_change} illustrates the resulting optimal solutions and objective values. Clearly, the optimal value function is concave under shape parameters (a)--(d), which is consistent with Theorem~\ref{thm:conservatism}. The optimal value function is increasing on the entire interval [0,1] under sets (a)--(c). This is because sets (a)--(c) are star-shaped with a star center $\Probh_N\in\calP_N$, and thus, the sequence of TRO ambiguity sets constructed using sets (a)--(c) satisfy the hierarchical properties in Theorem~\ref{thm:calP_nondecreasing}. In contrast, the optimal value function first increases on $\theta\in[0,0.85]$ and then decreases. This is because, by the construction of the set (d), $\Probh_N\not\in\calP_N$, and thus, $\Probh_N$ is not a star center of $\calP_N$. It follows that the sequence of TRO ambiguity sets constructed using set (d) does not satisfy the hierarchical property. We also observe that different choices of the shape parameter $\calP_N$  result in different spectra of optimal solutions and values. For example, the TRO model with shape parameter (a) suggests ordering more under a larger $\theta$ and suggests ordering less when using shape parameters (c) and~(d). In \ref{apdx:add_expt_results:conservatism}, we assess the conservatism of these optimal solutions by measuring how much the estimated optimal value, $\upsilonh_N(\theta)$, exceeds the actual expected objective function value (cost), $\E_{\Prob^\star}[f(\xb_N(\theta),\xib)]$, associated with implementing the optimal solution $\xb_N(\theta)$. The results in  \ref{apdx:add_expt_results:conservatism} show that the conservatism of the TRO optimal solution generally increases with $\theta$. However, as demonstrated in \ref{apdx:add_expt_results:out_of_sample}, adopting solutions on the spectrum of TRO optimal solutions can yield lower out-of-sample costs than the SAA and DRO solutions. This highlights the practical benefits of the TRO approach.

\begin{figure}[t]
    \centering
    \includegraphics[scale=0.7]{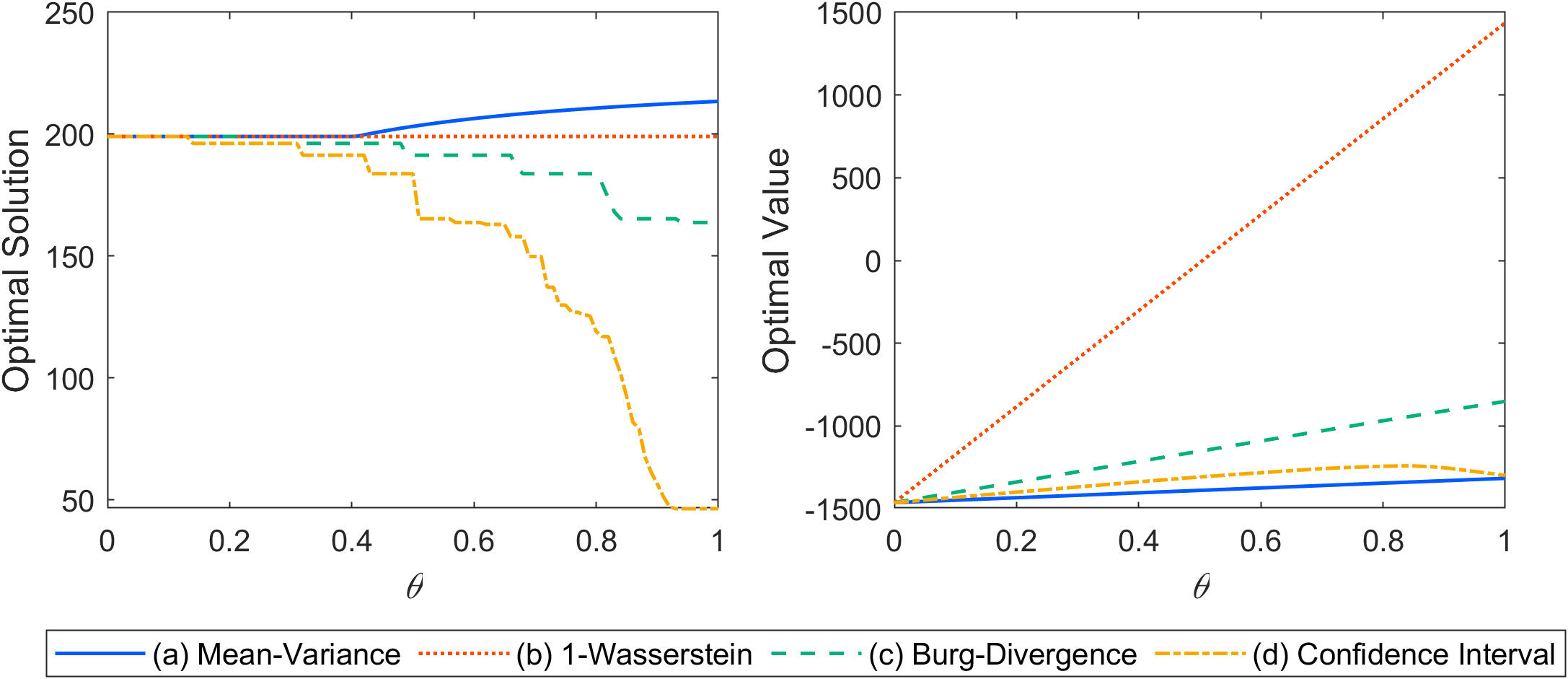}
    \caption{Optimal solution and optimal value for different values of $\theta$ in the inventory control problem.}
    \label{fig:expt_IC_opt_change}
\end{figure}

Let us now investigate the bias and standard deviation of the TRO estimator $\upsilonh_N(\theta)$. We estimate these quantities as follows. First, for each $\theta\in\{0,0.01,0.02,\dots,1\}$, we solve our TRO model with $R=1,000$ samples, each consisting of $N=10$ scenarios.  Let  $\{\upsilonh_{N,r}(\theta)\}_{r=1}^R$ denote the resulting optimal values of the TRO model from the $R$ replications under each $\theta$. We compute the mean of $\{\upsilonh_{N,r}(\theta)\}_{r=1}^R$ as $\texttt{mean}_N(\theta)=R^{-1} \sum_{r=1}^R \upsilonh_{N,r}(\theta)$. Then, we estimate the bias as $\texttt{bias}_N(\theta)=\texttt{mean}_N(\theta)-\upsilon^\star$ and the standard deviation as $\texttt{std}_N(\theta)=R^{-1}\sum_{r=1}^R[\upsilonh_{N,r}(\theta)-\texttt{mean}_N(\theta)]^2$. Figure~\ref{fig:expt_IC_mean_variance} presents the estimated bias and standard deviation. As expected, the SAA estimator $\upsilonh_N(0)$ exhibits a downward bias. In contrast, the TRO estimator $\upsilonh_N(\theta)$ is an unbiased estimator for some $\theta$.  Moreover, when $\theta$ is sufficiently small, the absolute bias of  $\upsilonh_N(\theta)$ is smaller than that of the SAA estimator $\upsilonh_N(0)$. Note that the absolute bias of $\upsilonh_N(\theta)$ varies across different shape parameters employed in the model. For example, when using set (a), the absolute bias of $\upsilonh_N(\theta)$ is smaller than that of the SAA estimator for $\theta\in(0,1]$. In contrast, when using sets (b)--(d), the absolute bias of $\upsilonh_N(\theta)$ is smaller for $\theta\in(0,0.09]$. These results are consistent with Theorem~\ref{thm:debias_theta} and Corollary~\ref{cor:debias_theta}. Figure~\ref{fig:expt_IC_mean_variance} shows that the standard deviation of $\upsilonh_N(\theta)$ decreases when $\theta$ increases from zero, demonstrating that the TRO estimator's variability reduces with  $\theta$. Moreover, using sets (a), (c), or (d) as the shape parameter produces TRO estimators that achieve both lower bias and standard deviation than the SAA estimator. Furthermore,  we observe that none of the shape parameters consistently produce a TRO estimator with the best bias-variance trade-off; see \ref{apdx:add_expt_results:bias_variance} for further discussions.
\begin{figure}[t]
    \centering
    \includegraphics[scale=0.7]{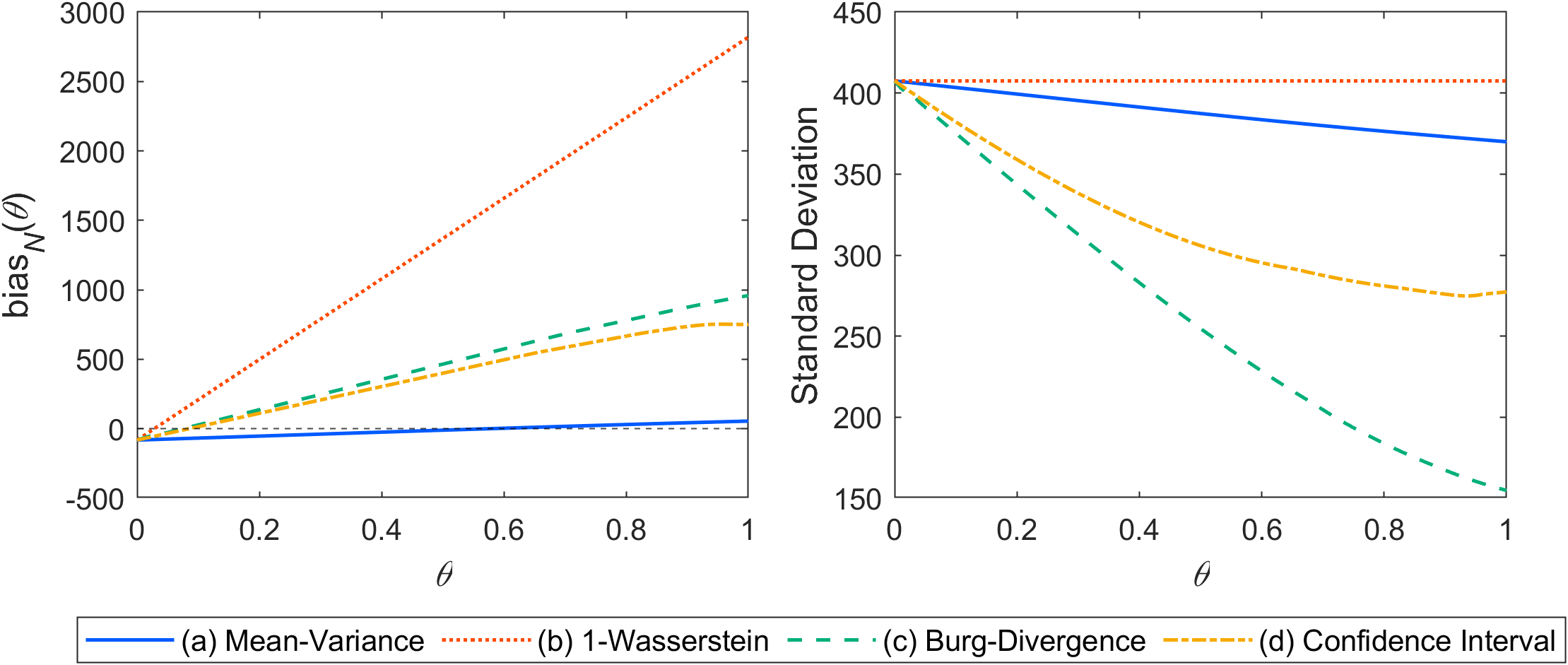}
    \caption{Bias and standard deviation of $\upsilonh_N(\theta)$ for different values of $\theta$ with $N=10$ in the inventory control problem.}
    \label{fig:expt_IC_mean_variance}
\end{figure}

Finally, we demonstrate the asymptotic properties of $\upsilonh_N(\theta_N)$. First, to show the asymptotic convergence of $\upsilonh_N(\theta_N)$, we solve the TRO model with $N\in\{10,50,100,500,1000\}$ and $\theta_N=10/N=o(N^{-1/2})$ (consistent with the rate suggested in Lemma~\ref{lem:asy_tight_tradeoff_ambig_set}) and  compute the absolute difference $|\upsilonh_N(\theta_N)-\upsilon^\star|$ between the optimal value of the TRO model $\upsilonh_N(\theta_N)$ and the true optimal value $\upsilon^\star$. Table~\ref{table:expt_IC_absolute_diff} presents $|\upsilonh_N(\theta_N)-\upsilon^\star|$ for $N\in\{10,50,100,500,1000\}$. It is clear that $|\upsilonh_N(\theta_N)-\upsilon^\star|$  decreases with $N$, illustrating the asymptotic convergence of $\upsilonh_N(\theta_N)$ to $\upsilon^\star$ as shown in Theorem~\ref{thm:asymptotic_convergence}.  Next, to demonstrate the asymptotic distribution of $\upsilonh_N(\theta_N)$, we solve the TRO model with $R=1,000$ sets of data, each consisting of $N$ scenarios, to obtain the optimal values $\{\upsilonh_{N,r}(\theta_N)\}_{r=1}^R$. We then compute the Kolmogorov–Smirnov (KS) statistic that quantifies the difference between the standard normal distribution and the empirical distribution of $\sqrt{N}\big[\upsilonh_N(\theta)-\upsilon^\star\big]/V^\star$ based on the samples $\{\upsilonh_{N,r}(\theta_N)\}_{r=1}^R$, where $V^\star=\Var_{\Prob^\star}\big(f(\xb^\star,\xib)\big)$ and $\xb^\star$ is the (unique) true optimal solution. Table~ \ref{table:expt_IC_KS_stat} presents the KS statistics for $N\in\{10,100,1000\}$. The KS statistic converges to zero when $N$ increases, suggesting that $\sqrt{N}\big[\upsilonh_N(\theta)-\upsilon^\star\big]/V^\star$ converges weakly to the standard normal distribution. These observations are consistent with Theorem~\ref{thm:asy_dist}. These results also emphasize that the asymptotic properties of $\upsilonh_N(\theta_N)$ hold for general shape parameters.

\begin{table}[t]\centering
\footnotesize
\caption{Absolute difference between $\upsilonh_N(\theta_N)$ and $\upsilon^\star$ for different $N$ in the inventory control problem.} \label{table:expt_IC_absolute_diff}
\ra{0.8}  
\begin{tabular}{@{}l|rrrrr@{}} \toprule
                        & $N = 10$ & $N = 50$ & $N = 100$ & $N = 500$ & $N = 1000$ \\ \midrule
(a) Mean-Variance       & 150.31   & 96.11    & 140.18    & 24.42     & 2.79       \\
(b) 1-Wasserstein       & 2535.99  & 457.37   & 137.80    & 31.65     & 25.24      \\
(c) Burg-Divergence     & 1004.17  & 29.20    & 95.42     & 21.28     & 2.03       \\
(d) Confidence Interval & 998.49   & 35.05    & 123.69    & 24.41     & 3.34       \\
\bottomrule
\end{tabular}
\end{table}
\begin{table}[t]\centering 
\footnotesize
\caption{Kolmogorov–Smirnov statistics for $\sqrt{N}\big[\upsilonh_N(\theta)-\upsilon^\star\big]/V^\star$ with $V^\star=\Var_{\Prob^\star}\big(f(\xb^\star,\xib)\big)$ for different $N$ in the inventory control problem.} \label{table:expt_IC_KS_stat}
\ra{0.8}  
\begin{tabular}{@{}l|rrr@{}} \toprule
                        & $N = 10$ & $N = 100$ & $N = 1000$ \\ \midrule
(a) Mean-Variance       & 0.1121   & 0.0666    & 0.0312     \\
(b) 1-Wasserstein       & 0.9986   & 0.7328    & 0.2743     \\
(c) Burg-Divergence     & 0.9071   & 0.1796    & 0.0365     \\
(d) Confidence Interval & 0.7422   & 0.1103    & 0.0269     \\
\bottomrule
\end{tabular}
\end{table}

\subsection{Portfolio Optimization} \label{subsec:portfolio_optimization_problem}

Consider the classical mean-risk portfolio optimization problem, where there are $n$ trading assets, and an investor needs to decide the proportion $\xb=(x_1,\dots,x_n)^\tp$ of the total investment amount allocated to each trading asset before observing the random return $\xib=(\xi_1,\dots,\xi_n)^\tp$.  Given the portfolio $\xb$, the corresponding portfolio loss is $-\xb^\tp\xib$. The goal is to minimize the sum of the expected portfolio loss and the portfolio risk measured by conditional value-at-risk (CVaR): 
\begin{equation} \label{eqn:mean-CVaR_model}
    \underset{\xb\in\calX}{\text{minimize}} \quad \beta\,\E_{\Prob^\star}(-\xb^\tp\xib) + (1-\beta)\,\Prob^\star\mhyphen\CVaR_{\alpha}(-\xb^\tp\xib), 
\end{equation}
where $\Prob^\star$ is the true distribution of $\xib$, $\beta\in(0,1)$, $\calX=\{\xb\in\R^n\mid \one^\tp\xb=1,\,\xb\geq0\}$, and 
\begin{equation} \label{eqn:CVaR}
    \Prob\mhyphen\CVaR_\alpha(-\xb^\tp\xib)=\min_t \bigg\{ t + \frac{1}{1-\alpha}\E_{\Prob}\big[(-\xb^\tp\xib-t)_+\big]\bigg\} 
\end{equation}
for some $\alpha\in(0,1)$ \citep{Rockafellar_Uryasev:2000}. Using \eqref{eqn:CVaR}, we can reformulate \eqref{eqn:mean-CVaR_model} into
\begin{equation} \label{eqn:mean-CVaR_model_linear}
    \underset{\xb\in\calX,\,t\in\R}{\text{minimize}} \quad \E_{\Prob^\star}\bigg[(1-\beta)t +\beta(-\xb^\tp\xib)+\frac{1-\beta}{1-\alpha} (-\xb^\tp\xib-t)_+\bigg]=\E_{\Prob^\star}[f(\xb,t,\xib)],
\end{equation}
where $f(\xb,t,\xib)=(1-\beta)t+\beta (-\xb^\tp\xib) + [(1-\beta)/(1-\alpha)](-\xb^\tp\xib-t)^+$. Using \eqref{eqn:mean-CVaR_model_linear}, we formulate the following TRO model for this problem 
\begin{equation} \label{eqn:mean-CVaR_model_linear_TRO}
    \underset{\xb\in\calX,\,t\in\R}{\text{minimize}} \quad (1-\theta)\, \frac{1}{N}\sum_{i=1}^Nf(\xb,t,\xibh_i)+\theta\sup_{\Prob\in\calP_N}\E_{\Prob}[f(\xb,t,\xib)], 
\end{equation}
where $\{\xibh_i\}_{i=1}^N$ is the set of samples.

Note that for any $\xb\in\calX$, a minimizer of \eqref{eqn:CVaR} over $t\in\R$ is the value-at-risk $\Prob^\star\mhyphen\VaR_\alpha(-\xb^\tp\xib)=\inf\{t\in\R\mid \Prob^\star(-\xb^\tp\xib\leq t)\geq1-\alpha\}$ \citep{Shapiro_et_al:2014}, which is finite whenever $\E_{\Prob^\star}|\xi_j|<\infty$ for all $j\in\{1,\dots,n\}$. Thus, we can impose a large upper bound on $|t|$ so that the feasible set is compact.  Moreover, $f(\xb,t,\xib)$ is Lipschitz in $(\xb,t)$ with Lipschitz constant depending on $\norms{\xib}_1$ (see \ref{apdx:num_expt_PO_Lip}). Thus, Assumption~\ref{assumption:loss_and_ambig_set_boundedness} holds if  $\sup_{\Prob\in\calP_N} \E_{\Prob}\norms{\xib}_1<\infty$. This is true for the shape parameters (ambiguity sets) $\calP_N$ used in our experiment: (a) the mean-variance ambiguity set, (b) the 1-Wassestein ambiguity set with $\ell_1$-norm on $\R^n$, and (c) a $\phi$-divergence ball based on total variational distance \citep{Huang_et_al:2021}. Table~\ref{table:portfolio_optimization_reformulations} summarizes these sets and the parameter settings we used for each. It is straightforward to verify that ambiguity sets (a)--(c) satisfy Assumption \ref{assumption:ambig_set_regularity}. In \ref{apdx:num_expt_PO_reform}, we provide tractable reformulations of the TRO model \eqref{eqn:mean-CVaR_model_linear_TRO} under each set. For illustrative purposes, we consider four major indices (S\&P 500, DAX, HSI, and FTSE 1000). As in \cite{Yam_et_al:2016}, we use the multivariate normal distribution as the true distribution of the (monthly) return $\xib$ with mean $\mub$ and covariance $\Sigmab$ given by
$$\mub = \begin{pmatrix}   0.06116 \\ 0.109547 \\ 0.090358 \\ 0.040923\end{pmatrix} \quad \text{and}  \quad \Sigmab = \begin{pmatrix}
0.018632 & 0.020056 & 0.020646 & 0.015213 \\
0.020056 & 0.034507 & 0.027412 & 0.020652 \\
0.020646 & 0.027412 & 0.048680 & 0.021663 \\
0.015213 & 0.020652 & 0.021663 & 0.018791
\end{pmatrix}.$$
These quantities are estimated based on ten-year historical data. Under this distribution,  we can compute the true optimal value $\upsilon^\star=0.0719$, where we set $\beta=0.5$ and $\alpha=0.95$.
\begin{table}[t]\centering\small 
\footnotesize
\caption{Shape parameters for the portfolio optimization problem. \textit{Notation:} $\widehat{\mub}_N$ is the sample mean; $\widehat{\Sigmab}_N$ is the sample covariance matrix; $\Delta_N\subseteq\R^N$ is the probability simplex.} \label{table:portfolio_optimization_reformulations}
\ra{0.8}  
\begin{tabular}{@{}l|c|c@{}} \toprule
& Ambiguity Set $\calP_N$ & Parameters \\ \midrule
(a) Mean-Variance & $\calP_N=\{\Prob\in\calP(\R^n)\mid \E_{\Prob}(\xib)=\mub,\, \Var_{\Prob}(\xib)=\Sigmab\}$ & $(\mub,\Sigmab)=(\widehat{\mub}_N,\widehat{\Sigmab}_N)$ \\ [1ex]
(b) 1-Wasserstein & $\calP_N=\{\Prob\in\calP(\R^n)\mid W_1(\Prob,\Probh_N) \leq r\}$ & $r=0.1$  \\ [1ex]
(c) Total Variation & $\calP_N=\{\pb\in\Delta_N \mid \norms{\pb-\frac{1}{N}\one}_1 \leq \frac{r}{N}\}$ & $r=100$  \\
\bottomrule
\end{tabular}
\end{table}

Let us first analyze the optimal value and solution to the TRO model \eqref{eqn:mean-CVaR_model_linear_TRO} for this problem. Figure~\ref{fig:expt_PO_opt_val_change} illustrates the optimal value for different values of $\theta$. Clearly, the optimal value function is concave, which is consistent with Theorem~\ref{thm:conservatism}. Also, the optimal value function is increasing. This is because sets (a)--(c) are star-shaped with a star center $\Probh_N\in\calP_N$, and thus, the sequence of the TRO ambiguity sets constructed using these sets satisfies the hierarchical properties in Theorem~\ref{thm:calP_nondecreasing}. Figure~\ref{fig:expt_PO_opt_sol_change} illustrates the optimal proportion invested in each asset obtained from solving the TRO model with different values of $\theta\in\{0,0.01,0.02,\dots,1\}$ and $N=500$. We again observe that different choices of the shape parameter $\calP_N$ result in different spectra of optimal solutions. The TRO model with shape parameter (b) suggests investing equally in each asset under larger $\theta$; in particular, the optimal proportion invested in each asset converges to $1/4$ for sufficiently large $\theta$. In contrast, the TRO model with shape parameters (a) and (c) suggests investing more in assets with smaller variances under a larger $\theta$. For example, the optimal proportion invested in asset 1, which has the smallest variance among the four assets, increases with $\theta$. The optimal proportion invested in asset 4, which has a similar variance but a smaller return compared with asset~1, remains zero for most values of~$\theta$. Finally, the optimal proportion invested in assets~2 and 3, which have higher returns and variances compared with asset 1, decreases under a larger $\theta$.
\begin{figure}[t]
    \centering
    \includegraphics[scale=0.6]{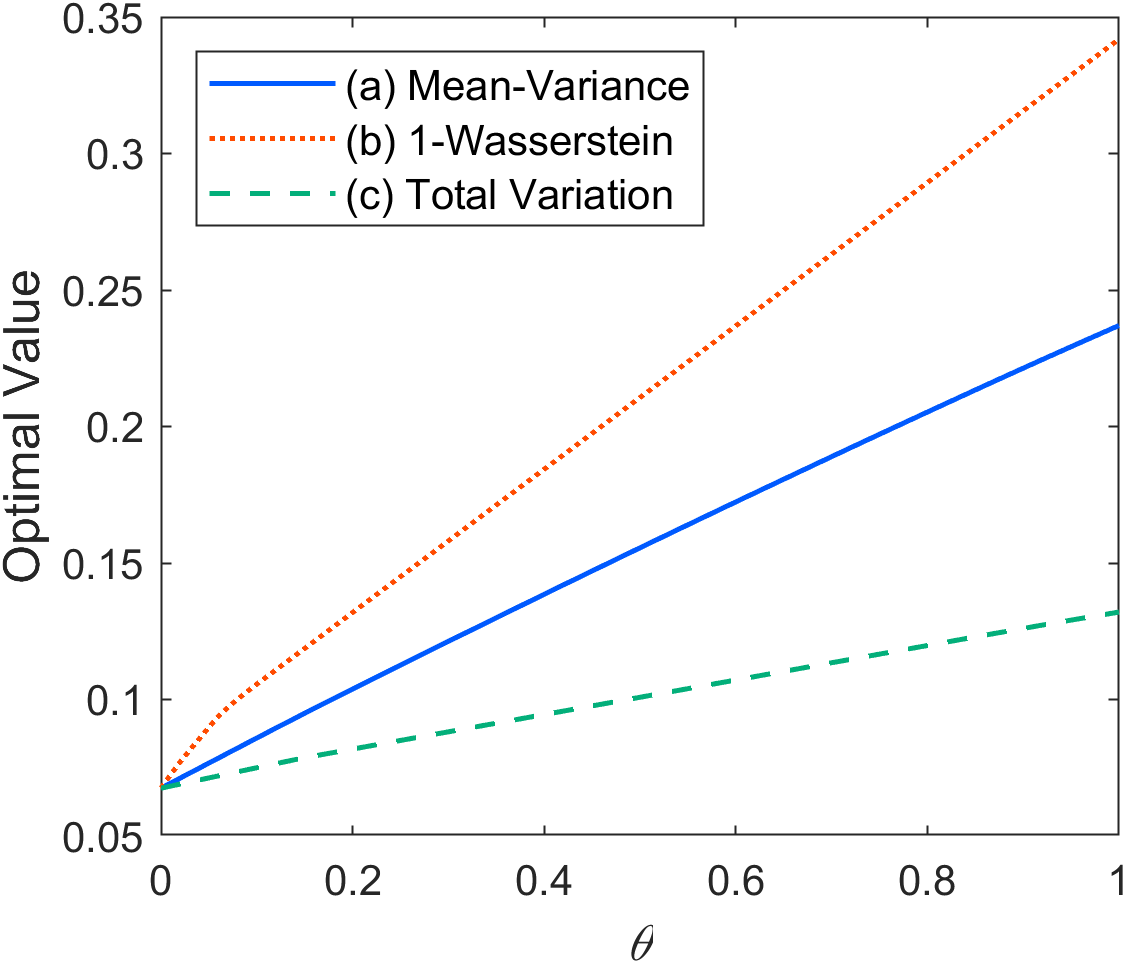}
    \caption{Optimal value for different values of $\theta$ in the portfolio optimization problem.}
    \label{fig:expt_PO_opt_val_change}
\end{figure}
\begin{figure}[t]
    \centering
    \includegraphics[scale=0.65]{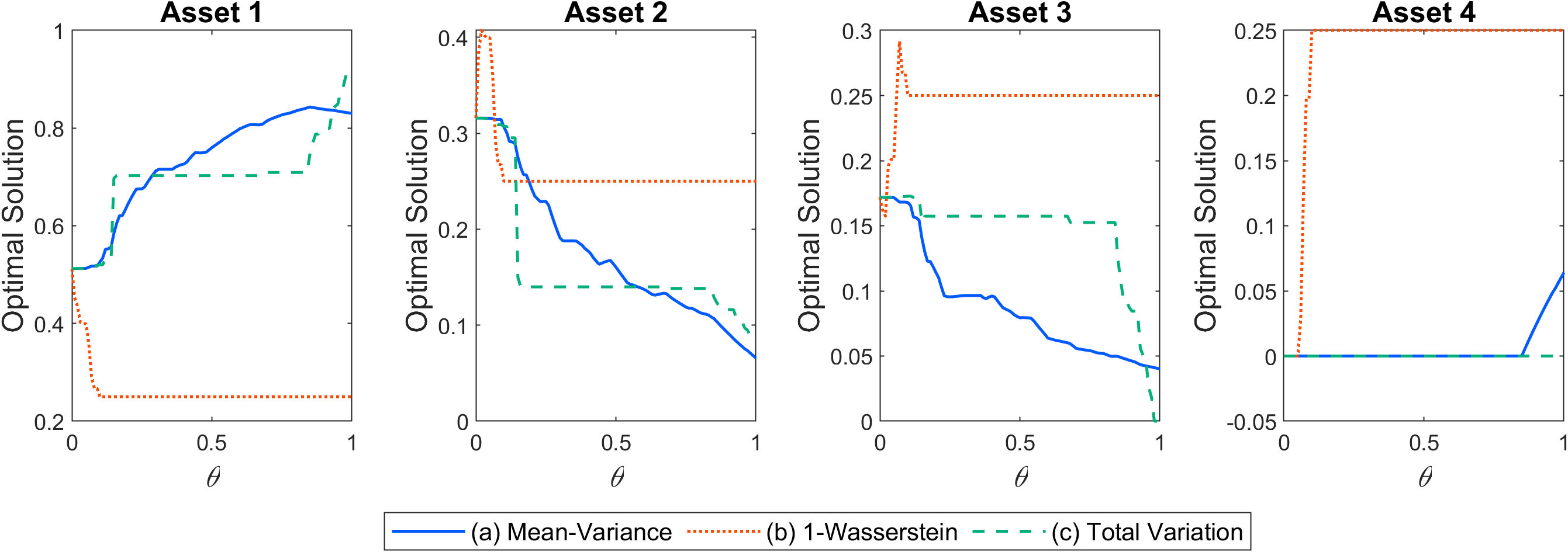}
    \caption{Optimal solution for different values of $\theta$ in the portfolio optimization problem.}
    \label{fig:expt_PO_opt_sol_change}
\end{figure}

As shown in \ref{apdx:add_expt_results:distance_based_DRO}, the spectra obtained from our TRO model with a TRO ambiguity set characterized by a distance-based shape parameter and the DRO model equipped with that shape parameter could be different for some choices of the statistical distance in the distance-based shape parameter. For example, the spectra of solutions obtained from the TRO model employing the total variation ambiguity set as the shape parameter in the TRO ambiguity set and the DRO model with the total variation ambiguity set are different. However, the two spectra using 1-Wasserstein ambiguity set are approximately the same. Moreover, similar to the inventory control problem, we observe that the conservatism of the TRO optimal solution increases with $\theta$; see \ref{apdx:add_expt_results:conservatism}. However, as demonstrated in \ref{apdx:add_expt_results:out_of_sample}, adopting solutions on the spectrum of TRO optimal solutions can yield lower out-of-sample costs than the SAA and DRO solutions.

Next, we analyze the bias and standard deviation of the TRO estimator $\upsilonh_N(\theta)$ presented in Figure~\ref{fig:expt_PO_mean_variance}. We estimate these quantities as discussed in Section~\ref{subsec:inventory_control_problem} with $N=10$.  Again, the SAA estimator $\upsilonh_N(0)$ exhibits a downward bias. We also observe that our TRO estimator $\upsilonh_N(\theta)$ is an unbiased estimator for some $\theta$. Moreover, the absolute bias of $\upsilonh_N(\theta)$ with sufficiently small $\theta>0$ is smaller than that of $\upsilonh_N(0)$. For example, when using set (c), the absolute bias of $\upsilonh_N(\theta)$ is smaller for $\theta\in(0,1]$. These results are consistent with Theorem~\ref{thm:debias_theta} and Corollary~\ref{cor:debias_theta}. In contrast to the inventory optimization problem, where the standard deviation of $\upsilonh_N(\theta)$ decreases with $\theta$, we observe that the standard deviation may increase or decrease, depending on the choice of the shape parameter $\calP_N$. With shape parameter (c), the TRO estimator $\upsilonh_N(\theta)$ has a larger standard deviation compared with the SAA estimator $\upsilonh_N(0)$ for $\theta\in(0,1]$. In contrast, with shape parameter (b), $\upsilonh_N(\theta)$ has a smaller standard deviation than $\upsilonh_N(0)$ for $\theta\in(0,1]$. Finally, with shape parameter (a), the standard deviation of $\upsilonh_N(\theta)$ is smaller than that of $\upsilonh_N(0)$ for $\theta\in(0,0.21]$. These results show that for this problem, our TRO model could produce estimators with a smaller bias when the TRO ambiguity set is constructed using shape parameters (a)--(c) and a smaller standard deviation when constructed using shape parameters (a) for small $\theta$ and (b) for $\theta\in(0,1]$. Again, we observe that none of the shape parameters consistently produce a TRO estimator with the best bias-variance trade-off; see \ref{apdx:add_expt_results:bias_variance} for further discussions.
\begin{figure}
    \centering
    \includegraphics[scale=0.7]{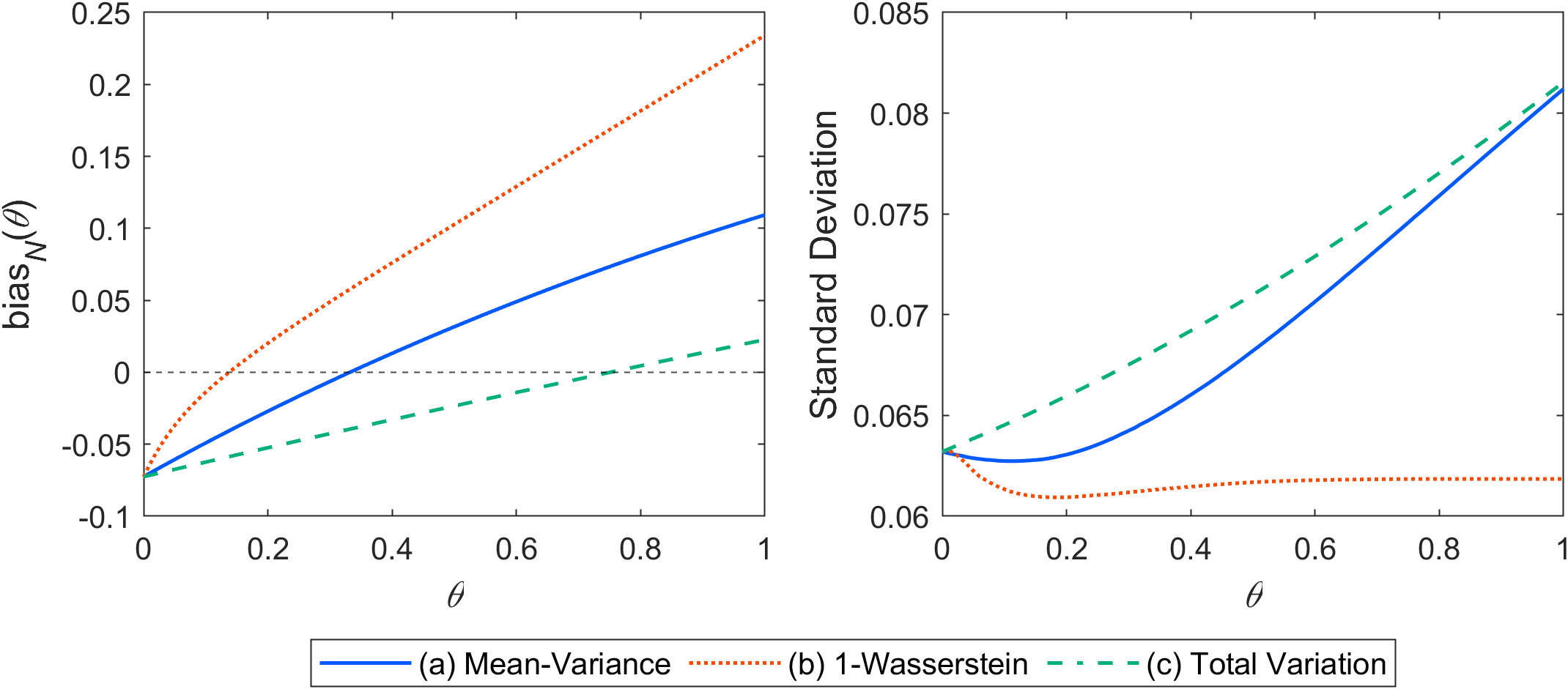}
    \caption{Bias and standard deviation of $\upsilonh_N(\theta)$ for different values of $\theta$ with $N=10$ in the portfolio optimization problem.}
    \label{fig:expt_PO_mean_variance}
\end{figure}

Finally, we demonstrate the asymptotic properties of $\upsilonh_N(\theta_N)$. First, we compute the absolute difference $|\upsilonh_N(\theta_N)-\upsilon^\star|$ between the optimal value of the TRO model $\upsilonh_N(\theta_N)$ and the true optimal value $\upsilon^\star$ as discussed in Section~\ref{subsec:inventory_control_problem}. Table~\ref{table:expt_PO_absolute_diff} presents $|\upsilonh_N(\theta_N)-\upsilon^\star|$ for $N\in\{10,50,100,500,1000\}$. Clearly, $|\upsilonh_N(\theta_N)-\upsilon^\star|$  decreases with $N$, illustrating the asymptotic convergence of $\upsilonh_N(\theta_N)$ to $\upsilon^\star$ as shown in Theorem~\ref{thm:asymptotic_convergence}. Second, to demonstrate the asymptotic distribution of $\upsilonh_N(\theta_N)$, we compute the KS statistic between the standard normal distribution and the empirical distribution of $\sqrt{N}\big[\upsilonh_N(\theta)-\upsilon^\star\big]/V^\star$ as discussed in Section~\ref{subsec:inventory_control_problem}.  Table~\ref{table:expt_PO_KS_stat} presents the KS statistics for $N\in\{10,100,1000\}$. The KS statistic converges to zero when $N$ increases, suggesting that $\sqrt{N}\big[\upsilonh_N(\theta)-\upsilon^\star\big]/V^\star$ converges weakly to the standard normal distribution. These observations are consistent with Theorem~\ref{thm:asy_dist}.
\begin{table}[t]\centering
\footnotesize
\caption{Absolute difference between $\upsilonh_N(\theta_N)$ and $\upsilon^\star$ for different $N$ in the portfolio optimization problem.} \label{table:expt_PO_absolute_diff}
\ra{0.8}  
\begin{tabular}{@{}l|rrrrr@{}} \toprule
                      & $N = 10$ & $N = 50$ & $N = 100$ & $N = 500$ & $N = 1000$ \\ \midrule
(a) Mean-Variance     & 0.0953   & 0.0391   & 0.0239    & 0.0059    & 0.0002     \\
(b) 1-Wasserstein     & 0.2792   & 0.0117   & 0.0422    & 0.0010    & 0.0035     \\
(c) Total Variation   & 0.0699   & 0.0442   & 0.0204    & 0.0070    & 0.0007     \\
\bottomrule
\end{tabular}
\end{table}
\begin{table}[t]\centering
\footnotesize
\caption{Kolmogorov–Smirnov statistics for $\sqrt{N}\big[\upsilonh_N(\theta)-\upsilon^\star\big]/V^\star$ with $V^\star=\Var_{\Prob^\star}\big(f(\xb^\star,\xib)\big)$ for different $N$ in the portfolio optimization problem.} \label{table:expt_PO_KS_stat}
\ra{0.8}  
\begin{tabular}{@{}l|rrr@{}} \toprule
                    & $N = 10$ & $N = 100$ & $N = 1000$ \\ \midrule
(a) Mean-Variance   & 0.5040 & 0.1725 & 0.0683 \\
(b) 1-Wasserstein   & 0.9275 & 0.5103 & 0.2674 \\
(c) Total Variation & 0.1144 & 0.0995 & 0.0319 \\
\bottomrule
\end{tabular}
\end{table}

\section{Conclusion} \label{sec:conclusion}

In this paper, we propose and analyze a new TRO approach for modeling uncertainty in optimization problems that serves as a middle ground between the optimistic approach that adopts a distributional belief and the pessimistic approach that protects against distributional ambiguity. We equip the TRO model with a TRO ambiguity set $\calP'_{N,\theta}$ characterized by a size parameter $\theta$ controlling the level of optimism and a shape parameter  $\calP_N$ representing distributional ambiguity, which could be any ambiguity set satisfying mild assumptions. Our theoretical investigations and results include the following. First, we investigate the conservatism of the TRO model by analyzing the properties of the TRO ambiguity set and the characteristics of the model's optimal value and solutions. We derive necessary and sufficient conditions for $\calP'_{N,\theta}$ to satisfy the hierarchical property and quantify the difference in the optimal value and the set of optimal solutions (and hence conservatism) incurred by perturbation in $\theta$. Moreover, the TRO model enables decision-makers to explore a spectrum of optimal solutions, from optimistic to conservative solutions.  Second, we show that our TRO model could produce an unbiased estimator of the true optimal value. Additionally, we derive an upper bound on the generalization error (GE) of the TRO estimator and discuss how this bound might be tighter than the GE bounds for the SAA or DRO models. Moreover, we show that the GE of our TRO model has an exponentially decaying tail under specific choices of the shape parameter. Finally, we prove the almost sure convergence of the optimal value and the set of optimal solutions of the TRO model to their true counterparts. In addition, we derive the asymptotic distribution of the optimal value to the TRO model. These properties hold for TRO models with TRO ambiguity sets constructed using general shape parameters, such as moment- and distance-based ambiguity sets. We numerically demonstrate these theoretical results using an inventory control problem and a portfolio optimization problem.

Our work opens avenues for further research in various directions. These include extending and identifying properties of the TRO approach for risk-averse settings. Moreover, developing computationally efficient algorithms aimed at obtaining the complete spectrum of optimal solutions of the TRO model would be of significant value in many application domains. Another interesting area is investigating the characteristics of the spectrum of TRO optimal solutions of the TRO model under different shape parameters. 

\appendix
\newpage

\section{Proofs} \label{apdx:math_proof}

\subsection{Proof of Theorem~\ref{thm:calP_nondecreasing}}

\begin{proof}
First, we prove part (i). Suppose that $\calP_N$ is star-shaped and $\Probh_N\in\calP_N$ is a star center of $\calP_N$. We show that $\{\calP_{N,\theta}\mid\theta\in[0,1]\}$ satisfies the hierarchical property. For any $0\leq \theta_1<\theta_2\leq 1$ and $\Q\in\calP_N$, we have
\begin{align}
  (1-\theta_1)\, \Probh_N + \theta_1\,\Q  &= (1-\theta_1)\, \Probh_N + \theta_2\,\Probh_N - \theta_2\,\Probh_N + \theta_1\,\Q \nonumber \\
  &=   (1-\theta_2)\, \Probh_N + (\theta_2-\theta_1)\,\Probh_N + \theta_1\,\Q   \nonumber\\
  &=  (1-\theta_2)\Probh_N + \theta_2\Bigg[ \Bigg(1-\frac{\theta_1}{\theta_2}\Bigg)\Probh_N+\frac{\theta_1}{\theta_2}\Q\Bigg]  \label{eqn_pf:prop:calP_nondecreasing_1},
\end{align}
where $\theta_1/\theta_2\in[0,1)$. Since $\calP_N$ is star-shaped and $\Probh_N\in\calP_N$ is a star center of $\calP_N$, the measure $(1-\theta_1/\theta_2)\Probh_N+(\theta_1/\theta_2)\Q$ in the second term of \eqref{eqn_pf:prop:calP_nondecreasing_1} belongs to $\calP_N$. It follows from the definition of $\calP'_{N,\theta_2}$ that $(1-\theta_1)\, \Probh_N + \theta_1\,\Q\in\calP'_{N,\theta_2}$ for any $\Q\in\calP_N$, and thus, $\calP'_{N,\theta_1}\subseteq\calP'_{N,\theta_2}$. This shows that $\{\calP_{N,\theta}\mid\theta\in[0,1]\}$ satisfies the hierarchical property.

Now, suppose that $\{\calP_{N,\theta}\mid\theta\in[0,1]\}$ satisfies the hierarchical property, i.e., $\calP'_{N,\theta_1}\subseteq\calP'_{N,\theta_2}$ for all $0\leq\theta_1<\theta_2\leq 1$. We show that $\calP_N$ is star-shaped with a star center $\Probh_N$. For any $\Q\in\calP_N$, by \eqref{eqn_pf:prop:calP_nondecreasing_1}, we have that
\begin{equation} \label{eqn_pf:prop:calP_nondecreasing_2}
    (1-\theta_1)\, \Probh_N + \theta_1\,\Q = (1-\theta_2)\Probh_N + \theta_2\Bigg[ \Bigg(1-\frac{\theta_1}{\theta_2}\Bigg)\Probh_N+\frac{\theta_1}{\theta_2}\Q\Bigg] \in\calP'_{N,\theta_2}
\end{equation}
since $(1-\theta_1)\, \Probh_N + \theta_1\,\Q\in\calP'_{N,\theta_1}\subseteq\calP'_{N,\theta_2}$. By definition of $\calP'_{N,\theta_2}$, the inclusion in \eqref{eqn_pf:prop:calP_nondecreasing_2} implies that $(1-\theta_1/\theta_2)\Probh_N+(\theta_1/\theta_2)\Q\in\calP_N$. Since $\theta_1/\theta_2\in[0,1)$ is arbitrary, we have $(1-\alpha)\Probh_N+\alpha\Q\in\calP_N$ for all $\alpha\in[0,1]$ and $\Q\in\calP_N$. (Note that when $\alpha=1$, we have $(1-\alpha)\Probh_N+\alpha\Q=\Q\in\calP_N$.) This shows that $\calP_N$ is star-shaped and $\Probh_N$ is a star center.

Next, we prove part (ii). Suppose that $\{\calP_{N,\theta}\mid\theta\in[0,1]\}$ satisfies the strict hierarchical property, i.e., $\calP'_{N,\theta}$ is increasing in $\theta$. We show that $\calP_N$ is star-shaped with a star center $\Probh_N$ and $\calP_N\ne\{\Probh_N\}$. From part (i), we have that $\calP_N$ is star-shaped with a star center $\Probh_N$. Suppose, on the contrary, that $\calP_N=\{\Prob_N\}$. Then,  we have $\Prob'_{N,\theta}=\{\Probh_N\}$ for all $\theta\in[0,1]$, contradicting that $\calP'_{N,\theta}$ is increasing in $\theta$. This shows that $\calP_N\ne\{\Prob_N\}$.

Now, suppose that $\calP_N$ is star-shaped with a star center $\Probh_N\in\calP_N$ and $\calP_N\ne\{\Probh_N\}$. We show  that $\{\calP_{N,\theta}\mid\theta\in[0,1]\}$ satisfies the strict hierarchical property. From part (i),  we have $\calP'_{N,\theta_1}\subseteq\calP'_{N,\theta_2}$, for any $0\leq \theta_1<\theta_2\leq 1$. To show the strict inclusion $\calP'_{N,\theta_1}\subset\calP'_{N,\theta_2}$, we need to show that there exists a probability measure $\M$ such that $\M\in\calP'_{N,\theta_2}$ but $\M\not\in\calP'_{N,\theta_1}$. We first claim that this is equivalent to
\begin{equation}\label{eqn:pf_calP_increasing1}
    \calP_N\setminus\Bigg\{\Bigg(1-\frac{\theta_1}{\theta_2}\Bigg)\,\Probh_N + \frac{\theta_1}{\theta_2}\,\Q \,\Bigg|\, \Q\in\calP_N \Bigg\}\ne\emptyset.
\end{equation}
Indeed, \eqref{eqn:pf_calP_increasing1} is equivalent to the existence of $\Q'\in\calP_N$ such that $\Q'\ne (1-\theta_1/\theta_2)\Probh_N+(\theta_1/\theta_2)\Q$ for all $\Q\in\calP_N$. Note that $\Q'\ne (1-\theta_1/\theta_2)\Probh_N+(\theta_1/\theta_2)\Q$ for all $\Q\in\calP_N$ is equivalent to $\theta_2\Q' + (1-\theta_2)\Probh_N \ne \theta_1\Q + (1-\theta_1)\Probh_N$ for all $\Q\in\calP_N$. Thus, we have that $\theta_2\Q'+(1-\theta_2)\Probh_N\in\calP'_{N,\theta_2}$ cannot be expressed as $\theta_1\Q + (1-\theta_1)\Probh_N$ for any $\Q\in\calP_N$, i.e., $\theta_2\Q'+(1-\theta_2)\Probh_N\not\in\calP'_{N,\theta_1}$. This completes the proof of the claim.

Next, suppose, for the sake of contradiction, that condition \eqref{eqn:pf_calP_increasing1} does not hold, i.e., 
\begin{equation}\label{eqn:pf_calP_increasing2}
    \calP_N\subseteq \Bigg\{\Bigg(1-\frac{\theta_1}{\theta_2}\Bigg)\,\Probh_N + \frac{\theta_1}{\theta_2}\,\Q \,\Bigg|\, \Q\in\calP_N \Bigg\}.
\end{equation}
We claim that \eqref{eqn:pf_calP_increasing2} is equivalent to $\calP_N=\{\Probh_N\}$ is a singleton, which contradicts with $\calP_N\ne\{\Probh_N\}$. Consider an arbitrary probability measure $\M\in\calP_N$. By \eqref{eqn:pf_calP_increasing2}, we can write $\M=(1-\theta_1/\theta_2)\Probh_N+(\theta_1/\theta_2)\Q_1$ for some $\Q_1\in\calP_N$. Moreover, note that for any $i\in\N$ and probability measure $\Q_i\in\calP_N$, by \eqref{eqn:pf_calP_increasing2}, we can write $\Q_i=(1-\theta_1/\theta_2)\Probh_N+(\theta_1/\theta_2)\Q_{i+1}$ for some $\Q_{i+1}\in\calP_N$. Using this recursion, we have
\begin{align*}
    \M = \Bigg(1-\frac{\theta_1}{\theta_2}\Bigg)\,\Probh_N + \frac{\theta_1}{\theta_2}\,\Q_1 
    &= \Bigg(1-\frac{\theta_1}{\theta_2}\Bigg)\,\Probh_N\cdot \Bigg(1+\frac{\theta_1}{\theta_2}\Bigg) + \Bigg(\frac{\theta_1}{\theta_2}\Bigg)^2\,\Q_2 \\
    &= \cdots \\
    &=  \Bigg(1-\frac{\theta_1}{\theta_2}\Bigg)\,\Probh_N\cdot \sum_{i=0}^{n-1} \Bigg(\frac{\theta_1}{\theta_2}\Bigg)^i + \Bigg(\frac{\theta_1}{\theta_2}\Bigg)^n\,\Q_n
\end{align*}
for any $n\in\N$. Since $\theta_1<\theta_2$, $\sum_{i=0}^{n-1} (\theta_1/\theta_2)^i \rightarrow 1/(1-\theta_1/\theta_2)$ and $(\theta_1/\theta_2)^n\rightarrow 0$ as $n\rightarrow\infty$. Thus, for any $\varepsilon>0$, there exists $n'\in\N$ such that $|(1-\theta_1/\theta_2) \sum_{i=0}^{n-1} (\theta_1/\theta_2)^i - 1|<\varepsilon/2$ and $|(\theta_1/\theta_2)^n|<\varepsilon/2$ for all $n>n'$. Therefore, for any $B\in\calB$, we have that
\begin{align*}
    |\M(B)-\Probh_N(B)| &= \Bigg|\Bigg(1-\frac{\theta_1}{\theta_2}\Bigg)\,\Probh_N(B)\cdot \sum_{i=0}^{n-1} \Bigg(\frac{\theta_1}{\theta_2}\Bigg)^i + \Bigg(\frac{\theta_1}{\theta_2}\Bigg)^n\,\Q_n(B) - \Probh_N(B) \Bigg|\\
    &\leq  \Bigg|\Bigg(1-\frac{\theta_1}{\theta_2}\Bigg) \sum_{i=0}^{n-1} \Bigg(\frac{\theta_1}{\theta_2}\Bigg)^i - 1 \Bigg| \cdot\Probh_N(B) + \Bigg|  \Bigg(\frac{\theta_1}{\theta_2}\Bigg)^n \Bigg| \cdot\Q_n(B) \\
    &\leq \frac{\varepsilon}{2} \,\Probh_N(B) + \frac{\varepsilon}{2} \,\Q_n(B) \\
    &\leq \varepsilon
\end{align*}
for all $n>n'$, where the first inequality follows from triangular inequality. Since $\varepsilon>0$ is arbitrary, it follows that $\M=\Probh_N$ for any $\M\in\calP_N$, implying that $\calP_N=\{\Probh_N\}$. This contradicts the assumption that $\calP_N\ne\{\Probh_N\}$ and completes the proof.  
\end{proof}

\subsection{Relationship between Star-Shapedness and Convexity} \label{apdx:lem:star_shaped_convex}

In Lemma~\ref{lem:star_shaped_convex}, we show that convexity of a set implies star-shapedness.

\begin{lemma} \label{lem:star_shaped_convex}
If $\calP_N$ is convex and $\Probh_N\in\calP_N$, then $\calP_N$ is star-shaped with a star center $\Probh_N$
\end{lemma}

\begin{proof}
Since $\Probh_N  \in\calP_N$, it follows from the convexity of $\calP_N$ that
$$(1-\alpha)\Probh_N+\alpha\Prob\in\calP_N,\quad\forall\alpha\in[0,1],\,\Prob\in\calP_N.$$
This shows that $\calP_N$ is star-shaped with a star center $\Probh_N$. 
\end{proof}

\subsection{Proof of Proposition~\ref{prop:star_shape_moment}}

\begin{proof}
Suppose that $\calK_i$ is star-shaped on $\calS_i$ with a star center $\E_{\Probh_N}[\Phi_i(\xib)]$ for all $i\in\{1,\dots,p\}$. We show that $\calP_N$ is star-shaped with a star center $\Probh_N$. Indeed, for all $i\in\{1,\dots,p\}$, by the star-shapedness of $\calK_i$, we have $(1-\alpha)\E_{\Probh_N}[\Phi_i(\xib)] + \alpha\E_{\Q}[\Phi_i(\xib)] =\E_{(1-\alpha)\Probh_N+\alpha\Q}[\Phi_i(\xib)]\in\calK_i$ for any $\alpha\in[0,1]$ and $\Q\in\calP_N$. It follows from the definition of $\calP_N$ that $(1-\alpha)\Probh_N+\alpha\Q\in\calP_N$ for any $\alpha\in[0,1]$ and $\Q\in\calP_N$, i.e., $\calP_N$ is star-shaped with a star center $\Probh_N$.

Now, suppose that $\calP_N$ is star-shaped with a star center $\Probh_N$. We show that $\calK_i$ is star-shaped on $\calS_i:=\{\E_{\Q}[\Phi_i(\xib)]\mid \Q\in\calP_N\}\subseteq\R^{d_i\times d_i}$ with a star center $\E_{\Probh_N}[\Phi_i(\xib)]\in\calS_i$ for all $i\in\{1,\dots,p\}$. Since $(1-\alpha)\Probh_N + \alpha\Q\in\calP_N$ for any $\alpha\in[0,1]$ and $\Q\in\calP_N$, we have $\E_{(1-\alpha)\Probh_N+\alpha\Q}[\Phi_i(\xib)]=(1-\alpha)\E_{\Probh_N}[\Phi_i(\xib)] + \alpha\E_{\Q}[\Phi_i(\xib)] \in\calK_i$ for all $i\in\{1,\dots,p\}$. This, in turn, implies that  $(1-\alpha)\E_{\Probh_N}[\Phi_i(\xib)] + \alpha\Psi \in\calK_i$ for any $\alpha\in[0,1]$ and $\Psi\in\calS_i$. This completes the proof. 
\end{proof}

\subsection{Proof of Proposition~\ref{prop:star_shape_distance}}

\begin{proof}
Suppose that $\sfd$ is quasi-convex about $\Probh_N$ in the first argument. We show that $\calP_N(\varepsilon)$ is star-shaped with a star center $\Probh_N$ for all $\varepsilon\geq 0$. Indeed, for any given $\varepsilon\geq 0$, by quasi-convexity, we have $\sfd\big((1-\alpha)\Probh_N+\alpha\Q,\Probh_N\big) \leq \sfd(\Q,\Probh_N)\leq \varepsilon$ for any $\alpha\in[0,1]$ and $\Q\in\calP_N(\varepsilon)$. It follows from the definition of $\calP_N(\varepsilon)$ that $(1-\alpha)\Probh_N+\alpha\Q\in\calP_N(\varepsilon)$ for any $\alpha\in[0,1]$ and $\Q\in\calP_N(\varepsilon)$, i.e., $\calP_N(\varepsilon)$ is star-shaped with a star-center $\Probh_N$.

Now, suppose that $\calP_N(\varepsilon)$ is star-shaped with a star center $\Probh_N$ for all $\varepsilon\geq 0$. We show that $\sfd$ is quasi-convex about $\Probh_N$. Since $\calP_N(\varepsilon)$ is star-shaped, we have $(1-\alpha)\Probh_N+\alpha\Q\in\calP_N(\varepsilon)$ for any $\alpha\in[0,1]$ and $\Q\in\calP_N(\varepsilon)$. It follows from the definition of $\calP_N(\varepsilon)$ that $\sfd\big((1-\alpha)\Probh_N+\alpha\Q,\Probh_N\big) \leq \varepsilon$ for any $\alpha\in[0,1]$ and $\Q\in\calP_N(\varepsilon)$. Suppose, for the sake of contradiction, that $\sfd$ is not quasi-convex about $\Probh_N$ in the first argument. That is, there exist $\alpha\in(0,1)$ and $\Q\in\calP(\Xi)$ such that $\sfd\big((1-\alpha)\Probh_N+\alpha\Q,\Probh_N\big) >\sfd(\Q,\Probh_N)$. Let $\bar{\varepsilon}:=\sfd(\Q,\Probh_N)\in[0,\infty)$. Since $\calP_N(\bar{\varepsilon})$ is star-shaped with a star center $\Probh_N$ and $\Q\in\calP_N(\bar{\varepsilon})$, we have $(1-\alpha)\Probh_N+\alpha\Q\in\calP_N(\bar{\varepsilon})$. Thus, $\sfd\big((1-\alpha)\Probh_N+\alpha\Q,\Probh_N\big)\leq \bar{\varepsilon} = \sfd(\Q,\Probh_N)$, contradicting the assumption that $\sfd$ is not quasi-convex. This completes the proof.  
\end{proof}

\subsection{Proof of Theorem~\ref{thm:sensitivity_in_theta}}

\begin{proof}
First, we show that for any $\{\theta_1,\theta_2\}\subset[0,1]$, we have
\begin{equation} \label{eqn:Hausdorff_Lip_cont_in_theta}
    \bbmH(\calP'_{N,\theta_1},\calP'_{N,\theta_2}) \leq  2C_N |\theta_1-\theta_2|.
\end{equation}
For any $\Prob_1\in\calP'_{N,\theta_1}$ and $\Prob_2\in\calP'_{N,\theta_2}$, we can write $\Prob_i=(1-\theta_i)\Probh_N + \theta_i\Q_i$ for some $\Q_i\in\calP_N$ and $i\in\{1,2\}$. Then,
\allowdisplaybreaks
\begin{align}
    \Big|\E_{\Prob_1}[f(\xb,\xib)] - \E_{\Prob_2}[f(\xb,\xib)]\Big|
    & = \Bigg| (\theta_2-\theta_1)\cdot\frac{1}{N}\sum_{i=1}^N f(\xb,\xibh_i) + \theta_1\cdot \E_{\Q_1}[f(\xb,\xib)] - \theta_2\cdot\E_{\Q_2}[f(\xb,\xib)] \Bigg| \nonumber \\
    &\leq |\theta_1-\theta_2| \Bigg|\frac{1}{N}\sum_{i=1}^N f(\xb,\xibh_i) \Bigg|  + \Big| \theta_1 \E_{\Q_1}[f(\xb,\xib)] - \theta_2 \E_{\Q_2}[f(\xb,\xib)] \Big| \nonumber \\
    &\leq C_N|\theta_1-\theta_2| + \Big| \theta_1 \E_{\Q_1}[f(\xb,\xib)] - \theta_2 \E_{\Q_2}[f(\xb,\xib)] \Big|, \label{eqn:pf_prop_Lip_cont_in_theta_5}
\end{align}
where the last inequality follows from Assumption \ref{assumption:loss_and_ambig_set_boundedness} and the definition of $C_N$ in \eqref{eqn:def_of_CN}. Note that
\begin{subequations}
\begin{align}
    \sup_{\Prob_1\in\calP'_{N,\theta_1}} \inf_{\Prob_2\in\calP'_{N,\theta_2}} \bbmd(\Prob_1,\Prob_2) 
    &=\sup_{\Prob_1\in\calP'_{N,\theta_1}} \inf_{\Prob_2\in\calP'_{N,\theta_2}} \sup_{\xb\in\calX} \Big|\E_{\Prob_1}[f(\xb,\xib)] - \E_{\Prob_2}[f(\xb,\xib)]\Big|  \nonumber\\
    &\leq  C_N|\theta_1-\theta_2| + \sup_{\Q_1\in\calP_N} \inf_{\Q_2\in\calP_N} \sup_{\xb\in\calX}  \Big| \theta_1 \E_{\Q_1}[f(\xb,\xib)] - \theta_2 \E_{\Q_2}[f(\xb,\xib)] \Big| \label{eqn:pf_prop_Lip_cont_in_theta_6}\\
    &\leq C_N|\theta_1-\theta_2| + \sup_{\Q_1\in\calP_N} \sup_{\xb\in\calX}  \Big| \theta_1 \E_{\Q_1}[f(\xb,\xib)] - \theta_2 \E_{\Q_1}[f(\xb,\xib)] \Big| \label{eqn:pf_prop_Lip_cont_in_theta_7}\\
    &\leq C_N|\theta_1-\theta_2| + \sup_{\Q_1\in\calP_N} \sup_{\xb\in\calX}  |\theta_1-\theta_2|\big|\E_{\Q_1}[f(\xb,\xib)]\big| \nonumber \\
    & \leq 2C_N|\theta_1-\theta_2|. \label{eqn:pf_prop_Lip_cont_in_theta_8}
\end{align}
\end{subequations}
Inequality \eqref{eqn:pf_prop_Lip_cont_in_theta_6} follows from \eqref{eqn:pf_prop_Lip_cont_in_theta_5}, where we note that the arguments in supremum and infimum are changed to $\Q_1\in\calP_N$ and $\Q_2\in\calP_N$, respectively; inequality \eqref{eqn:pf_prop_Lip_cont_in_theta_7} follows from the fact that choosing $\Q_2=\Q_1$ for the infimum problem yields an upper bound; inequality \eqref{eqn:pf_prop_Lip_cont_in_theta_8} follows from Assumption~\ref{assumption:loss_and_ambig_set_boundedness} and the definition of $C_N$ in \eqref{eqn:def_of_CN}. Interchanging the role of $\Prob_1$ and $\Prob_2$, as well as $\calP'_{N,\theta_1}$ and $\calP'_{N,\theta_2}$, we can obtain a similar inequality:
$$\sup_{\Prob_2\in\calP'_{N,\theta_2}} \inf_{\Prob_1\in\calP'_{N,\theta_1}} \bbmd(\Prob_1,\Prob_2) \leq 2C_N|\theta_1-\theta_2|.$$
Therefore, from the definition of $\bbmH$ in \eqref{def:Hausdorff_distance}, we have
$$\bbmH(\calP'_{N,\theta_1},\calP'_{N,\theta_2}) = \max\Bigg\{\sup_{\Prob_1\in\calP'_{N,\theta_1}} \inf_{\Prob_2\in\calP'_{N,\theta_2}} \bbmd(\Prob_1,\Prob_2) ,\,  \sup_{\Prob_2\in\calP'_{N,\theta_2}} \inf_{\Prob_1\in\calP'_{N,\theta_1}} \bbmd(\Prob_1,\Prob_2)\Bigg\} \leq  2C_N|\theta_1-\theta_2|.$$
Now, we prove assertion (i):
$$|\upsilonh_N(\theta_1)-\upsilonh_N(\theta_2)|\leq \bbmH(\calP'_{N,\theta_1},\calP'_{N,\theta_2}) \leq  2C_N |\theta_1-\theta_2|,$$
where the first inequality follows from the quantitative stability analysis; see Proposition~\ref{prop:known_QSA_DRO} in \ref{apdx:known_QSA_DRO}.

For (ii), under the second-order growth condition, we have 
$$D\Big(\calXh_N(\theta_2),\calXh_N(\theta_1)\Big)\leq \sqrt{\frac{3}{\tau}\bbmH(\calP'_{N,\theta_1},\calP'_{N,\theta_2}) } \leq \sqrt{\frac{6C_N}{\tau}\, |\theta_1-\theta_2|},$$
where the first inequality follows again from the quantitative stability analysis; see Proposition \ref{prop:known_QSA_DRO} in \ref{apdx:known_QSA_DRO}.
This completes the proof.  
\end{proof}

\subsection{Proof of Theorem~\ref{thm:conservatism}}

\begin{proof}

First, we prove part (i). To show the concavity of $\upsilonh_N(\theta)$, we define the function $g(\theta;\xb):=\E_{\Probh_N}[f(\xb,\xib)]+\theta\big\{\sup_{\Prob\in\calP_N}\E_{\Prob}[f(\xb,\xib)]-\E_{\Probh_N}[f(\xb,\xib)]\big\}$, which is linear in $\theta$ for any $\xb\in\calX$. Since $\upsilonh_N(\theta)$ is the pointwise minimum of the linear functions $\{g(\theta;\xb)\mid \xb\in\calX\}$, i.e., $\upsilonh_N(\theta)=\min_{\xb\in\calX} g(\theta;\xb)$, the function $\upsilonh_N(\theta)$ is concave. Next, to show \eqref{eqn:expansion_in_theta}, note that $\theta=(1-\theta)\cdot0 +\theta\cdot 1$. Thus, concavity of $\upsilonh_N$ implies
$$\rh_N(\theta) = \upsilonh_N(\theta) - \Big[(1-\theta)\cdot\upsilonh_N(0) + \theta\cdot\upsilonh_N(1)\Big] \geq 0.$$
Also, by Theorem \ref{thm:sensitivity_in_theta}, we have
\begin{align*}
    \rh_N(\theta) &= \upsilonh_N(\theta) - \Big[(1-\theta)\cdot\upsilonh_N(0) + \theta\cdot\upsilonh_N(1)\Big]  \\
    &\leq (1-\theta)\cdot \Big|\upsilonh_N(\theta)-\upsilonh_N(0)\big| + \theta\cdot \Big|\upsilonh_N(\theta)-\upsilonh_N(1)\big| \\
    &\leq (1-\theta)\cdot 2C_N\theta + \theta\cdot 2C_N(1-\theta)\\
    &\leq 4C_N\theta(1-\theta).
\end{align*}
This completes the proof of part (i).

Finally, for part (ii), we have
\allowdisplaybreaks
\begin{subequations}
\begin{align}
    &\quad\,\,D\Big(\calXh_N(\theta),(1-\theta)\calXh_N(0)+\theta\calXh_N(1)\Big) \nonumber \\
    &=\sup_{\xb\in\calXh_N(\theta)} \inf_{\xbbar\in (1-\theta)\calXh_N(0)+\theta\calXh_N(1)} \norms{\xb-\xbbar} \nonumber \\
    &=\sup_{\xb\in\calXh_N(\theta)} \inf_{\xb'\in\calXh_N(0),\,\xb''\in\calXh_N(1)} \norm{\xb-\big[(1-\theta)\xb'+\theta\xb''\big]}  \nonumber\\
    &\leq \sup_{\xb\in\calXh_N(\theta)} \inf_{\xb'\in\calXh_N(0),\,\xb''\in\calXh_N(1)} \Big\{ (1-\theta)\norms{\xb-\xb'} + \theta\norms{\xb-\xb''}\Big\} \label{eqn_pf:thm_conservatism_1}\\
    &\leq (1-\theta) \sup_{\xb\in\calXh_N(\theta)}\inf_{\xb'\in\calXh_N(0)} \norms{\xb-\xb'} + \theta \sup_{\xb\in\calXh_N(\theta)}\inf_{\xb''\in\calXh_N(1)} \norms{\xb-\xb''} \label{eqn_pf:thm_conservatism_2}\\ 
    &= (1-\theta) D\Big(\calXh_N(\theta),\calXh_N(0)\Big) + \theta  D\Big(\calXh_N(\theta),\calXh_N(1)\Big)  \nonumber \\
    &\leq (1-\theta) \sqrt{\frac{6C_N\theta}{\tau}}+\theta\sqrt{\frac{6C_N(1-\theta)}{\tau}} \label{eqn_pf:thm_conservatism_3}\\
    &= \sqrt{\frac{6C_N\theta(1-\theta)}{\tau}}\Big( \sqrt{\theta} + \sqrt{1-\theta}\Big) \nonumber.
\end{align}
\end{subequations}
Inequality \eqref{eqn_pf:thm_conservatism_1} follows from triangle inequality; inequality \eqref{eqn_pf:thm_conservatism_2} follows from separating the supremum operator to each summand; inequality \eqref{eqn_pf:thm_conservatism_3} follows from Theorem \ref{thm:sensitivity_in_theta}.  
\end{proof}

\subsection{Proof of Proposition~\ref{prop:bias_non_neg_UB}}

\begin{proof}
Since $\E_{\Prob^N}[\upsilonh_N(0)]\leq\upsilon^\star$, we have
\begin{align*}
    \E_{\Prob^N}[\upsilonh_N(\theta)]-\upsilon^\star &\leq  \E_{\Prob^N}[\upsilonh_N(\theta)]-\E_{\Prob^N}[\upsilonh_N(0)] \\
    &= \E_{\Prob^N}[ (1-\theta)\cdot\upsilonh_N(0)+\theta\cdot\upsilonh_N(1)+\rh_N(\theta) ] -\E_{\Prob^N}[\upsilonh_N(0)] \ \ (\text{by \eqref{eqn:expansion_in_theta} in Theorem \ref{thm:conservatism}}) \\
    &= \theta\Big\{\E_{\Prob^N}\big[\upsilonh_N(1)\big]-\E_{\Prob^N}\big[\upsilonh_N(0)\big]\Big\} + \E_{\Prob^N}[\rh_N(\theta)].
\end{align*}
 From Theorem \ref{thm:conservatism}, we have $\rh_N(\theta)\in[0,4C_N\theta(1-\theta)]$. Thus,  $R_N(\theta):=\E_{\Prob^N}[\rh_N(\theta)] \in[0,4\Cbar_N\theta(1-\theta)]$. Finally, if $\Probh_N\in\calP_N$, then $\upsilonh_N(0) \leq \upsilonh_N(1)$. It follows that the upper bound on the bias in \eqref{eqn:bias_non_neg_UB} is non-negative.  
\end{proof}

\subsection{Proof of Theorem~\ref{thm:debias_theta}}

\begin{proof}
From Theorem \ref{thm:sensitivity_in_theta}, we have that $\upsilonh_N(\theta)$ is Lipschitz continuous with Lipschitz constant $C_N<\infty$ (almost surely). Thus, $\E_{\Prob^N}[\upsilonh_N(\theta)]$ is also Lipschitz continuous with Lipschitz constant $\Cbar_N=\E_{\Prob^N}(C_N)<\infty$. Since $\E_{\Prob^N}[\upsilonh_N(0)]\leq \upsilon^\star\leq \E_{\Prob^N}[\upsilonh_N(1)]$, it follows by the intermediate value theorem (see Theorem 4.23 of \citealp{Rudin:1976}) that there exists $\theta^\textup{u}_N\in[0,1]$ for which $\E_{\Prob^N}[\upsilonh_N(\theta^\textup{u}_N)]=\upsilon^\star$. This 
 completes the proof.  
\end{proof}

\subsection{Proof of Corollary~\ref{cor:debias_theta}}

\begin{proof}
By Theorem \ref{thm:conservatism}, $\upsilonh_N(\theta)$ is concave on $[0,1]$, and so is $\E_{\Prob^N}[\upsilonh_N(\theta)]$. Since $\E_{\Prob^N}[\upsilonh_N(0)]\leq \E_{\Prob^N}[\upsilonh_N(1)]$, we have that $\E_{\Prob^N}[\upsilonh_N(\theta)]$ is either (a) non-decreasing on $[0,1]$ or (b) first non-decreasing and then non-increasing (see Lemma 1.1.4 of \citealp{Niculescu_Persson:2018}). By Theorem \ref{thm:debias_theta}, there exists $\theta_N^\text{u}$ such that $\E_{\Prob^N}[\upsilonh_N(\theta)]=\upsilon^\star$. When it is not unique, take $\theta_N^\text{u}=\inf\{\theta\in[0,1]\mid \E_{\Prob^N}[\upsilonh_N(\theta)]=\upsilon^\star\}$. It follows that $\E_{\Prob^N}[\upsilonh_N(\theta)]$ is non-decreasing on $[0,\theta_N^\text{u}]$ and thus, $|\E_{\Prob^N}[\upsilonh_N(\theta)]-\upsilon^\star| \leq |\E_{\Prob^N}[\upsilonh_N(0)]-\upsilon^\star|$ for all $\theta \in[0,\theta_N^\text{u}]$.  
\end{proof}

\subsection{Proof of Theorem~\ref{thm:rate_of_theta_LIL}}

\begin{proof}
First, note that $\theta_N^\text{u}$ satisfies $\E_{\Prob^N}[\upsilonh_N(\theta_N^\text{u})]=\upsilon^\star$, where Theorem \ref{thm:debias_theta} ensures the existence of $\theta_N^\text{u}$. Moreover, Theorem \ref{thm:conservatism} implies that
$$\E_{\Prob^N}[\upsilonh_N(\theta_N^\text{u})] = (1-\theta_N^\text{u}) \E_{\Prob^N}[\upsilonh_N(0)] + \theta_N^\text{u} \E_{\Prob^N}[\upsilonh_N(1)] + R_N(\theta_N^\text{u}),$$
where $R_N(\theta_N^\text{u})\in[0,4\Cbar_N\theta_N^\text{u}(1-\theta_N^\text{u})]$. Combining these two equations, we have
\begin{equation} \label{eqn_pf:thm_rate_of_theta_LIL_1}
    \upsilon^\star-\E_{\Prob^N}[\upsilonh_N(0)] = \theta_N^\text{u} \Big\{\E_{\Prob^N}\big[\upsilonh_N(1)\big]-\E_{\Prob^N}\big[\upsilonh_N(0)\big]\Big\} + R_N(\theta_N^\text{u}).
\end{equation}
Note that left-hand-side of \eqref{eqn_pf:thm_rate_of_theta_LIL_1} corresponds to the bias of the SAA estimator. Under assumptions (a)--(c), Theorem 6 of \cite{Banholzer_et_al:2022} gives $\upsilon^\star-\E_{\Prob^N}[\upsilonh_N(0)]=o(\sqrt{\log\log N}/\sqrt{N})$. For the ease of notation, let $\Delta_N=\E_{\Prob^N}\big[\upsilonh_N(1)\big]-\E_{\Prob^N}\big[\upsilonh_N(0)\big]$ and $b_N=\sqrt{\log\log N}/\sqrt{N}$. By the assumption that $\inf_{N\in\N} \Delta_N >0$, we know that $\Delta_N\geq \Delta$ for some $\Delta>0$. Then,
\begin{equation} \label{eqn_pf:thm_rate_of_theta_LIL_2}
    0\leq \theta_N^\text{u} \Delta \leq \theta_N^\text{u} \Delta_N + R_N(\theta_N^\text{u}) =o(b_N),
\end{equation}
where the second inequality follows from $R_N(\theta_N^\text{u})\geq 0$. In other words, \eqref{eqn_pf:thm_rate_of_theta_LIL_2} implies that
$$0 \leq \big(b_N^{-1} \theta_N^\text{u}\big)\, \Delta \leq o(1),$$
showing that $ b_N^{-1} \theta_N^\text{u}\rightarrow 0$ as $N\rightarrow \infty$. This completes the proof.  
\end{proof}

\subsection{Proof of Theorem~\ref{thm:rate_of_theta_AN}}

\begin{proof}
Theorem 5.7 of \cite{Shapiro_et_al:2014} gives the following asymptotic normality of $\upsilonh_N(0)$:
\begin{equation} \label{eqn_pf:thm_rate_of_theta_AN_1}
    X_N=\sqrt{N}\Big( \upsilonh_N(0)-\upsilon^\star\Big)\Rightarrow \inf_{\xb\in\calX^\star} \G(\xb),
\end{equation}
where ``$\Rightarrow$'' denotes the convergence in distribution, $\calX^\star$ is the set of optimal solutions to \eqref{prob:SO}, and $\G$ is a Gaussian process indexed by $\calX$ with mean zero and covariance function $\Cov(\G(\xb_1),\G(\xb_2)) = \Cov_{\Prob^\star}(f(\xb_1,\xib),f(\xb_2,\xib))$. Together with the asymptotic uniform integrability of $\{X_N\}_{N\in\N}$, \eqref{eqn_pf:thm_rate_of_theta_AN_1} implies that
\begin{equation*} 
    \E_{\Prob^N}\Big[\sqrt{N}\Big( \upsilonh_N(0)-\upsilon^\star\Big)\Big]= \E\Big[\inf_{\xb\in\calX^\star} \G(\xb)\Big] + o(1)
\end{equation*}
(see, e.g., Theorem 2.20 of \citealp{van_der_Vaart:2000}). Hence, the bias of the SAA estimator is given by
\begin{equation} \label{eqn_pf:thm_rate_of_theta_AN_2}
    \E_{\Prob^N}[\upsilonh_N(0)]-\upsilon^\star= \frac{1}{\sqrt{N}}\,\E\Big[\inf_{\xb\in\calX^\star} \G(\xb)\Big] + o(1/\sqrt{N}).
\end{equation}
If $\calX^\star$ is a singleton, then $\inf_{\xb\in\calX^\star} \G(\xb)$ reduces to a normal distribution with mean zero. Therefore, \eqref{eqn_pf:thm_rate_of_theta_AN_2} implies that the bias of the SAA estimator is of order $o(1/\sqrt{N})$. A similar  argument in the proof of Theorem~\ref{thm:rate_of_theta_LIL} shows that $\theta_N^\text{u}=o(1/\sqrt{N})$.

Next, we consider that $\calX^\star$ is not a singleton. In the trivial case when $\E_{\Prob^N}[\inf_{\xb\in\calX^\star} \G(\xb)]=0$, we immediately have $\theta_N^\text{u}=o(1/\sqrt{N})$, which directly implies $\theta_N^\text{u}=O(1/\sqrt{N})$. Now, consider the case when $\E_{\Prob^N}[\inf_{\xb\in\calX^\star} \G(\xb)]<0$. From \eqref{eqn_pf:thm_rate_of_theta_LIL_1} and \eqref{eqn_pf:thm_rate_of_theta_AN_2}, we have
\begin{equation} \label{eqn_pf:thm_rate_of_theta_AN_3}
    \theta_N^\text{u} \Big\{\E_{\Prob^N}\big[\upsilonh_N(1)\big]-\E_{\Prob^N}\big[\upsilonh_N(0)\big]\Big\} + R_N(\theta_N^\text{u}) = -\frac{1}{\sqrt{N}}\,\E\Big[\inf_{\xb\in\calX^\star} \G(\xb)\Big] + o(1/\sqrt{N}),
\end{equation}
where $R_N(\theta_N^\text{u})\in[0,4C\theta_N^\text{u}(1-\theta_N^\text{u})]$. Again, for the ease of notation, let $\Delta_N=\E_{\Prob^N}\big[\upsilonh_N(1)\big]-\E_{\Prob^N}\big[\upsilonh_N(0)\big]>0$. By the assumption that $\inf_{N\in\N} \Delta_N >0$, we know that $\Delta_N\geq \Delta$ for some $\Delta>0$. Since $R_N(\theta_N^\text{u})\geq 0$, \eqref{eqn_pf:thm_rate_of_theta_AN_3} implies
\begin{equation} \label{eqn_pf:thm_rate_of_theta_AN_4}
    0 \leq \big(\sqrt{N}\theta_N^\text{u}\big) \Delta \leq  -\E\Big[\inf_{\xb\in\calX^\star} \G(\xb)\Big] + o(1).
\end{equation}
Thus, \eqref{eqn_pf:thm_rate_of_theta_AN_4} shows that $\limsup_{N\rightarrow\infty} \sqrt{N} \theta_N^\text{u}$ is upper bounded, concluding that $\theta_N^\text{u}=O(1/\sqrt{N})$. 
\end{proof}

\subsection{Proof of Theorem~\ref{thm:TRO_estimator_variance}}

\begin{proof}
First, we show that $C_N$ defined in \eqref{eqn:def_of_CN} satisfies $C_N\leq L\cdot\diam(\Xi)+M$ almost surely. Note that
\begin{equation} \label{eqn_pf:thm:TRO_estimator_variance_1} 
    f(\xb,\xib)=\big|f(\xb,\xib)-f(\xb,\xib_0)\big| + \big|f(\xb,\xib_0)\big|\leq L\cdot  \norms{\xib-\xib_0} + M \leq L\cdot\diam(\Xi)+M,
\end{equation}
where the first inequality follows from assumptions (a) and (c), and the second inequality follows from assumption (b). Therefore,  we have $\E_{\Prob}[f(\xb,\xib)]\leq L\cdot\diam(\Xi)+M$ for any $\Prob\in\calP(\Xi)$. It follows from the definition of $C_N$ that $C_N\leq L\cdot\diam(\Xi)+M$ almost surely.

Next, we show the desired inequality. Note that almost surely, we have
\begin{subequations}
    \begin{align}\upsilonh_N(0)\leq\upsilonh_N(\theta)&\leq\upsilonh_N(0)+\Big\{\theta\big[\upsilonh_N(1)-\upsilonh_N(0)\big]+4\theta C_N(1-\theta)\Big\},  \label{eqn_pf:thm:TRO_estimator_variance_2} \\
    &\leq\upsilonh_N(0)+\Big[2\theta C_N+4\theta C_N(1-\theta)\Big],  \label{eqn_pf:thm:TRO_estimator_variance_3} \\
    &\leq \upsilonh_N(0)+2\theta(3-2\theta)\big(L\cdot\diam(\Xi)+M\big).  \label{eqn_pf:thm:TRO_estimator_variance_4} 
\end{align}
\end{subequations}
The first inequality in \eqref{eqn_pf:thm:TRO_estimator_variance_2} follows from assumption (d), the second inequality in \eqref{eqn_pf:thm:TRO_estimator_variance_2} follows from part (i) of Theorem~\ref{thm:conservatism}, \eqref{eqn_pf:thm:TRO_estimator_variance_3} follows from Theorem~\ref{thm:sensitivity_in_theta}, and \eqref{eqn_pf:thm:TRO_estimator_variance_4} follows from \eqref{eqn_pf:thm:TRO_estimator_variance_1}. 
Multiplying both sides of inequality \eqref{eqn_pf:thm:TRO_estimator_variance_4} by $-1$ and taking the expectation with respect to $\Prob^N$, we obtain
\begin{equation} \label{eqn_pf:thm:TRO_estimator_variance_5} 
    -\E_{\Prob^N}[\upsilonh_N(0)]-2\theta(3-2\theta)\big(L\cdot\diam(\Xi)+M\big) \leq-\E_{\Prob^N}[\upsilonh_N(\theta)]\leq -\E_{\Prob^N}[\upsilonh_N(0)].
\end{equation}
Together with \eqref{eqn_pf:thm:TRO_estimator_variance_2}--\eqref{eqn_pf:thm:TRO_estimator_variance_4} and \eqref{eqn_pf:thm:TRO_estimator_variance_5}, we have
\begin{align}
    &\quad\,\,\upsilonh_N(0)-\E_{\Prob^N}[\upsilonh_N(0)]-2\theta(3-2\theta)\big(L\cdot\diam(\Xi)+M\big) \nonumber \\
    &\leq \upsilonh_N(\theta)-\E_{\Prob^N}[\upsilonh_N(\theta)] \nonumber \\
    &\leq  \upsilonh_N(0) -\E_{\Prob^N}[\upsilonh_N(0)]+2\theta(3-2\theta)\big(L\cdot\diam(\Xi)+M\big), \nonumber 
\end{align}
which implies that
\begin{equation} \label{eqn_pf:thm:TRO_estimator_variance_6}
    \Big| \Big\{\upsilonh_N(\theta)-\E_{\Prob^N}[\upsilonh_N(\theta)]\Big\} - \Big\{\upsilonh_N(0)-\E_{\Prob^N}[\upsilonh_N(0)]\Big\} \Big| \leq 2\theta(3-2\theta)\big(L\cdot\diam(\Xi)+M\big)
\end{equation}
almost surely. Therefore, we have
\begin{align*}
    &\quad\,\,\sqrt{\Var_{\Prob^N}\big(\upsilonh_N(\theta)\big)} \\
    &=\sqrt{\E_{\Prob^N} \Big[\big(\upsilonh_N(\theta)-\E_{\Prob^N}[\upsilonh_N(\theta)]\big)^2\Big]} \\
    &=\sqrt{\E_{\Prob^N} \Big[\Big(\big(\upsilonh_N(0)-\E_{\Prob^N}[\upsilonh_N(0)]\big)+\Big[\big(\upsilonh_N(\theta)-\E_{\Prob^N}[\upsilonh_N(\theta)]\big)-\big(\upsilonh_N(0)-\E_{\Prob^N}[\upsilonh_N(0)]\big)\Big]\Big)^2\Big]} \\
    &\leq \sqrt{\E_{\Prob^N} \Big[\big(\upsilonh_N(0)-\E_{\Prob^N}[\upsilonh_N(0)]\big)^2\Big]} + \sqrt{\E_{\Prob^N} \Big[\Big(\Big\{\upsilonh_N(\theta)-\E_{\Prob^N}[\upsilonh_N(\theta)]\Big\} - \Big\{\upsilonh_N(0)-\E_{\Prob^N}[\upsilonh_N(0)]\Big\}\Big)^2\Big]}\\
    &\leq\sqrt{\Var_{\Prob^N}\big(\upsilonh_N(0)\big)}+2\theta(3-2\theta)\big(L\cdot\diam(\Xi)+M\big),
\end{align*}
where the first inequality follows from Minkowski's inequality, and the second inequality follows from \eqref{eqn_pf:thm:TRO_estimator_variance_6}. This completes the proof. 
\end{proof}

\subsection{Proof of Theorem~\ref{thm:generalization_error}}

\begin{proof}

First, we prove the following lemma.
\begin{lemma}\label{lem:generalization_error_inter}
Let $A_1$ and $A_2$ be random variables, $\Prob$ be any probability distribution, and $\theta \in (0,1)$.  For any $\delta>0$, the following inequality holds.
\begin{equation} \label{eqn_pf:thm:generalization_error_1}
    \Prob\big((1-\theta) A_1 + \theta A_2>\delta\big) \leq \inf_{\gamma\in[0,\delta]}\bigg\{\Prob\bigg(A_1>\frac{\gamma}{1-\theta}\bigg)+\Prob\bigg(A_2>\frac{\delta-\gamma}{\theta}\bigg) \bigg\}
\end{equation}
\end{lemma}

\begin{proof}{Proof of Lemma~\ref{lem:generalization_error_inter}.}
Consider the event $(1-\theta)A_1+\theta A_2>\delta$. We have either $A_1>\gamma/(1-\theta)$ or $A_2>(\delta-\gamma)/\theta$ for any $\gamma\in[0,\delta]$. It follows from the union bound (Boole's inequality) that
\begin{equation} \label{eqn_pf:thm:generalization_error_2}
   \Prob\big((1-\theta) A_1 + \theta A_2>\delta\big) \leq \Prob\bigg(A_1>\frac{\gamma}{1-\theta}\bigg)+\Prob\bigg(A_2>\frac{\delta-\gamma}{\theta}\bigg)
\end{equation}
for any $\gamma\in[0,\delta]$. Inequality \eqref{eqn_pf:thm:generalization_error_1} follows directly from \eqref{eqn_pf:thm:generalization_error_2}, as the right-hand side of \eqref{eqn_pf:thm:generalization_error_2} holds for  any $\gamma\in[0,\delta]$. This completes the proof of Lemma \ref{lem:generalization_error_inter}. 
\end{proof}

%
%
%


With Lemma \ref{lem:generalization_error_inter}, we are ready to derive the desired generalization bound as follows:
\begin{subequations}
\begin{align}
    &\quad\,\, \Prob^N\Bigg(\sup_{\xb\in\calX}\bigg\{ \E_{\Prob^\star}[f(\xb,\xib)]- \sup_{\Prob\in\calP'_{N,\theta}}\E_{\Prob}[f(\xb,\xib)] \bigg\}>\delta\Bigg) \nonumber\\
    &=  \Prob^N\Bigg(\sup_{\xb\in\calX}\bigg\{ (1-\theta)\cdot\Big(\E_{\Prob^\star}[f(\xb,\xib)]-\E_{\Probh_N}[f(\xb,\xib)]\Big) + \theta\cdot \Big( \E_{\Prob^\star}[f(\xb,\xib)]-\sup_{\Prob\in\calP_N}\E_{\Prob}[f(\xb,\xib)]\Big) \bigg\}>\delta\Bigg) \nonumber \\
    &\leq  \Prob^N\Bigg( (1-\theta) \sup_{\xb\in\calX}\bigg\{\E_{\Prob^\star}[f(\xb,\xib)]-\E_{\Probh_N}[f(\xb,\xib)]\bigg\} + \theta \sup_{\xb\in\calX}\bigg\{ \E_{\Prob^\star}[f(\xb,\xib)]-\sup_{\Prob\in\calP_N}\E_{\Prob}[f(\xb,\xib)] \bigg\}>\delta\Bigg) \nonumber \\
    &\leq \inf_{\gamma\in[0,\delta]}\Bigg\{ \Prob^N\Bigg(\sup_{\xb\in\calX}\bigg\{ \E_{\Prob^\star}[f(\xb,\xib)]- \E_{\Probh_N}[f(\xb,\xib)] \bigg\}>\frac{\gamma}{1-\theta}\Bigg)  \nonumber \\
    &\hspace{43mm}+ \Prob^N\Bigg(\sup_{\xb\in\calX}\bigg\{ \E_{\Prob^\star}[f(\xb,\xib)]- \sup_{\Prob\in\calP_N}\E_{\Prob}[f(\xb,\xib)] \bigg\}>\frac{\delta-\gamma}{\theta}\Bigg)\Bigg\}  \label{eqn_pf:thm:generalization_error_3a} \\
    &\leq \inf_{\gamma\in[0,\delta]}\bigg\{\alpha_{N,1}\bigg(\frac{\gamma}{1-\theta}\bigg)+\alpha_{N,2}\bigg(\frac{\delta-\gamma}{\theta}\bigg)\bigg\}. \label{eqn_pf:thm:generalization_error_3b} 
\end{align}
\end{subequations}
Inequality~\eqref{eqn_pf:thm:generalization_error_3a} follows from \eqref{eqn_pf:thm:generalization_error_1} and inequality \eqref{eqn_pf:thm:generalization_error_3b}  follows from~\eqref{eqn:generalization_error_bound_SAA} and \eqref{eqn:generalization_error_bound_DRO}. 
\end{proof}

\subsection{Proof of Lemma~\ref{lem:DRO_component_asym_boundedness}}

\begin{proof}
For brevity, all convergence, equalities and inequalities hold almost surely in this proof. If $f$ is uniformly bounded, i.e., $\sup_{\xb\in\calX}\sup_{\xib\in\Xi}f(\xb,\xib)<\infty$, then the desired inequality follows immediately. Now, suppose that Assumption \ref{assumption:ambig_set_regularity}(a) holds. We can directly obtain an upper bound from \eqref{eqn:DRO_exp_UB_w_data}:
\begin{equation} \label{eqn_pf:DRO_component_asym_boundedness_1}
    \sup_{\Prob\in\calP_N} \E_{\Prob}|f(\xb,\xib)| \leq \sup_{\Prob\in\calPh} \E_{\Prob}[\kappa(\xib)] \cdot \diam(\calX) + \sup_{\Prob\in\calPh} \E_{\Prob}|f(\xb_0,\xib)|:= M<\infty
\end{equation}
for all sufficiently large $N$. Thus, we obtain $\limsup_{N\rightarrow\infty} \sup_{\xb\in\calX} \sup_{\Prob\in\calP_N} \E_{\Prob}|f(\xb,\xib)|\leq M$ by taking supremum over $\xb\in\calX$ on both sides of \eqref{eqn_pf:DRO_component_asym_boundedness_1}.

Now, suppose that Assumption \ref{assumption:ambig_set_regularity}(b) holds. Then, for any $\xb\in\calX$, we have
\begin{align} 
    \sup_{\Prob\in\calP_N} \E_{\Prob}[f(\xb,\xib)] &\leq \bigg|\sup_{\Prob\in\calP_N} \E_{\Prob}[f(\xb,\xib)] - \sup_{\Prob\in\calPh} \E_{\Prob}[f(\xb,\xib)] \Bigg| + \sup_{\Prob\in\calPh} \E_{\Prob}[f(\xb,\xib)] \nonumber \\
    &\leq \bbmH(\calP_N,\calPh) + \Bigg\{ \sup_{\Prob\in\calPh} \E_{\Prob}[\kappa(\xib)] \cdot \diam(\calX) + \sup_{\Prob\in\calPh} \E_{\Prob}|f(\xb_0,\xib)|\Bigg\}, \label{eqn_pf:DRO_component_asym_boundedness_2}
\end{align}
%
%
where the first term in \eqref{eqn_pf:DRO_component_asym_boundedness_2} follows from the definition of $\bbmH(\calP_N,\calPh)$ and the second term in \eqref{eqn_pf:DRO_component_asym_boundedness_2} follows from the same argument in \eqref{eqn_pf:DRO_component_asym_boundedness_1}. Since $\bbmH(\calP_N,\calPh)\rightarrow 0$, \eqref{eqn_pf:DRO_component_asym_boundedness_2} implies that 
\begin{equation*}
    \limsup_{N\rightarrow\infty}\sup_{\xb\in\calX}\sup_{\Prob\in\calP_N} \E_{\Prob}|f(\xb,\xib)|  \leq  \sup_{\Prob\in\calPh} \E_{\Prob}[\kappa(\xib)] \cdot \diam(\calX) + \sup_{\Prob\in\calPh} \E_{\Prob}|f(\xb_0,\xib)| < \infty.
\end{equation*}

%
%

Finally, suppose that Assumption \ref{assumption:ambig_set_regularity}(c) holds. For any $\Prob\in\calP_N$ identified as $\pb\in\R_+^N$, 
\begin{align} 
    \big|\E_{\Prob}[\kappa(\xib)] - \E_{\Probh_N}[\kappa(\xib)]\big|= \Bigg| \sum_{i=1}^N \bigg(p_i-\frac{1}{N}\bigg) \kappa(\xibh_i) \Bigg| &\leq \norm{\pb-\frac{1}{N}\one}_\infty \cdot \sum_{i=1}^N \kappa(\xibh_i) \nonumber \\
    &= \frac{1}{N} \norms{N\pb-\one}_\infty \cdot \sum_{i=1}^N \kappa(\xibh_i) \nonumber \\
    &= \norms{N\pb-\one}_\infty \cdot \E_{\Probh_N}[\kappa(\xib)]. \label{eqn_pf:DRO_component_asym_boundedness_4}
\end{align}
Hence, we have
\allowdisplaybreaks
\begin{align}
    \limsup_{N\rightarrow\infty} \sup_{\Prob\in\calP_N} \E_{\Prob}[\kappa(\xib)] &\leq \limsup_{N\rightarrow\infty} \sup_{\Prob\in\calP_N} \Big\{ \E_{\Probh_N}[\kappa(\xib)] + \big|\E_{\Prob}[\kappa(\xib)] - \E_{\Probh_N}[\kappa(\xib)]\big| \Big\} \nonumber \\
    &\leq \limsup_{N\rightarrow\infty} \E_{\Probh_N}[\kappa(\xib)] + \limsup_{N\rightarrow\infty} \Bigg\{ \sup_{\Prob\in\calP_N}  \norms{N\pb-\one}_\infty \cdot \E_{\Probh_N}[\kappa(\xib)] \Bigg\}  \label{eqn_pf:DRO_component_asym_boundedness_5} \\
    &\leq \limsup_{N\rightarrow\infty} \E_{\Probh_N}[\kappa(\xib)] + \Bigg(\limsup_{N\rightarrow\infty}  \sup_{\Prob\in\calP_N}  \norms{N\pb-\one}_\infty \Bigg)\Bigg( \limsup_{N\rightarrow\infty}   \E_{\Probh_N}[\kappa(\xib)]\Bigg) \nonumber \\ 
    &= \Bigg(1 + \limsup_{N\rightarrow\infty}  \sup_{\Prob\in\calP_N}  \norms{N\pb-\one}_\infty \Bigg) \Bigg( \limsup_{N\rightarrow\infty}   \E_{\Probh_N}[\kappa(\xib)]\Bigg), \nonumber 
\end{align}
where the second term in \eqref{eqn_pf:DRO_component_asym_boundedness_5} follows from \eqref{eqn_pf:DRO_component_asym_boundedness_4}. Since $\{\xibh_i\}_{i=1}^N$ are i.i.d. following $\Prob^\star$, the strong law of large numbers gives $ \limsup_{N\rightarrow\infty} \E_{\Probh_N}[\kappa(\xib)] =\E_{\Prob^\star}[\kappa(\xib)]<\infty$. Moreover, since $\limsup_{N\rightarrow\infty} \sup_{\Prob\in\calP_N}\norms{N\pb-\one}_\infty <\infty$, we have $\limsup_{N\rightarrow\infty} \sup_{\Prob\in\calP_N} \E_{\Prob}[\kappa(\xib)] <\infty$. Using a similar argument, we can obtain $\limsup_{N\rightarrow\infty} \sup_{\Prob\in\calP_N} \E_{\Prob}|f(\xb_0,\xib)| <\infty$. Therefore, we have
\begin{align}
    \limsup_{N\rightarrow\infty}\sup_{\xb\in\calX}\sup_{\Prob\in\calP_N} \E_{\Prob}|f(\xb,\xib)|     &\leq \limsup_{N\rightarrow\infty} \Bigg\{\sup_{\Prob\in\calP_N} \E_{\Prob}[\kappa(\xib)] \cdot \diam(\calX) +  \sup_{\Prob\in\calP_N} \E_{\Prob}|f(\xb_0,\xib)| \Bigg\} \nonumber \\
    & \leq \limsup_{N\rightarrow\infty} \sup_{\Prob\in\calP_N} \E_{\Prob}[\kappa(\xib)] \cdot \diam(\calX) +  \limsup_{N\rightarrow\infty} \sup_{\Prob\in\calP_N} \E_{\Prob}|f(\xb_0,\xib)| <\infty. \nonumber
\end{align}
This completes the proof. 
%
\end{proof}

\subsection{Proof of Theorem~\ref{thm:asymptotic_convergence}}

\begin{proof}
For brevity, all convergence, equalities and inequalities hold almost surely in this proof. First, we claim that the following uniform convergence holds:
\begin{equation} \label{eqn_pf:asymptotic_convergence_1}
    \sup_{\xb\in\calX} \Bigg| \Bigg\{ (1-\theta_N)\cdot \frac{1}{N} \sum_{i=1}^N f(\xb,\xibh_i) + \theta_N\cdot\sup_{\Prob\in\calP_N} \E_{\Prob}[f(\xb,\xib)] \Bigg\}  - \E_{\Prob^\star}[f(\xb,\xib)] \Bigg| \rightarrow 0
\end{equation}
as $N\rightarrow\infty$. To prove \eqref{eqn_pf:asymptotic_convergence_1}, note that
\begin{align}
    &\quad\, \sup_{\xb\in\calX} \Bigg| \Bigg\{ (1-\theta_N)\cdot \frac{1}{N} \sum_{i=1}^N f(\xb,\xibh_i) + \theta_N\cdot\sup_{\Prob\in\calP_N} \E_{\Prob}[f(\xb,\xib)] \Bigg\}  - \E_{\Prob^\star}[f(\xb,\xib)] \Bigg| \nonumber \\
    & \leq  (1-\theta_N)\cdot\sup_{\xb\in\calX} \Bigg| \frac{1}{N} \sum_{i=1}^N f(\xb,\xibh_i)  - \E_{\Prob^\star}[f(\xb,\xib)]  \Bigg| + \theta_N\cdot\sup_{\xb\in\calX} \Bigg| \sup_{\Prob\in\calP_N} \E_{\Prob}[f(\xb,\xib)] - \E_{\Prob^\star}[f(\xb,\xib)]  \Bigg|. \label{eqn_pf:prop_uniform_convergence_2}
\end{align}
Assumption \ref{assumption:GC-class} implies that the first term $\sup_{\xb\in\calX} \big| N^{-1} \sum_{i=1}^N f(\xb,\xibh_i)  - \E_{\Prob^\star}[f(\xb,\xib)] \big|$ converges to zero.
Next, for the second term, note that
\begin{equation} \label{eqn_pf:prop_uniform_convergence_3}
    \theta_N\cdot\sup_{\xb\in\calX} \Bigg| \sup_{\Prob\in\calP_N} \E_{\Prob}[f(\xb,\xib)] - \E_{\Prob^\star}[f(\xb,\xib)]  \Bigg| \leq \theta_N\Bigg(\sup_{\xb\in\calX} \sup_{\Prob\in\calP_N} \E_{\Prob} |f(\xb,\xib)| + \sup_{\xb\in\calX} \E_{\Prob^{\star}}|f(\xb,\xib)| \Bigg)
\end{equation}
Since $\limsup_{N\rightarrow\infty} \sup_{\xb\in\calX} \sup_{\Prob\in\calP_N} \E_{\Prob} |f(\xb,\xib)| \leq M$ by Lemma \ref{lem:DRO_component_asym_boundedness} and $\sup_{\xb\in\calX} \E_{\Prob^{\star}}|f(\xb,\xib)|<\infty$, the upper bound in \eqref{eqn_pf:prop_uniform_convergence_3} converges to zero by $\theta_N=o(1)$. Therefore, the upper bound in \eqref{eqn_pf:prop_uniform_convergence_2} converges to zero, which shows \eqref{eqn_pf:asymptotic_convergence_1}.

Now, with the use of \eqref{eqn_pf:asymptotic_convergence_1}, we can prove the desired asymptotic convergence results in a way similar to Theorem 5.3 of \cite{Shapiro_et_al:2014}. For completeness, we also provide the details here. To show assertion (i), note that
\begin{align*}
    |\upsilonh_N(\theta_N)-\upsilon^\star| &= \Bigg| \min_{\xb\in\calX} \Bigg\{ (1-\theta_N)\cdot \frac{1}{N} \sum_{i=1}^N f(\xb,\xibh_i) + \theta_N\cdot\sup_{\Prob\in\calP_N} \E_{\Prob}[f(\xb,\xib)] \Bigg\}  - \min_{\xb\in\calX} \E_{\Prob^\star}[f(\xb,\xib)]  \Bigg|\\
    &\leq \sup_{\xb\in\calX} \Bigg| \Bigg\{ (1-\theta_N)\cdot \frac{1}{N} \sum_{i=1}^N f(\xb,\xibh_i) + \theta_N\cdot\sup_{\Prob\in\calP_N} \E_{\Prob}[f(\xb,\xib)] \Bigg\}  - \E_{\Prob^\star}[f(\xb,\xib)] \Bigg|,
\end{align*}
which converges to zero by \eqref{eqn_pf:asymptotic_convergence_1}. Next, to show assertion (ii), suppose, on the contrary, that $D\big(\calXh_N(\theta_N),\calX^\star\big) \rightarrow 0$ does not hold almost surely, i.e., $\Prob^\infty\big(D\big(\calXh_N(\theta_N),\calX^\star\big) \not\rightarrow 0\big)>0$. Consider a data sequence such that the event $D\big(\calXh_N(\theta_N),\calX^\star\big) \not\rightarrow 0$ holds. Then, for some $\epsilon>0$, there exists a sequence $\{\xb_{N_j}\}_{j\in\N}$ such that $\xb_{N_j}\in\calXh_{N_j}(\theta_{N_j})$ and $d\big(\xb_{N_j},\calX^\star\big)>\epsilon$. Since $\calX$ is compact by Assumption \ref{assumption:loss_and_ambig_set_boundedness}, without loss of generality, we can assume that $\xb_{N_j}\rightarrow\xbbar$ for some $\xbbar\in\calX$. By continuity of the distance function $d$, we have $d(\xbbar,\calX^\star)=\lim_{j\rightarrow\infty}d(\xb_{N_j},\calX^\star)>\epsilon$. Thus, $\xbbar\not\in\calX^\star$, which implies that $\E_{\Prob^\star}[f(\xbbar,\xib)] > \upsilon^\star$. However, note that
$$\Big| \E_{\Prob^\star}[f(\xbbar,\xib)] - \upsilonh_{N_j}(\theta_{N_j})\Big| \leq  \Big| \E_{\Prob^\star}[f(\xbbar,\xib)] - \E_{\Prob^\star}[f(\xb_{N_j},\xib)]\Big| +  \Big| \E_{\Prob^\star}[f(\xb_{N_j},\xib)] - \upsilonh_{N_j}(\theta_{N_j})\Big|,$$
where the first term converges to zero by the continuity of $\E_{\Prob^\star}[f(\cdot,\xib)]$, and the second term converges to zero by \eqref{eqn_pf:asymptotic_convergence_1}. Therefore, we arrive at $\lim_{j\rightarrow\infty} \upsilonh_{N_j}(\theta_{N_j}) = \E_{\Prob^\star}[f(\xbbar,\xib)] > \upsilon^\star$, which contradicts assertion (i) that $\lim_{j\rightarrow\infty} \upsilonh_{N_j}(\theta_{N_j}) = \upsilon^\star$. 
\end{proof}

\subsection{Proof of Lemma~\ref{lem:asy_tight_tradeoff_ambig_set}}

\begin{proof}

Recall that, following the convention in empirical process theory, we view $\Prob\in\calP(\Xi)$ as an element in $\ell^\infty(\calH)$ defined as $\Prob (h) = \E_{\Prob}(h)$ for $h\in\calH$. We divide the proof of the desired weak convergence  $\S_N\Rightarrow\G'$ into two steps.

\textit{Step 1.} We first show that $\big(\S_N(h_1),\dots,\S_N(h_k)\big)\Rightarrow\big(\G'(h_1),\dots,\G'(h_k)\big)$ for any finite subset $\{h_1,\dots,h_k\}\subset\calH$. To prove this, we write $\S_N = \sqrt{N}(1-\theta_N)(\Probh_N-\Prob^\star) + \sqrt{N}\theta_N(\Prob_N-\Prob^\star)$. Note that $\sqrt{N}(\Probh_N-\Prob^\star)(h_1,\dots,h_k) \Rightarrow\big(\G'(h_1),\dots,\G'(h_k)\big)$ by Assumption~\ref{assumption:Donsker}. Also, from the assumption that $\theta_N=o(N^{-1/2})$, we have $(1-\theta_N)\rightarrow 1$. Hence, for the first term in $\S_N$, we have $(1-\theta_N)\cdot \sqrt{N}(\Probh_N-\Prob^\star)(h_1,\dots,h_k)\Rightarrow\big(\G'(h_1),\dots,\G'(h_k)\big)$ by Slutsky's Theorem (see, e.g., Example 1.4.7 of \citealp{van_der_Vaart_Wellner:1996}). Next, for the second term in $\S_N$, Lemma \ref{lem:DRO_component_asym_boundedness} implies that $\limsup_{N\rightarrow\infty}\sup_{h\in\calH}\sup_{\Prob\in\calP_N} \E_{\Prob}|h| \leq M$ almost surely for some constant $M$. Therefore, there exists constant $M'$ such that $\limsup_{N\rightarrow\infty}|\E_{\Prob_N}(h)-\E_{\Prob^\star}(h)| \leq \limsup_{N\rightarrow\infty} \sup_{h\in\calH} \sup_{\Prob\in\calP_N}|\E_{\Prob}(h)-\E_{\Prob^\star}(h)| \leq M'$ for any $\Prob_N\in\calP_N$ and $h\in\calH$ (since $\E_{\Prob^\star}|h|<\infty$). By assumption (b), $\sqrt{N}\theta_N$ converges to zero as $N\rightarrow\infty$ and hence, $\sqrt{N}\theta_N(\Prob_N-\Prob^\star)(h)$ also converges to zero almost surely. This implies that $\sqrt{N}\theta_N(\Prob_N-\Prob^\star)(h_1,\dots,h_k)\rightarrow (0,\dots,0)\in\R^k$ almost surely. Therefore, provoking Slutsky's Theorem again, we obtain the desired convergence, i.e., $\big(\S_N(h_1),\dots,\S_N(h_k)\big)\Rightarrow\big(\G'(h_1),\dots,\G'(h_k)\big)$. This completes step 1. 

\textit{Step 2.} We show that $\S_N$ is asymptotically tight. Note that for any $\varepsilon>0$,
\begin{subequations}
\begin{align}
    &\quad\,\limsup_{\delta\rightarrow 0}\limsup_{N\rightarrow\infty} \Prob^N\Bigg( \sup_{\norms{h-h'}_{L^2(\Prob^\star)}<\delta} \big| \S_N(h-h') \big| \geq \varepsilon \Bigg) \label{eqn_pf:thm_asy_dist_0} \\
    &\leq \limsup_{\delta\rightarrow 0}\limsup_{N\rightarrow\infty}  \Prob^N\Bigg( \sup_{\norms{h-h'}_{L^2(\Prob^\star)}<\delta} \big| (1-\theta_N)\sqrt{N}(\Probh_N-\Prob^\star)(h-h') \big| \geq \frac{\varepsilon}{2} \Bigg) \label{eqn_pf:thm_asy_dist_1}\\
    &\hspace{10mm}+  \limsup_{\delta\rightarrow 0}\limsup_{N\rightarrow\infty}\Prob^N\Bigg( \sup_{\norms{h-h'}_{L^2(\Prob^\star)}<\delta} \big| \theta_N\sqrt{N}(\Prob_N-\Prob^\star)(h-h') \big| \geq \frac{\varepsilon}{2} \Bigg) .  \label{eqn_pf:thm_asy_dist_2}
\end{align}
\end{subequations}
Since $\sqrt{N}(1-\theta_N)(\Probh_N-\Prob^\star)\Rightarrow\G'$, the sequence $\{\sqrt{N}(1-\theta_N)(\Probh_N-\Prob^\star)\}_{N\in\N}$ is asymptotically tight. Therefore, \eqref{eqn_pf:thm_asy_dist_1} equals zero by Theorem 1.5.7 of \cite{van_der_Vaart_Wellner:1996}. Now, we show that \eqref{eqn_pf:thm_asy_dist_2} also vanishes. It suffices to show that the sequence $\{\sqrt{N}\theta_N(\Prob_N-\Prob^\star)\}_{N\in\N}$ is asymptotically tight. Note that
$$\Q_N(h):=\sqrt{N}\theta_N(\Prob_N-\Prob^\star)(h)=\sqrt{N}\theta_N [\E_{\Prob_N}(h)-\E_{\Prob^\star}(h)].$$
By assumption (a), the metric space $(\calH,\norms{\cdot}_{L^2(\Prob^\star)})$ is totally bounded. Also, as shown earlier (in step~1), $\Q_N(h)$ converges to zero almost surely. It follows by Theorem~1.5.4 of \cite{van_der_Vaart_Wellner:1996} that $\Q_N(h)$ is asymptotically tight for any $h\in\calH$. Next, we claim that $\Q_N$ is asymptotically uniform equicontinuous in probability, i.e., for any $\varepsilon>0$, and $\eta>0$, there exists $\delta>0$ such that
$$\limsup_{N\rightarrow\infty}\Prob^N\bigg(\sup_{\norms{h-h'}_{L^2(\Prob^\star)}<\delta} |\Q_N(h-h')| >\varepsilon\bigg) < \eta.$$
For notational simplicity, we write $\norms{h-h'}=\norms{h-h'}_{L^2(\Prob^\star)}$. Note that
\begin{align*}
    \sup_{\norms{h-h'}<\delta}\big| \Q_N(h-h') \big| &= \sup_{\norms{h-h'}<\delta}\big| \sqrt{N}\theta_N (\Prob_N-\Prob^\star)(h-h') \big| \\
    &\leq \sqrt{N}\theta_N  \sup_{\norms{h-h'}<\delta} \Big\{ \big| \Prob_N(h-h') \big| + \big| \Prob^\star(h-h') \big| \Big\} \\
    &\leq  \sqrt{N}\theta_N  \sup_{\norms{h-h'}<\delta}  \big| \Prob_N(h-h') \big| +   \sqrt{N}\theta_N  \sup_{\norms{h-h'}<\delta}  \big| \Prob^\star(h-h') \big| .
\end{align*}
Therefore, we have
\begin{align} 
    &\quad\,\,\Prob^N\bigg(\sup_{\norms{h-h'}_{L^2(\Prob^\star)}<\delta} |\Q_N(h-h')| >\varepsilon\bigg) \nonumber \\
    &\leq \Prob^N\bigg(\sqrt{N}\theta_N \sup_{\norms{h-h'}<\delta} |\Prob_N(h-h')| >\frac{\varepsilon}{2}\bigg)  + \Prob^N\bigg(\sqrt{N}\theta_N \sup_{\norms{h-h'}<\delta} |\Prob^\star(h-h')| >\frac{\varepsilon}{2}\bigg) \nonumber \\
    &=: A_N + B_N. \label{eqn_pf:lem_asy_tight_tradeoff_ambig_set_1}
\end{align}
Consider the term $B_N$ in \eqref{eqn_pf:lem_asy_tight_tradeoff_ambig_set_1}. For any $h$ and $h'$ such that $\norms{h-h'}<\delta$, we have $|\Prob^\star(h-h')|=|\E_{\Prob^\star}(h-h')| \leq \E_{\Prob}|h-h'|\leq \sqrt{\E_{\Prob^\star}[(h-h')^2]}=\norms{h-h'}<\delta$, where the second inequality follows from Cauchy-Schwartz inequality. Also, by assumption (b), we have $\sqrt{N}\theta_N=o(1)$. Therefore,
\begin{equation}\label{eqn_pf:lem_asy_tight_tradeoff_ambig_set_2}
    \limsup_{N\rightarrow\infty} B_N \leq \Prob^\infty\bigg(\limsup_{N\rightarrow\infty} \big\{\sqrt{N}\theta_N \delta\big\} > \frac{\varepsilon}{2}\bigg) = 0.
\end{equation}
Consider the term $A_N$ in \eqref{eqn_pf:lem_asy_tight_tradeoff_ambig_set_1}. Note that for any $\Prob_N\in\calP_N$,
\begin{equation}  \label{eqn_pf:lem_asy_tight_tradeoff_ambig_set_3}
    \sup_{\norms{h-h'}<\delta} \E_{\Prob_N} |h-h'| \leq 2\sup_{h\in\calH} \E_{\Prob_N} |h| \leq  2\sup_{h\in\calH} \sup_{\Prob\in\calP_N} \E_{\Prob} |h|.
\end{equation}
By Lemma \ref{lem:DRO_component_asym_boundedness}, \eqref{eqn_pf:lem_asy_tight_tradeoff_ambig_set_3} implies that $\limsup_{N\rightarrow\infty} \sup_{\norms{h-h'}<\delta} \E_{\Prob_N} |h-h'| \leq M''$ almost surely for some constant $M''$. Since $\sqrt{N}\theta_N=o(1)$, we have  $\limsup_{N\rightarrow\infty}\big\{ \sqrt{N}\theta_N\sup_{\norms{h-h'}<\delta} \E_{\Prob_N} |h-h'|\big\} =0$. Therefore, we have
\begin{equation}\label{eqn_pf:lem_asy_tight_tradeoff_ambig_set_4}
    \limsup_{N\rightarrow\infty} A_N \leq \Prob^\infty\bigg(\limsup_{N\rightarrow\infty} \Big\{ \sqrt{N}\theta_N\sup_{\norms{h-h'}<\delta} \E_{\Prob_N} |h-h'| \Big\} > \frac{\varepsilon}{2}\bigg) = 0.
\end{equation}
Combining \eqref{eqn_pf:lem_asy_tight_tradeoff_ambig_set_2} and \eqref{eqn_pf:lem_asy_tight_tradeoff_ambig_set_4} with \eqref{eqn_pf:lem_asy_tight_tradeoff_ambig_set_1}, we have
\begin{equation}\label{eqn_pf:lem_asy_tight_tradeoff_ambig_set_5}
    0\leq \limsup_{N\rightarrow\infty} \Prob^N\bigg(\sup_{\norms{h-h'}_{L^2(\Prob^\star)}<\delta} |\Q_N(h-h')| >\varepsilon\bigg) \leq \limsup_{N\rightarrow\infty} A_N + \limsup_{N\rightarrow\infty} B_N \leq 0.
\end{equation}
Since \eqref{eqn_pf:lem_asy_tight_tradeoff_ambig_set_5} holds for any $\delta>0$, this shows that $\Q_N$ is asymptotically uniform continuous in probability. Since (a) the metric space $(\calH,\norms{\cdot}_{L^2(\Prob^\star)})$ is totally bounded, (b) $\Q_N(h)$ is asymptotically tight for any $h\in\calH$, and (c) $\Q_N$ is asymptotically uniform continuous in probability,  $\Q_N$ is asymptotically tight by Theorem~1.5.7 of \cite{van_der_Vaart_Wellner:1996}, implying that \eqref{eqn_pf:thm_asy_dist_2} equals zero. Since both \eqref{eqn_pf:thm_asy_dist_1} and \eqref{eqn_pf:thm_asy_dist_2} equal zero, \eqref{eqn_pf:thm_asy_dist_0} also equals zero, showing that the sequence $\{\S_N\}_{N\in\N}$ is asymptotically tight. This completes step~2.

Combining the two steps, we have (a) $\{\S_N\}$ is asymptotically tight and (b) the marginals $\big(\S_N(h_1),\dots,\S_N(h_k)\big)$ converge weakly to $\big(\G'(h_1),\dots,\G'(h_k)\big)$. It follows from Theorem 1.5.4 of \cite{van_der_Vaart_Wellner:1996} that $\S_N\Rightarrow\G'$. 
\end{proof}

\subsection{Proof of Theorem~\ref{thm:asy_dist}}

\begin{proof}

By our assumption, there exists  $\Prob_N^\star \in\argmax_{\Prob\in\calP_N} \E_{\Prob}[f(\xb,\xib)]$ such that $\Prob_N^\star\in\calP_N$ for any $\xb\in\calX$. Thus, by Lemma~\ref{lem:asy_tight_tradeoff_ambig_set}, we have
\begin{equation} \label{eqn_pf:thm_asy_dist_3}
    \sqrt{N}\bigg[(1-\theta_N)\Probh_N(\cdot)+\theta_N\sup_{\Prob\in\calP_N}\Prob(\cdot)-\Prob^\star(\cdot)\bigg] \Rightarrow \G'(\cdot) \,\,\, \text{in} \,\,\, \ell^\infty(\calH),
\end{equation}
where $\sup_{\Prob\in\calP_N}\Prob(\cdot)\in\calP_N$ denotes the worst-case distribution of the input function $h\in\calH$. Recall that $\calH=\{f(\xb,\cdot)\mid\xb\in\calX\}$. Note that the map from $\ell^\infty(\calH)$ to $\ell^\infty(\calX)$ given by $g(\cdot)\mapsto g(h(\cdot,\cdot))$ is continuous, where $\ell^\infty(\calX)$ is the Banach space of bounded functions $\psi:\calX\rightarrow\R$ equipped with the supremum norm $\norms{\psi}=\sup_{\xb\in\calX} |\psi(\xb)|$. By continuous mapping theorem (see Theorem 1.3.6 of \citealp{van_der_Vaart_Wellner:1996}), \eqref{eqn_pf:thm_asy_dist_3} implies that
\begin{equation} \label{eqn_pf:thm_asy_dist_4}
\sqrt{N}\bigg[(1-\theta_N)\E_{\Probh_N}[f(\cdot,\xib)]+\theta_N\sup_{\Prob\in\calP_N}\E_{\Prob}[f(\cdot,\xib)] - \E_{\Prob^\star}[f(\cdot,\xib)]\bigg] \Rightarrow \G(\cdot) \,\,\, \text{in} \,\,\, \ell^\infty(\calX).
\end{equation}
Consider the functional $V:\ell^\infty(\calX)\rightarrow\R$ by $V(\psi)=\inf_{\xb\in\calX} \psi(\xb)$. Since $\calX$ is compact by Assumption~\ref{assumption:loss_and_ambig_set_boundedness}, the Hadamard directional derivative of $V$ at $\psi$ is given by $V'_\psi(\phi)=\inf_{\xb\in B(\psi)} \phi(\xb)$, where $B(\psi)=\argmin_{\xb\in\calX} \psi(\xb)$ (see, e.g., Corollary 2.2 of \citealp{Carcamo_et_al:2020}).  Thus, together with \eqref{eqn_pf:thm_asy_dist_4}, applying the Delta's method (see, e.g., Theorem 2.2 of \citealp{Carcamo_et_al:2020}), we obtain the desired assertions: (i) $\sqrt{N}(\upsilonh_N-\upsilon^\star) \Rightarrow \inf_{\xb\in \calX^\star} \G(\xb)$ and (ii)
\begin{equation} \label{eqn_pf:thm_asy_dist_5}
    \upsilonh_N-\upsilon^\star = \inf_{\xb\in\calX^\star}  \bigg\{ (1-\theta_N)\E_{\Probh_N}[f(\xb,\xib)] + \theta_N \sup_{\Prob\in\calP_N}\E_{\Prob}[f(\xb,\xib)] - \E_{\Prob^\star}[f(\xb,\xib)] \bigg\} + o_{\Prob^\star}(N^{-1/2}).
\end{equation}
Finally, since $-\infty<\upsilon^\star=\E_{\Prob^\star}[f(\xb,\xib)]$ for any $\xb\in \calX^\star$, \eqref{eqn_pf:thm_asy_dist_5} directly implies $\upsilonh_N= \inf_{\xb\in \calX^\star}\big\{ (1-\theta_N)\E_{\Prob^\star}[f(\xb,\xib)] + \theta_N \sup_{\Prob\in\calP_N}\E_{\Prob}[f(\xb,\xib)]\big\} +o_{\Prob^\star}(N^{-1/2})$. This completes the proof. 
\end{proof}


\subsection{Proof of Theorem~\ref{thm:asy_dist_distance_based}}

\begin{proof}
Using~\eqref{eqn:DRO_expansion}, we can rewrite the objective function of the TRO model with shape parameter $\calP_{N,r_N}$ as follows:
\begin{align*}
    \sup_{\Prob\in\calP'_{N,\theta}} \E_{\Prob}[f(\xb,\xib)]&=(1-\theta_N) \E_{\Probh_N}[f(\xb,\xib)]+\theta_N\sup_{\Prob\in\calP_{N,r_N}} \E_{\Prob}[f(\xb,\xib)] \\
    &=\E_{\Probh_N}[f(\xb,\xib)]+ \big(\theta_N r^\gamma_N\big) g_N(\xb)+  \big(\theta_N r^\gamma_N\big) \varepsilon_N(\xb),
\end{align*}
which resembles the expansion \eqref{eqn:DRO_expansion} with $r^\gamma_N$ replaced by $\theta_N r^\gamma_N$. Thus, we can prove the desired assertions by following the same proof techniques of Theorem~1 in  \cite{Blanchet_Shapiro:2023}. 
\end{proof}
\color{black}

\section{Additional Discussions} \label{apdx:add_discuss}

\subsection{Differences between the Huber Contamination Model and Our TRO Model} \label{apdx:add_discuss:Huber}

As mentioned in Section~\ref{sec:introduction}, the form of our TRO ambiguity set resembles those considered in robust statistics, particularly in Huber contamination models \citep{Huber:1964}. However, the underlying idea of the Huber contamination model fundamentally differs from that of our TRO model. In robust statistics, the Huber contamination model is used to address data contamination or outliers. Specifically,  in this model, the data $\{\xibh_i\}_{i=1}^N$ is assumed to be drawn from a mixture distribution $(1-\varepsilon)\Prob_0+\varepsilon\Q$, where $\Prob_0$ is the distribution of interest (e.g., the true distribution), $\Q\in\calP(\Xi)$ represents some arbitrary distribution, and $\varepsilon$ represents the contamination ratio (see, e.g., \citealp{Chen_et_al:2018, Copas:1988, Huber:1964, Mu_Xiong:2023}). This stream of literature often focuses on developing statistical procedures to estimate the distribution of interest $\Prob_0$. In contrast, as discussed in Section~\ref{sec:introduction}, we consider the case where the true distribution $\Prob^\star$ of $\xib$ is unknown. We assume that one has a (potentially small) set of historical observations $\{\xibh_i\}_{i=1}^N$ of $\xib$ from the unknown true distribution $\Prob^\star$ (i.e., we do not assume that the data is contaminated). Our TRO model is an alternative approach for modeling uncertainty in problem~\eqref{prob:SO} that serves as a middle ground between the optimistic approach, which adopts a distributional belief, and the pessimistic approach, which protects against distributional ambiguity. The TRO ambiguity set $\calP'_{N,\theta} = \big\{ (1-\theta)\Probh_N + \theta\Q \mid \Q\in\calP_N\big\}$  is a key ingredient of our TRO model in \eqref{model:trade-off_model}. The size parameter $\theta\in[0,1]$ controls the trade-off between solving the problem under a distributional belief and solving it under the worst-case distribution that resides in the shape parameter $\calP_N$. Hence, $\theta$ plays a different role in our model than $\varepsilon$ in the Huber contamination model (which quantifies the contamination level in the data).

\subsection{Additional Clarification Related to Remark~\ref{rem:assumption_in_debias_theta_thm}} \label{apdx:add_discuss:remark}

Suppose that the (data-driven) ambiguity set $\calP_{N,\alpha}$ satisfies $\Prob^N(\Prob^\star\in\calP_{N,\alpha})\geq 1-\alpha$ for some small $\alpha\in(0,1)$. In this case, the set $\calP_{N,\alpha}$ potentially consists of a wide range of distributions such that $\E_{\Prob^N}[\upsilonh_N(\theta)\mid \Prob^\star\in\calP_{N,\alpha}]\geq \upsilon^\star+\Delta(\alpha)$ for some $\Delta(\alpha)>0$, i.e., the expected value of  $\upsilonh_N(\theta)$ (conditional on $\Prob^\star\in\calP_{N,\alpha}$) is strictly greater than the true optimal value $\upsilon^\star$. If the function $f$ is bounded from below, say, $f(\xb,\xib)\geq M$ for all $\xb\in\calX$ and $\xib\in\Xi$, then
\begin{align}
    \E_{\Prob^N}[\upsilonh_N(\theta)]&=\Prob^N(\Prob^\star\in\calP_{N,\alpha})\E_{\Prob^N}[\upsilonh_N(\theta)\mid \Prob^\star\in\calP_{N,\alpha}] + \Prob^N(\Prob^\star\not\in\calP_{N,\alpha})\E_{\Prob^N}[\upsilonh_N(\theta)\mid \Prob^\star\not\in\calP_{N,\alpha}] \nonumber \\
    &\geq(1-\alpha)[\upsilon^\star+\Delta(\alpha)]+M \Prob^N(\Prob^\star\not\in\calP_{N,\alpha}). \label{eqn:TRO_mean_bound}
\end{align}
If we pick a sufficiently small $\alpha$, the probability $\Prob^N(\Prob^\star\not\in\calP_{N,\alpha})$ can be arbitrarily small since it satisfies $\Prob^N(\Prob^\star\not\in\calP_{N,\alpha})\leq\alpha$. As $\alpha$ decreases, the value of $\Delta(\alpha)$ is non-decreasing since $\calP_{N,\alpha}$ consists of a larger number of distributions. Thus, the first term in \eqref{eqn:TRO_mean_bound} increases when $\alpha$ decreases. Therefore, we can choose $\alpha$ sufficiently small such that $\E_{\Prob^N}[\upsilonh_N(\theta)]\geq\upsilon^\star$.

\section{An Example of a Sequence of TRO Ambiguity Sets} \label{apdx:example_star_center}

Let $\Probh_N=\delta_0$, i.e., the Dirac measure on $0$, and $\calP_N=\{(1-t)\delta_1 + t\delta_e\mid t\in[0,1],\,e\in\{0,2\}\}$. That is, $\calP_N$ contains the one-point distributions $\{\delta_0,\delta_1,\delta_2\}$, as well as all two-point distributions with support on either $\{0,1\}$ or $\{1,2\}$. Note that $\calP_N$ is star-shaped with a star center $\delta_1$. Indeed, for any $\alpha\in[0,1]$ and $\Q=(1-t)\delta_1 + t\delta_e\in\calP_N$, we have
\begin{equation*}
    (1-\alpha)\delta_1 + \alpha\Q = (1-\alpha)\delta_1 + \alpha \big[(1-t)\delta_1 + t\delta_e\big] = (1-\alpha t)\delta_1 + \alpha t \delta_e\in\calP_N 
\end{equation*}
since $\alpha t \in[0,1]$. However, $\Probh_N=\delta_0$ is not a star center. To see this, note that
$$\frac{1}{2}\delta_ 0 + \frac{1}{2}\bigg(\frac{1}{2}\delta_1 + \frac{1}{2}\delta_2\bigg)=\frac{1}{2}\delta_0+\frac{1}{4}\delta_1+\frac{1}{4}\delta_2\not\in\calP_N$$
since it is a three-point distribution. 

Now, we show that $\calP'_{N,\theta}$ is \textit{not} non-decreasing. In particular, we show that for any $0<\theta_1<\theta_2\leq 1$, there exists $\M\in\calP'_{N,\theta_1}$ but $\M\not\in\calP'_{N,\theta_2}$. Indeed, since $\frac{1}{2}\delta_1+\frac{1}{2}\delta_2\in\calP_N$, we construct the measure $\M\in\calP'_{N,\theta_1}$ as follows:
\begin{equation} \label{eqn:eg:example_star_center}
    \M=(1-\theta_1)\delta_0 + \theta_1\bigg(\frac{1}{2}\delta_1 + \frac{1}{2}\delta_2\bigg).
\end{equation}
We show that $\M$ defined in \eqref{eqn:eg:example_star_center} does not belong to $\calP'_{N,\theta_2}$. Note that 
\begin{align*}
   \M=(1-\theta_1)\delta_0 + \theta_1\bigg(\frac{1}{2}\delta_1 + \frac{1}{2}\delta_2\bigg)&= (1-\theta_2)\delta_0 + (\theta_2-\theta_1)\delta_0 + \theta_1\bigg(\frac{1}{2}\delta_1 + \frac{1}{2}\delta_2\bigg)\\
   &=(1-\theta_2)\delta_0 + \theta_2\Bigg\{\bigg(1-\frac{\theta_1}{\theta_2}\bigg)\delta_0 + \frac{\theta_1}{\theta_2}\bigg(\frac{1}{2}\delta_1 + \frac{1}{2}\delta_2\bigg)\Bigg\}.
\end{align*}
Since $(1-\frac{\theta_1}{\theta_2})\delta_0+\frac{\theta_1}{\theta_2}(\frac{1}{2}\delta_1 + \frac{1}{2}\delta_2)$ is a three-point distribution that does not belong to $\calP_N$, we have $\M\not\in\calP'_{N,\theta_2}$.

\section{Robust Optimization Ambiguity Set} \label{apdx:RO_ambig_set}

\begin{proposition}
For a fixed $\xb\in\calX$, if there exists $\xib_0\in\calU$ such that $f(\xb,\xib_0)\geq f(\xb,\xibh_i)$ for all $i\in\{1,\dots,N\}$, then $\sup_{\xib\in\calU} f(\xb,\xib) = \sup_{\Prob\in\calP_N} \E_{\Prob}[f(\xb,\xib)]$, where $\calP_N=\conv\big( \Probh_N\cup\big\{ \delta_{\xib}\mid \xib\in\calU\big\} \big)$.
\end{proposition}

\begin{proof}
First, if $\sup_{\xib\in\calU} f(\xb,\xib)=\infty$, then there exists a sequence $\{\xib_j\}_{j\in\N}\subset\calU$ such that $f(\xb,\xib_j)\rightarrow\infty$ as $j\rightarrow\infty$. Since $\xib_j\in\calU$, we have $\delta_{\xib_j}\in\calP_N$ and $\E_{\delta_{\xib_j}}[f(\xb,\xib)]=f(\xb,\xib_j)\rightarrow\infty$ as $j\rightarrow\infty$. This shows that $\sup_{\Prob\in\calP_N}\E_{\Prob}[f(\xb,\xib)]=\infty$.

Now, assume that $\sup_{\xib\in\calU} f(\xb,\xib)<\infty$, and thus $f(\xb,\xib)<\infty$ for all $\xib\in\calU$. Note that for any $\Prob\in\calP_N$, by definition of $\calP_N$, there exists $K\in\N$, $\lambda\in[0,1]$, and $\{\alpha_k\}_{k=1}^K\subseteq[0,1]$ with $\alpha_k\geq 0$ and $\sum_{k=1}^K \alpha_k=1$ such that $\Prob=(1-\lambda)\Probh_N+\lambda \sum_{k=1}^K \alpha_i \delta_{\bar{\xib}_k}$ for some $\bar{\xib}_k\in\calU$, $k\in\{1,\dots,K\}$. Then,
\begin{subequations}
\begin{align}
    &\quad\, \sup_{\Prob\in\calP_N}\E_{\Prob}[f(\xb,\xib)]  \nonumber \\
    &= \sup\Bigg\{(1-\lambda)\cdot\frac{1}{N}\sum_{i=1}^N f(\xb,\xibh_i) + \lambda \sum_{k=1}^K \alpha_k f(\xb,\bar{\xib}_k) \,\Bigg|\, \begin{array}{l}
         K\in\N,\, \lambda\in[0,1],\, \alpha_k\in[0,1], \\ \sum_{k=1}^K \alpha_k = 1,\, \bar{\xib}_k\in\calU,\, k\in\{1,\dots,K\}    \end{array} \Bigg\} \label{eqn_pf:prop_RO_ambig_set_1} \\
    &= \sup_{\lambda\in[0,1]}\Bigg\{ (1-\lambda)\cdot\frac{1}{N}\sum_{i=1}^N f(\xb,\xibh_i) + \lambda\cdot\sup \Bigg\{\sum_{k=1}^K \alpha_k f(\xb,\bar{\xib}_k) \,\Bigg|\, \begin{array}{l}
         K\in\N,\, \alpha_k\in[0,1], \sum_{k=1}^K \alpha_k = 1, \\ \bar{\xib}_k\in\calU,\, k\in\{1,\dots,K\}    \end{array} \Bigg\} \Bigg\}  \nonumber\\
    &= \sup_{\lambda\in[0,1]}\Bigg\{ (1-\lambda)\cdot\frac{1}{N}\sum_{i=1}^N f(\xb,\xibh_i) + \lambda \cdot \sup_{\xib\in\calU} f(\xb,\xib)  \Bigg\} \label{eqn_pf:prop_RO_ambig_set_2}\\
    &= \sup_{\xib\in\calU} f(\xb,\xib). \label{eqn_pf:prop_RO_ambig_set_3}
\end{align}
\end{subequations}
Equality \eqref{eqn_pf:prop_RO_ambig_set_1} follows from $\Prob\in\calP_N$ and the definition of $\calP_N$. Equality \eqref{eqn_pf:prop_RO_ambig_set_2} follows from the fact that
\begin{equation} \label{eqn_pf:prop_RO_ambig_set_4}
    \sum_{k=1}^K \alpha_k f(\xb,\bar{\xib}_k) \leq \sup_{\xib\in\calU} f(\xb,\xib)\leq \sup\Bigg\{\sum_{k=1}^K \alpha_k f(\xb,\bar{\xib}_k) \,\Bigg|\, \begin{array}{l}
         K\in\N,\, \alpha_k\in[0,1], \sum_{k=1}^K \alpha_k = 1, \\ \bar{\xib}_k\in\calU,\, k\in\{1,\dots,K\}    \end{array} \Bigg\} 
\end{equation}
for any $k\in\N$, $\alpha_k\in[0,1]$ and $\bar{\xib_k}\in\calU$ for $k\in\{1,\dots,K\}$ with $\sum_{k=1}^K \alpha_k=1$. Here, the first inequality in \eqref{eqn_pf:prop_RO_ambig_set_4} follows from  $f(\xb,\bar{\xib})\leq \sup_{\xib\in\calU} f(\xb,\xib)$ for all $\bar{\xib}\in\calU$ while the second inequality in \eqref{eqn_pf:prop_RO_ambig_set_4} follows from letting $K=1$. Taking supremum over $K$, $\{\alpha_k\}_{k=1}^K$ and $\{\bar{\xib}_k\}_{k=1}^K$ in \eqref{eqn_pf:prop_RO_ambig_set_4}, we obtain the desired equality in \eqref{eqn_pf:prop_RO_ambig_set_2}. Finally, equality~\eqref{eqn_pf:prop_RO_ambig_set_3} follows from                             the fact that $\lambda=1$ is optimal to \eqref{eqn_pf:prop_RO_ambig_set_2} since
$$f(\xb,\xibh_i)\leq f(\xb,\xib_0)\leq \sup_{\xib\in\calU} f(\xb,\xib)$$
for all $i\in\{1,\dots,N\}$, where the first inequality follows from our assumption that $f(\xb,\xib_0)\geq f(\xb,\xibh_i)$ for all $i\in\{1,\dots,N\}$. This shows that $(1/N)\sum_{i=1}^N f(\xb,\xibh_i)\leq \sup_{\xib\in\calU} f(\xb,\xib)$. 
\end{proof}

\section{Convergence of Distance-Based Ambiguity Sets} \label{apdx:distance_ambig_set_convergence}

\begin{proposition} \label{prop:distance_ambig_set_convergence}
Consider the distance-based ambiguity set $\calP_N=\{\Prob\in\calP(\Xi)\mid \sfd(\Prob,\Probh_N)\leq r\}$ for some radius $r>0$, where $\sfd$ is any statistical distance satisfying 
\begin{enumerate}
    \item [(i)] (\textit{normalization}) $\sfd(\Prob,\Prob)=0$ for any $\Prob\in\calP(\Xi)$,
    \item [(ii)](\textit{symmetry}) $\sfd(\Prob_1,\Prob_2)=\sfd(\Prob_2,\Prob_1)$ for any $\{\Prob_1,\Prob_2\}\subseteq\calP(\Xi)$,
    \item [(iii)] (\textit{triangle inequality})  $\sfd(\Prob_1,\Prob_2)\leq\sfd(\Prob_1,\Prob_3)+\sfd(\Prob_3,\Prob_2)$ for any $\{\Prob_1,\Prob_2,\Prob_3\}\subseteq\calP(\Xi)$, and
    \item [(iv)]  (\textit{convexity}) $\sfd$ is convex in the first argument. 
\end{enumerate}
Let $\calPh=\{\Prob\in\calP(\Xi)\mid \sfd(\Prob,\Prob^\star)\leq r\}$. If $\Delta:=\sup_{\Prob_1\in\calP(\Xi),\,\Prob_2\in\calP(\Xi)} \bbmd(\Prob_1,\Prob_2)<\infty$ and $\sfd(\Probh_N,\Prob^\star)\rightarrow0$ almost surely, then $\bbmH(\calP_N,\calPh)\rightarrow 0$ almost surely as $N\rightarrow\infty$.
\end{proposition}

\begin{proof}
The idea of the proof follows from the proof of Hoffman's Lemma for moment problems (see Theorem 2 in \citealp{Liu_et_al:2019}). First, we claim that for any $\Prob_1\in\calP_N$, $(1-\rho)\Prob_1+\rho\Prob^\star\in\calPh$, where $\rho=\sfd(\Probh_N,\Prob^\star)/\big(r+\sfd(\Probh_N,\Prob^\star)\big)$. To see this, note that 
\begin{align}
    \sfd\big((1-\rho)\Prob_1+\rho\Prob^\star,\Prob^\star\big)&\leq (1-\rho)\sfd(\Prob_1,\Prob^\star) \nonumber\\ 
    &\leq \frac{r}{r+\sfd(\Probh_N,\Prob^\star)} \big[\sfd(\Prob_1,\Probh_N)+\sfd(\Probh_N,\Prob^\star)\big]  \nonumber\\
    &\leq \frac{r}{r+\sfd(\Probh_N,\Prob^\star)} \big[r+\sfd(\Probh_N,\Prob^\star)\big] = r, \label{eqn_pf:prop_distance_ambig_set_convergence_1}
\end{align}
where the first inequality follows from properties (iv) and (i), the second inequality follows from the definition of $\rho$ and property (iii). Then, we have
\begin{align}
    \bbmD(\Prob_1,\calPh) = \inf_{\Prob\in\calPh} \bbmd(\Prob_1,\Prob) &\leq \bbmd\big(\Prob_1, (1-\rho)\Prob_1+\rho\Prob^\star\big)  \nonumber\\
    &= \sup_{\xb\in\calX} \bigg| \E_{\Prob_1}[f(\xb,\xib)] - \Big\{ (1-\rho)\E_{\Prob_1}[f(\xb,\xib)] + \rho\E_{\Prob^\star}[f(\xb,\xib)]\Big\} \bigg| \nonumber \\
    &= \rho \sup_{\xb\in\calX} \bigg| \E_{\Prob_1}[f(\xb,\xib)] - \E_{\Prob^\star}[f(\xb,\xib)]\bigg|  \nonumber\\
    &= \frac{\sfd(\Probh_N,\Prob^\star)}{r+\sfd(\Probh_N,\Prob^\star)}\cdot \bbmd(\Prob_1,\Prob^\star) \leq \frac{\Delta}{r} \cdot \sfd(\Probh_N,\Prob^\star), \label{eqn_pf:prop_distance_ambig_set_convergence_2}
\end{align}
where the last inequality follows from the fact that $\sfd(\Probh_N,\Prob^\star)\geq0$ and $\bbmd(\Prob_1,\Prob^\star)\leq\Delta$. This implies that $\sup_{\Prob_1\in\calP_N}  \bbmD(\Prob_1,\calPh) \leq (\Delta/r) \sfd(\Probh_N,\Prob^\star)$. Similarly,
following a similar argument in \eqref{eqn_pf:prop_distance_ambig_set_convergence_1}, we can show that for any $\Prob_2\in\calPh$, $(1-\tau)\Prob_2+\tau\Probh_N\in\calP_N$, where $\tau=\sfd(\Prob^\star,\Probh_N)/\big(r+\sfd(\Prob^\star,\Probh_N)\big)$. A similar argument as in \eqref{eqn_pf:prop_distance_ambig_set_convergence_2} shows that 
$$\bbmD(\Prob_2,\calPh_N) \leq \frac{\Delta}{r}\cdot\sfd(\Prob^\star,\Probh_N)=\frac{\Delta}{r}\cdot\sfd(\Probh_N,\Prob^\star),$$
where the last inequality follows from property (ii). This implies that $\sup_{\Prob_2\in\calPh}  \bbmD(\Prob_2,\calPh_N) \leq (\Delta/r) \sfd(\Probh_N,\Prob^\star)$. Therefore, we have
$$\bbmH(\calP_N,\calPh)=\max\Bigg\{ \sup_{\Prob_1\in\calP_N} \bbmD(\Prob_1,\calPh),\, \sup_{\Prob_2\in\calPh} \bbmD(\Prob_2,\calP_N) \Bigg\} \leq \frac{\Delta}{r}\cdot\sfd(\Probh_N,\Prob^\star) \rightarrow 0$$
almost surely as $N\rightarrow\infty$. 
\end{proof}

\section{Some Quantitative Stability Analysis Results} \label{apdx:known_QSA_DRO}

Consider two DRO models with two different ambiguity sets:
$$\upsilon_i=\inf_{\xb\in\calX}\sup_{\Prob\in\calP_i} \E_{\Prob}[f(\xb,\xib)]$$
for $i\in\{1,2\}$, where we assume that $\upsilon_i$ is finite and
the set of optimal solutions $\calX^\star_i$ is non-empty. In quantitative stability analysis, we analyze how the change in the ambiguity set would affect the optimal value and the set of optimal solutions to the DRO model. In this appendix, we summarize some relevant results in the existing literature, in particular, the upper bounds on the differences between the optimal values and the set of optimal solutions from the two DRO models (see, e.g., \citealp{Liu_Xu:2013, Pichler_Xu:2022, Sun_Xu:2016}). For the sake of completeness, we also provide the proof of these results.

\begin{proposition}\label{prop:known_QSA_Hausdorff_bound}
The (pointwise) absolute difference between the two objective functions is upper bounded by the Hausdorff distance between the two ambiguity sets, i.e., for any $\xb\in\calX$,
\begin{equation} \label{eqn:known_QSA_Hausdorff_bound}
     \bigg|\sup_{\Q\in\calP_1} \E_{\Q}[f(\xb,\xib)] - \sup_{\Q'\in\calP_2} \E_{\Q'}[f(\xb,\xib)]\bigg| \leq \bbmH(\calP_1,\calP_2).
\end{equation}
\end{proposition}

\begin{proof}
By definition of the pseudometric $\bbmd$ in \eqref{eqn:dist_btw_prob_measures}, for any $\Q\in\calP_1$ and $\Q'\in\calP_2$, we have $\big| \E_{\Q}[f(\xb,\xib)] - \E_{\Q'}[f(\xb,\xib)] \big| \leq \bbmd(\Q,\Q')$ for all $\xb\in\calX$. Since 
\begin{align*}
    \sup_{\Q\in\calP_1} \E_{\Q}[f(\xb,\xib)] - \sup_{\Q'\in\calP_2} \E_{\Q'}[f(\xb,\xib)] &= \sup_{\Q\in\calP_1} \inf_{\Q'\in\calP_2} \Big\{  \E_{\Q}[f(\xb,\xib)] - \E_{\Q'}[f(\xb,\xib)]\Big\} \\
    &\leq  \sup_{\Q\in\calP_1} \inf_{\Q'\in\calP_2}  \bbmd(\Q,\Q')
\end{align*}
and
\begin{align*}
    \sup_{\Q'\in\calP_2} \E_{\Q'}[f(\xb,\xib)] -  \sup_{\Q\in\calP_1}  \E_{\Q}[f(\xb,\xib)] &= \sup_{\Q'\in\calP_2} \inf_{\Q\in\calP_1} \Big\{  \E_{\Q'}[f(\xb,\xib)] - \E_{\Q}[f(\xb,\xib)]\Big\} \\
    &\leq \sup_{\Q'\in\calP_2} \inf_{\Q\in\calP_1} \bbmd(\Q,\Q'),
\end{align*}
we obtain
\begin{align}
    \bigg|\sup_{\Q\in\calP_1} \E_{\Q}[f(\xb,\xib)] - \sup_{\Q'\in\calP_2} \E_{\Q'}[f(\xb,\xib)]\bigg| &\leq \max\Bigg\{ \sup_{\Q\in\calP_1} \inf_{\Q'\in\calP_2}  \bbmd(\Q,\Q'),\,   \sup_{\Q'\in\calP_2} \inf_{\Q\in\calP_1} \bbmd(\Q,\Q')\Bigg\} \nonumber\\
    &=\bbmH(\calP_1,\calP_2). \nonumber
\end{align}
This completes the proof.
\end{proof}

\begin{proposition} \label{prop:known_QSA_DRO}
The following assertions hold.
\begin{enumerate}
    \item [(i)] $|\upsilon_1-\upsilon_2| \leq \bbmH(\calP_1,\calP_2)$.
    \item [(ii)] If, in addition, $\sup_{\Prob\in\calP_1} \E[f(\xb,\xib)]$ satisfies the second order growth condition at  $\calX^\star_1$, i.e., there exists $\tau>0$ such that
    $$\sup_{\Prob\in\calP_1} \E[f(\xb,\xib)] \geq \upsilon_1+\tau [d(\xb,\calX^\star_1)]^2$$
    for all $\xb\in\calX$, then 
    $$D(\calX^\star_2,\calX^\star_1)\leq \sqrt{\frac{3}{\tau}\bbmH(\calP_1,\calP_2)}.$$
\end{enumerate}
\end{proposition}

\begin{proof}
Part (i) follows directly from
\begin{align*}
    |\upsilon_1-\upsilon_2| &= \bigg| \inf_{\xb\in\calX} \sup_{\Prob\in\calP_1} \E_{\Prob}[f(\xb,\xib)] - \inf_{\xb\in\calX} \sup_{\Prob\in\calP_2} \E_{\Prob}[f(\xb,\xib)] \bigg| \\ 
    &\leq \sup_{\xb\in\calX} \bigg|\sup_{\Q\in\calP_1} \E_{\Q}[f(\xb,\xib)] - \sup_{\Q'\in\calP_2} \E_{\Q'}[f(\xb,\xib)]\bigg| \\
    &\leq \bbmH(\calP_1,\calP_2),
\end{align*}
where the last inequality follows from Proposition \ref{prop:known_QSA_Hausdorff_bound}. For part (ii), suppose, on the contrary, that 
$$D(\calX^\star_2,\calX^\star_1)=\sup_{\xb\in\calX^\star_2} d(\xb,\calX^\star_1)>\sqrt{\frac{3}{\tau}\bbmH(\calP_1,\calP_2)}.$$
Then, there exists $\xbbar\in\calX^\star_2$ such that $d(\xbbar,\calX^\star_1)>\sqrt{3\tau^{-1}\bbmH(\calP_1,\calP_2)}$. The second order growth condition implies that
\begin{equation} \label{eqn_pf:prop_known_QSA_DRO_1}
    \sup_{\Prob\in\calP_1}\E[f(\xbbar,\xib)] - \upsilon_1 \geq \tau\big[d(\xbbar,\calX^\star_1)\big]^2 > 3\,\bbmH(\calP_1,\calP_2). 
\end{equation}
However, by Proposition \ref{prop:known_QSA_Hausdorff_bound} and part (i), we have
\begin{align*}
    \sup_{\Prob\in\calP_1}\E[f(\xbbar,\xib)] - \upsilon_1 &\leq \bigg\{\sup_{\Prob\in\calP_2}\E[f(\xbbar,\xib)] + \bbmH(\calP_1,\calP_2)\bigg\} - \big[\upsilon_2 - \bbmH(\calP_1,\calP_2) \big] \\
    &= \sup_{\Prob\in\calP_2}\E[f(\xbbar,\xib)] - \upsilon_2 + 2\,\bbmH(\calP_1,\calP_2) \\
    &= 2\,\bbmH(\calP_1,\calP_2),
\end{align*}
where the first equality follows from $\xbbar\in\calX^\star_2$. Thus, this contradicts with \eqref{eqn_pf:prop_known_QSA_DRO_1}. 
\end{proof}

\section{Details of Numerical Experiments} \label{apdx:num_expt}

\subsection{Inventory Control -- Reformulations} \label{apdx:num_expt_IC_reform}

Recall the TRO model \eqref{eqn:inventory_contol_trade_off} under shape parameters (a)--(d) in Table~\ref{table:inventory_control_reformulations}. This model cannot be solved directly because of the inner supremum problem $\sup_{\Prob\in\calP_N} \E_{\Prob}\big[(\xi-x)_+-\xi\big]$. In this section, we provide tractable reformulation to this inner supremum problem, and thus our TRO model.

First, consider ambiguity set (a) in Table \ref{table:inventory_control_reformulations}. Using the worst-case distribution derived in \cite{Gallego_Moon:1993}, we have
$$
\sup_{\Prob\in\calP_N} \E_{\Prob}\big[(\xi-x)_+-\xi\big]=
\begin{cases}
     \displaystyle \frac{1}{2}\bigg[-(x-\muh_N)+\sqrt{\sigmah_N^2 + (x-\muh_N)^2}\bigg]-\muh_N & \text{if }\,  \displaystyle x\geq \frac{\sigmah_N^2+\muh_N^2}{2\muh_N}, \vspace{2mm}\\ 
     \displaystyle -\frac{\muh_N^2}{\sigmah_N^2+\muh_N^2}x & \text{if } \, \displaystyle 0\leq x\leq \frac{\sigmah_N^2+\muh_N^2}{2\muh_N}.
\end{cases}
$$
Therefore, the TRO model under set (a) is equivalent to
\begin{align*}
    \underset{x\geq 0}{\text{minimize}}\,\, &(c-h)x + (p-h) \Bigg\{ (1-\theta) \frac{1}{N}\sum_{i=1}^N \big[(\xih_i-x)_+ -\xih_i\big] \\ 
    &\quad+ \theta\Bigg[\bigg\{\frac{1}{2}\bigg[\sqrt{\sigmah_N^2 + (x-\muh_N)^2}-(x-\muh_N)\bigg]-\muh_N\bigg\}\cdot \one_{x\geq \frac{\muh_N^2+\sigmah_N^2}{2\muh_N}} -\frac{\muh_N^2}{\muh_N^2+\sigmah_N^2}x\cdot \one_{0\leq x\leq \frac{\muh_N^2+\sigmah_N^2}{2\muh_N}}\Bigg]\Bigg\}.
\end{align*}

Next, consider ambiguity set (b) in Table \ref{table:inventory_control_reformulations}. Applying the reformulation in \cite{Lee_et_al:2021}, we have
$$\sup_{\Prob\in\calP_N} \E_{\Prob}\big[(\xi-x)_+-\xi\big]=   r + \frac{1}{N}\sum_{i=1}^N \big[(\xih_i-x)_+-\xih_i\big].$$
Therefore, the TRO model under set (b) is equivalent to
$$\underset{x\geq 0}{\text{minimize}}\quad (c-h)x + (p-h) \bigg\{ (1-\theta) \frac{1}{N}\sum_{i=1}^N \big[(\xih_i-x)_+ -\xih_i\big] + \theta\bigg\{  r + \frac{1}{N}\sum_{i=1}^N \big[(\xih_i-x)_+-\xih_i\big]\bigg\}\bigg\}.$$

Now, consider ambiguity set (c) in Table \ref{table:inventory_control_reformulations}. By the strong duality result for $\phi$-divergence DRO problems in \cite{Bayraksan_Love:2015}, we have
$$\sup_{\Prob\in\calP_N} \E_{\Prob}\big[(\xi-x)_+-\xi\big]= \inf_{\lambda\geq 0,\, \tau} \Bigg\{ \tau+ \lambda r - \frac{\lambda}{N}\sum_{i=1}^N \log\bigg( 1-\frac{(\xih_i-x)_+-\xih_i-\tau}{\lambda} \bigg)\Bigg\},$$
where $-0\log(1-s/0)=0$ for $s\leq 0$ and $-0\log(1-s/0)=\infty$ for $s>0$.
Therefore, the TRO model under set (c) is equivalent to
\begin{align*}
    \underset{x\geq 0,\,\lambda\geq 0,\,\tau}{\text{minimize}}\quad & (c-h)x + (p-h) \bigg\{ (1-\theta) \frac{1}{N}\sum_{i=1}^N \big[(\xih_i-x)_+ -\xih_i\big] \\
    &\quad + \theta\bigg\{  \tau+ \lambda r - \frac{\lambda}{N}\sum_{i=1}^N \log\bigg( 1-\frac{(\xih_i-x)_+-\xih_i-\tau}{\lambda} \bigg)\bigg\}\bigg\}.
\end{align*}

Finally, consider the ambiguity set (d) in Table \ref{table:inventory_control_reformulations}. For notational simplicity, let $\gamma := t_{N-1,\alpha/2} \sigmah_N/\sqrt{N}$. Then, we have
$$\sup_{\Prob\in\calP_N} \E_{\Prob}\big[(\xi-x)_+-\xi\big]= \sup_{\xi\geq 0: |\xi-\muh_N|\leq \gamma} \big\{(\xi-x)_+-\xi\big\}= \max\big\{-x, -(\muh_N -\gamma)_+ \big\},$$
where the last equality follows from the fact that $(\xi-x)_+-\xi=\max\{-x,-\xi\}$ and the objective function is non-increasing in $\xi$. Thus, the supremum over $\xi\in[(\muh_N-\gamma)_+,\muh_N+\gamma]$ is attained at the lower bound. Therefore, the TRO model under set (d) is equivalent to
$$\underset{x\geq 0}{\text{minimize}}\quad (c-h)x + (p-h) \bigg\{ (1-\theta) \frac{1}{N}\sum_{i=1}^N \big[(\xih_i-x)_+ -\xih_i\big] + \theta\max\big\{-x, -(\muh_N -\gamma)_+ \big\}\bigg\}.$$

\subsection{Inventory Control -- Ambiguity Set (d)} \label{apdx:num_expt_IC_check_assumption}

We show that ambiguity set (d) in Table~\ref{table:inventory_control_reformulations}, i.e., $\calP_N=\{\delta_{\xi}\mid |\xi-\muh_N| \leq t_{N-1,\alpha/2}\sigmah_N/\sqrt{N}\}$, satisfies Assumption \ref{assumption:ambig_set_regularity}(b). Recall that $f(x,\xi)=(\xi-x)_+-\xi$, which is Lipschitz in $\xi$ with modulus $1$. We claim that $\bbmH(\calP_N,\calPh)\rightarrow0$ as $N\rightarrow\infty$ almost surely, where $\calPh=\{\delta_{\mu^\star}\}$ with $\mu^\star=\E_{\Prob^\star}(\xi)$. Indeed,
\begin{subequations}
\begin{align}
    \bbmH(\calP_N,\calPh)=\sup_{\Prob\in\calP_N} \bbmD(\Prob,\calPh)&=\sup_{\xi\geq 0:|\xi-\muh_N|\leq t_{N-1,\alpha/2}\sigmah_N/\sqrt{N}} \,\, \sup_{x\geq 0} \big|f(x,\xi)-f(x,\mu^\star)| \label{eqn:inventory_control_ambig_set_d_1} \\
    &\leq \sup_{\xi\geq 0:|\xi-\muh_N|\leq t_{N-1,\alpha/2}\sigmah_N/\sqrt{N}} |\xi - \mu^\star| \label{eqn:inventory_control_ambig_set_d_2} \\
    &= \max\Bigg\{ \bigg|\bigg(\muh_N-t_{N-1,\alpha/2}\frac{\sigmah_N}{\sqrt{N}}\bigg)_+-\mu^\star\Bigg|,\, \bigg|\muh_N+t_{N-1,\alpha/2}\frac{\sigmah_N}{\sqrt{N}}-\mu^\star\Bigg|\Bigg\} \label{eqn:inventory_control_ambig_set_d_3} \\
    &\rightarrow 0 \nonumber
\end{align}
\end{subequations}
almost surely. Here, \eqref{eqn:inventory_control_ambig_set_d_1} follows from the definitions in \eqref{def:Hausdorff_distance} and \eqref{eqn:dist_btw_prob_measures}, and that $\sup_{\Prob\in\calPh} \bbmD(\Prob,\calPh) \leq \sup_{\Prob\in\calP_N} \bbmD(\Prob,\calPh)$; \eqref{eqn:inventory_control_ambig_set_d_2} follows from the Lipschitz continuity of $f(x,\xi)$ in $\xi$; \eqref{eqn:inventory_control_ambig_set_d_3} follows from the fact that the supremum over a compact interval of the absolute value function $|\xi-\mu^\star|$ is attained either at the lower or the upper bound of the interval. Finally, by the strong law of large numbers, $|\muh_N-\mu^\star|\rightarrow0$ and $|\sigmah_N-\sigma^\star|\rightarrow0$ almost surely, where $\sigma^\star=\Var_{\Prob^\star}(\xi)<\infty$. Also, $t_{N-1,\alpha/2}\rightarrow z_{\alpha/2}<\infty$, where $z_{\alpha/2}$ is the upper $(1-\alpha)/2$-th quantile of a standard normal distribution. Thus, we have \eqref{eqn:inventory_control_ambig_set_d_3} converges to zero almost surely.  Thus, ambiguity set (d) satisfies Assumption \ref{assumption:ambig_set_regularity}(b).

\subsection{Portfolio Optimization -- Reformulations} \label{apdx:num_expt_PO_reform}

Recall the TRO model \eqref{eqn:mean-CVaR_model_linear_TRO} under shape paraemters (a)--(c) in Table~\ref{table:portfolio_optimization_reformulations}. Again, this model cannot be solved directly because of the inner supremum problem $\sup_{\Prob\in\calP_N}\E_{\Prob}[f(\xb,t,\xib)]$. In this section, we provide tractable reformulation to this inner supremum problem, and thus our TRO model.

First, consider ambiguity set (a) in Table \ref{table:portfolio_optimization_reformulations}. By Lemma~2.2 and Lemma~2.4 of \cite{Chen_et_al:2011}, we have 
$$\sup_{\Prob\in\calP_N}\E_{\Prob}[f(\xb,t,\xib)]=(1-\beta)t-\beta\xb^\tp\muh_N+\frac{1-\beta}{1-\alpha}\cdot\frac{1}{2}\bigg[-\xb^\tp\widehat{\mub}_N-t+\sqrt{\xb^\tp\widehat{\Sigmab}_N\xb+(-\xb^\tp\widehat{\mub}_N-t)^2} \bigg].$$
Therefore, the TRO model under set (a) is equivalent to
\begin{align*}
    \underset{\xb\in\calX,\,t\in\R}{\text{minimize}}\quad & (1-\theta)\cdot \frac{1}{N}\sum_{i=1}^N\Bigg\{ (1-\beta)t+\beta (-\xb^\tp\xibh_i) + \frac{1-\beta}{1-\alpha}(-\xb^\tp\xibh_i-t)^+\Bigg\} \\
    &\quad + \theta\bigg\{(1-\beta)t-\beta\xb^\tp\muh_N+\frac{1-\beta}{1-\alpha}\cdot\frac{1}{2}\bigg[-\xb^\tp\widehat{\mub}_N-t+\sqrt{\xb^\tp\widehat{\Sigmab}_N\xb+(-\xb^\tp\widehat{\mub}_N-t)^2} \bigg]\bigg\}.
\end{align*}

Now, consider ambiguity set (b) in Table  \ref{table:portfolio_optimization_reformulations}. By Theorem 6.3 of \cite{Mohajerin-Esfahani_Kuhn:2018}, we have
$$\sup_{\Prob\in\calP_N}\E_{\Prob}[f(\xb,t,\xib)]=(1-\beta)t + r\bigg(\beta + \frac{1-\beta}{1-\alpha}\bigg)\norms{\xb}_\infty + \frac{1}{N}\sum_{i=1}^N \bigg[\beta(-\xb^\tp\xibh_i) +\frac{1-\beta}{1-\alpha}(-\xb^\tp\xibh_i-t)_+\bigg].$$
Therefore, the TRO model under set (b) is equivalent to
\begin{align*}
    \underset{\xb\in\calX,\,t\in\R}{\text{minimize}}\quad & (1-\theta)\cdot \frac{1}{N}\sum_{i=1}^N\Bigg\{ (1-\beta)t+\beta (-\xb^\tp\xibh_i) + \frac{1-\beta}{1-\alpha}(-\xb^\tp\xibh_i-t)^+\Bigg\} \\
    &\quad + \theta\bigg\{ (1-\beta)t + r\bigg(\beta + \frac{1-\beta}{1-\alpha}\bigg)\norms{\xb}_\infty + \frac{1}{N}\sum_{i=1}^N \bigg[\beta(-\xb^\tp\xibh_i) +\frac{1-\beta}{1-\alpha}(-\xb^\tp\xibh_i-t)_+\bigg]\bigg\}.
\end{align*}

Finally, consider ambiguity set (c). Note that ambiguity set (c) is characterized by the $\ell_1$ norm constraint. Thus, as in Lemma~3.1 of \cite{Huang_et_al:2021}, using linear programming duality, $\sup_{\Prob\in\calP_N}\E_{\Prob}[f(\xb,t,\xib)]$ is equivalent to
\begin{subequations}
\begin{align}
 \underset{\tau\in\R,\,\lambda\in\R}{\text{minimize}} & \quad  (1-\beta)t + \tau+ r\lambda+\frac{1}{N}\sum_{i=1}^N(u_i^+ - u_i^-)\\
 \text{subject to} \hspace{0mm} & \quad \tau \geq \beta(-\xb^\tp\xibh_i) + \frac{1-\beta}{1-\alpha}\big(-\xb^\tp\xibh_i-t\big)_+ - u_i^+ + u_i^-,\quad\forall i\in\{1,\dots,n\}, \\
 & \quad \lambda\geq u_i^+ + u_i^-,\quad\forall i\in\{1,\dots,n\}, \\
 & \quad \lambda \geq 0,\, u_i^+\geq 0,\, u_i^- \geq 0,\quad\forall i\in\{1,\dots,n\}.
\end{align} %
\end{subequations}
Therefore, the TRO model under set (c) is equivalent to
\begin{subequations}
\begin{align}
\underset{\xb\in\calX,\,t,\,\tau,\,\lambda,\,u^+,\,u^-}{\text{minimize}} & \quad  (1-\theta)\cdot \frac{1}{N}\sum_{i=1}^N\Bigg\{ (1-\beta)t+\beta (-\xb^\tp\xibh_i) + \frac{1-\beta}{1-\alpha}(-\xb^\tp\xibh_i-t)^+\Bigg\} \nonumber \\
&\qquad \theta\Bigg\{(1-\beta)t + \tau+ r\lambda+\frac{1}{N}\sum_{i=1}^N(u_i^+ - u_i^-)\Bigg\}\\
\text{subject to} \hspace{5mm} & \quad \tau \geq \beta(-\xb^\tp\xibh_i) + \frac{1-\beta}{1-\alpha}\big(-\xb^\tp\xibh_i-t\big)_+ - u_i^+ + u_i^-,\quad\forall i\in\{1,\dots,n\}, \\
& \quad \lambda\geq u_i^+ + u_i^-,\quad\forall i\in\{1,\dots,n\}, \\
& \quad \lambda \geq 0,\, u_i^+\geq 0,\, u_i^- \geq 0,\quad\forall i\in\{1,\dots,n\}.
\end{align} %
\end{subequations}

\subsection{Portfolio Optimization -- Lipschitz Continuity} \label{apdx:num_expt_PO_Lip}

In this section, we show that the function  $f(\xb,t,\xib)=(1-\beta)t+\beta (-\xb^\tp\xib) + [(1-\beta)/(1-\alpha)](-\xb^\tp\xib-t)^+$ is Lipschitz continuous in $(\xb,t)$. Indeed, for any $\{(\xb_1,t_1),(\xb_2,t_2)\}\subseteq \calX\times\R$, we have
%
\begin{subequations}
\begin{align}
   &\quad\,\,|f(\xb_1,t_1,\xib)-f(\xb_2,t_2,\xib)|\nonumber\\
   &=\Bigg|\bigg[(1-\beta)t_1+\beta (-\xb_1^\tp\xib) + \frac{1-\beta}{1-\alpha}(-\xb_1^\tp\xib-t_1)^+\bigg]-\bigg[(1-\beta)t_2+\beta (-\xb_2^\tp\xib) + \frac{1-\beta}{1-\alpha}(-\xb_2^\tp\xib-t_2)^+\bigg]\Bigg|\nonumber\\
   &\leq (1-\beta)\big|t_1-t_2\big|+\beta\big|\xb_1^\tp\xib - \xb_2^\tp\xib\big|+\frac{1-\beta}{1-\alpha}\Big|(-\xb_1^\tp\xib-t_1)^+-(-\xb_2^\tp\xib-t_2)^+\Big|\nonumber\\
   &\leq (1-\beta)\big|t_1-t_2\big|+\beta\big|\xb_1^\tp\xib - \xb_2^\tp\xib\big|+\frac{1-\beta}{1-\alpha}\Big( \big|t_1-t_2\big|+\big|\xb_1^\tp\xib - \xb_2^\tp\xib\big|\Big) \label{eqn:expt_PO_Lips_1}\\
   &\leq (1-\beta)\bigg(1+\frac{1}{1-\alpha}\bigg) |t_1-t_2| + \frac{1-\alpha\beta}{1-\alpha}\norms{\xib}_1 \norms{\xb_1-\xb_2}_\infty \label{eqn:expt_PO_Lips_2} \\
   &\leq \Bigg[(1-\beta)\bigg(1+\frac{1}{1-\alpha}\bigg)  + \frac{1-\alpha\beta}{1-\alpha}\norms{\xib}_1\Bigg] \Big(|t_1-t_2|+ \norms{\xb_1-\xb_2}_\infty\Big) \nonumber \\
   &\leq  \Bigg[(1-\beta)\bigg(1+\frac{1}{1-\alpha}\bigg)  + \frac{1-\alpha\beta}{1-\alpha}\norms{\xib}_1\Bigg]\norm{\begin{pmatrix} \xb_1-\xb_2 \\ t_1-t_2 \end{pmatrix}}_1=:\kappa(\xib)\,\norm{\begin{pmatrix} \xb_1-\xb_2 \\ t_1-t_2 \end{pmatrix}}_1 \label{eqn:expt_PO_Lips_3}.
\end{align}
\end{subequations}
Here, \eqref{eqn:expt_PO_Lips_1} follows from the fact that $|(a_1)^+-(a_2)^+|\leq |a_1-a_2|$ for any $\{a_1,a_2\}\subset\R$; \eqref{eqn:expt_PO_Lips_2} follows from $\big|\xb_1^\tp\xib - \xb_2^\tp\xib\big|\leq\norms{\xib}_1 \norms{\xb_1-\xb_2}_\infty$; the inequality in \eqref{eqn:expt_PO_Lips_3} follows from $\norms{\xb_1-\xb_2}_\infty\leq\norms{\xb_1-\xb_2}_1$.

\section{Additional Results} \label{apdx:add_expt_results}

In this section, we provide additional experiment results. Specifically, in Section~\ref{apdx:add_expt_results:distance_based_DRO}, we analyze the spectra of solutions obtained from the DRO model with distance-based ambiguity sets and those obtained from our TRO model with the TRO ambiguity set constructed using these distance-based ambiguity sets as its shape parameter. In Section~\ref{apdx:add_expt_results:bias_variance}, we analyze the bias-variance trade-off of the TRO estimator. In Section~\ref{apdx:add_expt_results:conservatism}, we analyze the conservatism of our TRO model. In Section~\ref{apdx:add_expt_results:out_of_sample}, we analyze the out-of-sample performance of our TRO model.

\subsection{Spectra of Solutions Obtained from Distance-Based DRO Models} \label{apdx:add_expt_results:distance_based_DRO}

As noted by a reviewer of our paper, one could obtain a spectrum of optimal solutions to DRO models equipped with distance-based ambiguity sets by changing the radii of such sets.  In this section, we compare the spectra of solutions obtained from the DRO model with distance-based ambiguity sets, namely the 1-Wasserstein ambiguity set and the total variation ambiguity set (i.e., set (b) and set (c) in Table~\ref{table:portfolio_optimization_reformulations}), and those obtained from our TRO model with TRO ambiguity set constructed using these distance-based ambiguity sets as its shape parameter. For illustrative purposes, we focus on the portfolio optimization problem discussed in Section~\ref{subsec:portfolio_optimization_problem}.  We follow the same experimental settings detailed in Section~\ref{subsec:portfolio_optimization_problem} to obtain the spectra of optimal solution to the DRO and TRO models for this problem.  Specifically, we solve the TRO model with different size parameters ($\theta\in\{0,0.01,0.02,\dots,1\}$), the 1-Wasserstein DRO model with different radii ($r\in\{0, 0.002, 0.004, \dots, 0.1\}$), and the total variation DRO model with different radii ($r\in\{0, 0.002, 0.004, \dots, 0.2\}$).
\begin{figure}
    \centering
    \includegraphics[scale=0.7]{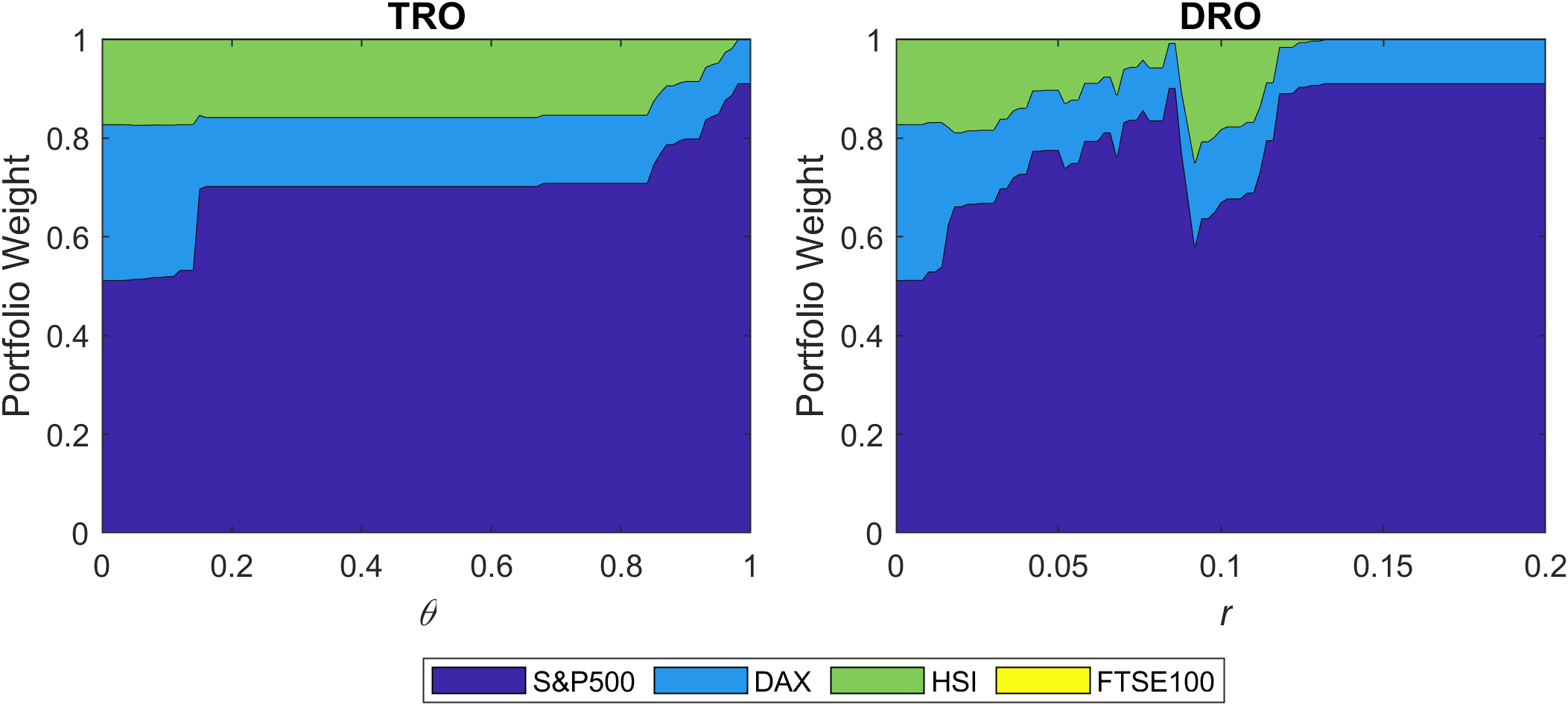}
    \caption{Spectrum of optimal solutions from the TRO model with TRO ambiguity set constructed using total variation ambiguity set as its shape parameter and the total variation DRO model in the portfolio optimization problem.}
    \label{fig:opt_sol_change_vs_DRO_TV_PO}
\end{figure}
\begin{figure}
    \centering
    \includegraphics[scale=0.7]{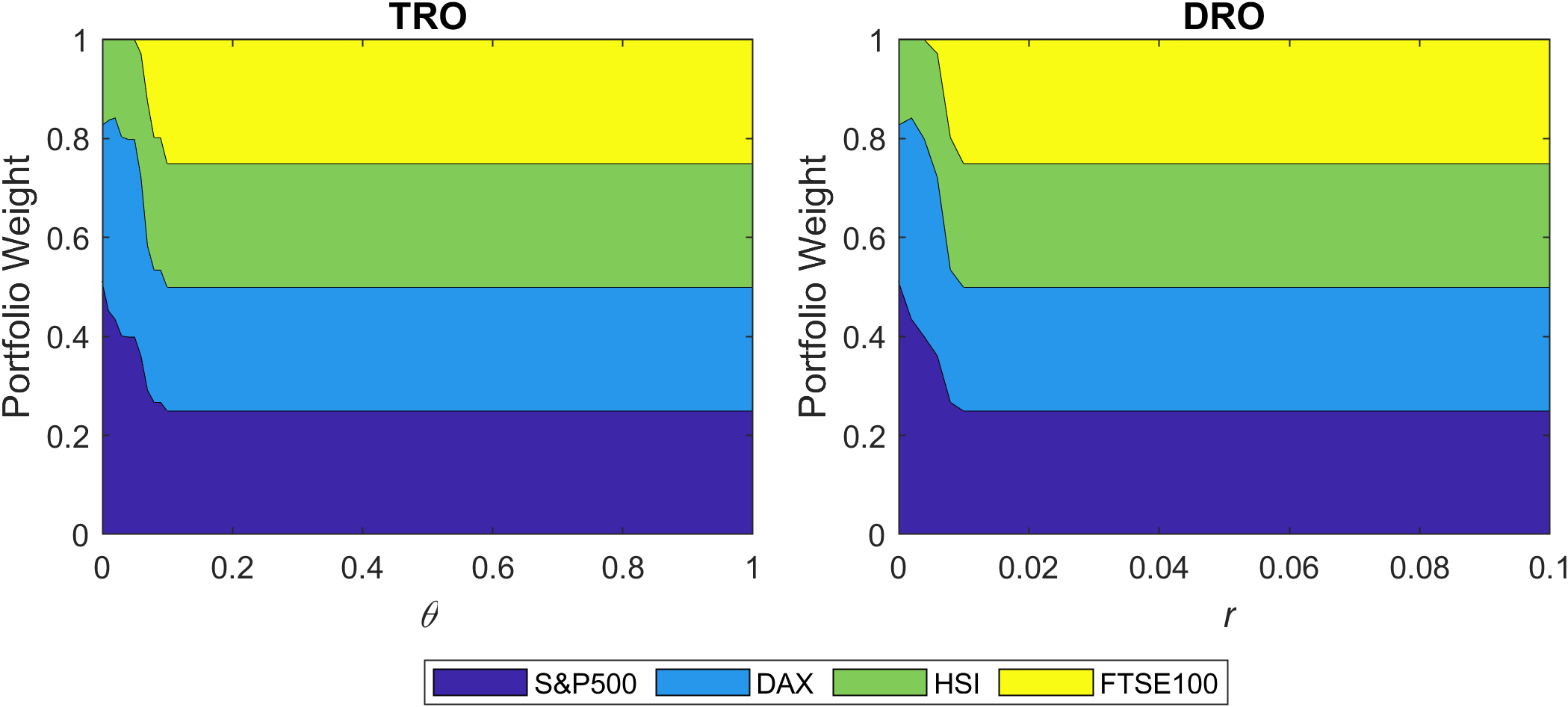}
    \caption{Spectrum of optimal solutions from the TRO model with TRO ambiguity set constructed using 1-Wasserstein ambiguity set as its shape parameter and the 1-Wasserstein DRO model in the portfolio optimization problem.}
    \label{fig:opt_sol_change_vs_DRO_Wass_PO}
\end{figure}

Figure~\ref{fig:opt_sol_change_vs_DRO_TV_PO} presents the spectra of solutions obtained from the TRO model employing the total variation ambiguity set as the shape parameter in the TRO ambiguity set and the DRO model with the total variation ambiguity set. Clearly, the spectra of TRO and DRO solutions could be different. Specifically, the portfolio weights change gradually as the TRO ambiguity set's size parameter $\theta$ increases. In contrast, the shape of the portfolio weights changes when the radius $r$ of the total variation set increases from $0.086$ to $0.092$.

Figure~\ref{fig:opt_sol_change_vs_DRO_Wass_PO} presents the spectra of solutions obtained from the TRO model employing the 1-Wasserstein ambiguity set as the shape parameter of the TRO ambiguity set and the DRO model with  1-Wasserstein ambiguity set. In contrast to the spectra obtained with the total variation ambiguity set, the spectra of the TRO and DRO models with the 1-Wasserstein set are approximately the same.  Specifically, the portfolio weights of S\&P500, DAX, and HSI decrease to $0.25$, and the portfolio weight of FTSE~100 increases to $0.25$ when the TRO ambiguity set's size parameter $\theta$ increases from $0$ to $0.1$ or when the radius $r$ in the DRO model increases from $0$ to $0.01$.

These results suggest that the spectra obtained from our TRO model with a TRO ambiguity set characterized by a distance-based shape parameter and the DRO model equipped with that shape parameter could be different for some choices of the statistical distance in the distance-based shape parameter. Finally, we note that it is not possible to obtain a spectrum of optimal solutions to the DRO model with other ambiguity sets, such as the mean-variance ambiguity set since such sets do not have a parameter that allows for controlling the conservatism. In contrast, our TRO model with TRO ambiguity set constructed using mean-variance ambiguity set as its shape parameter will enable decision-makers to explore a spectrum of solutions, ranging from optimistic to conservative solutions (see results in Section~\ref{subsec:portfolio_optimization_problem}). As discussed in Section~\ref{sec:conclusion}, we leave investigating the characteristics of the spectra of TRO optimal solutions of the TRO model under different shape parameters for future work.

\subsection{Bias-Variance Trade-off} \label{apdx:add_expt_results:bias_variance}

In this section, we numerically analyze the bias-variance trade-off of the TRO estimator $\upsilonh_N(\theta)$.  We  follow the same experimental settings detailed in Sections~\ref{subsec:inventory_control_problem} and~\ref{subsec:portfolio_optimization_problem} to estimate the bias,  $\texttt{bias}_N(\theta)$, and standard deviation, $\texttt{std}_N(\theta)$, of  $\upsilonh_N(\theta)$ for $\theta\in\Theta:=\{0,0.01,0.02,\dots,1\}$ with $N=10$. Figures~\ref{fig:bias_variance_tradeoff_IC} and \ref{fig:bias_variance_tradeoff_PO} present the bias-variance curves for the TRO model of the inventory control problem and portfolio optimization problem with TRO ambiguity set constructed using different shape parameters $\calP_N$. The  $(x, y)$ values of each point on each curve correspond to $(\texttt{std}_N(\theta),\texttt{bias}_N(\theta))$ for some $\theta\in[0,1]$. The black dot corresponds to  ($\texttt{std}_N(0)$, $\texttt{bias}_N(0)$), i.e., the standard deviation and bias of the SAA estimator.
\begin{figure}[t!]
    \centering
    \begin{subfigure}[t]{0.5\textwidth}
        \centering
        \includegraphics[scale=0.75]{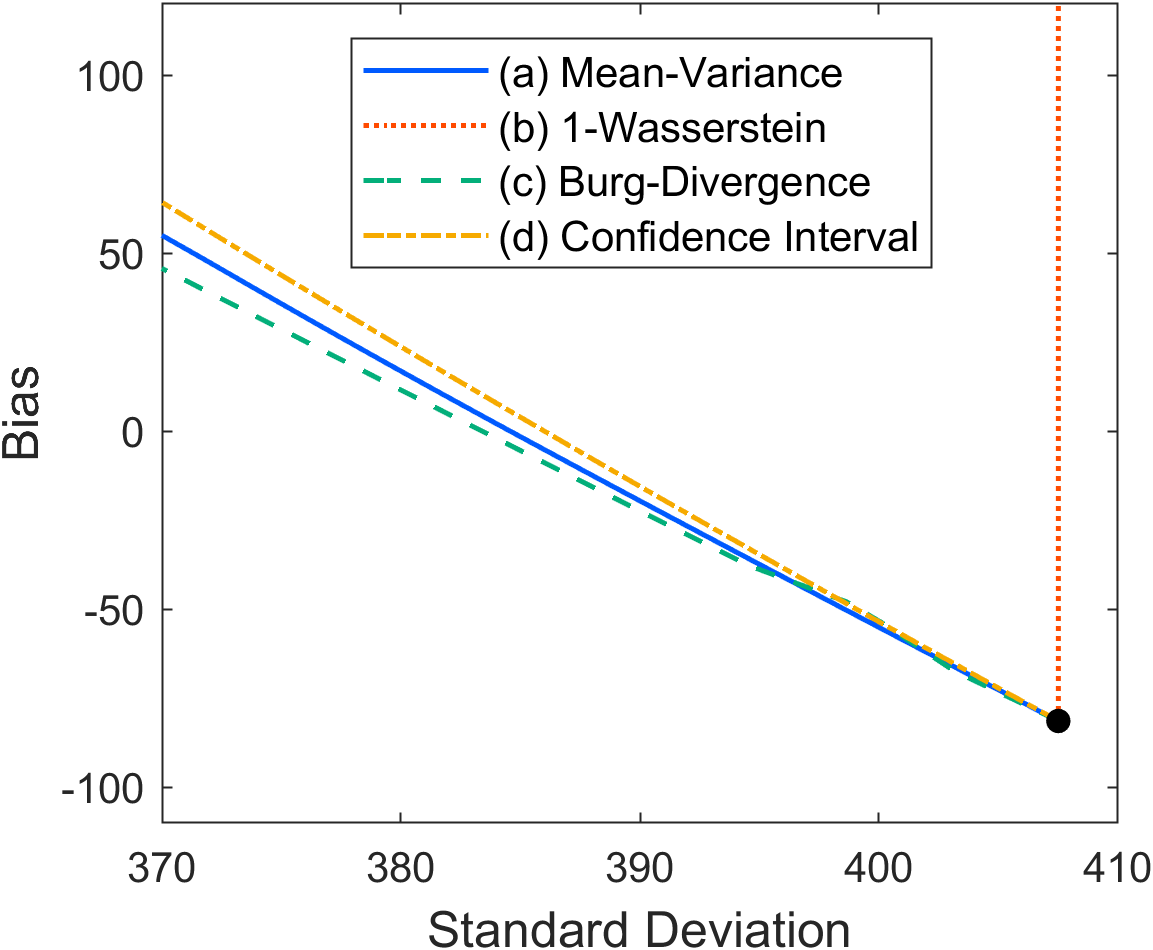}
        \caption{Inventory Control}  \label{fig:bias_variance_tradeoff_IC}
    \end{subfigure}%
    ~ 
    \begin{subfigure}[t]{0.5\textwidth}
        \centering
        \includegraphics[scale=0.75]{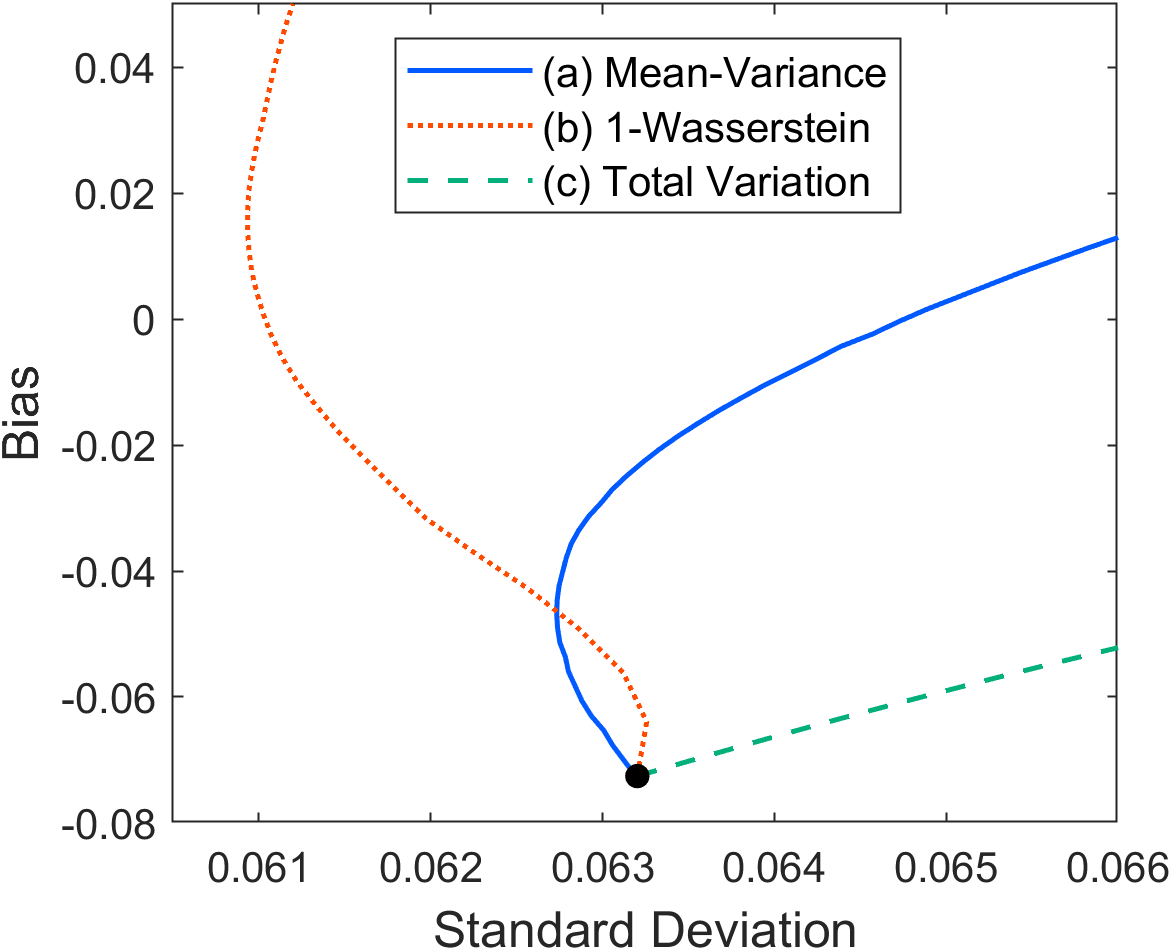}
        \caption{Portfolio Optimization}  \label{fig:bias_variance_tradeoff_PO}
    \end{subfigure}
    \caption{The bias-variance trade-off characterized by $\big(\texttt{std}_N(\theta),\,\texttt{bias}_N(\theta)\big)$ for different values of $\theta$ with $N=10$. 
    The black dot corresponds to  ($\texttt{std}_N(0)$, $\texttt{bias}_N(0)$), i.e., the standard deviation and bias of the SAA estimator.} \label{fig:bias_variance_tradeoff}
\end{figure}

Not surprisingly, the bias and standard deviation vary across different shape parameters employed in the TRO model. For the inventory control problem (see Figure~\ref{fig:bias_variance_tradeoff_IC}), employing set (b) as the shape parameter results in TRO estimators with constant standard deviation but varying bias across different $\theta$. In contrast, using sets (a), (c), or (d) as shape parameters produces TRO estimators that achieve both lower bias and standard deviation compared with the SAA estimator. This indicates that employing sets (a), (c), or (d) as shape parameters offers a superior bias-variance trade-off for this problem. We also observe that none of the shape parameters consistently produce a TRO estimator with the best bias-variance trade-off. For example, when all estimators have biases between $-62$ and $-42$, the TRO estimator using set (a) has the smallest standard deviation. When all estimators have biases less than $-62$ or greater than $-42$,  the TRO estimator using set (c) has the smallest standard deviation.

For the portfolio optimization problem (see Figure~\ref{fig:bias_variance_tradeoff_PO}), we again observe that none of the shape parameters consistently produce a TRO estimator with the best bias-variance trade-off. Specifically, when all estimators have biases less than $-0.0464$, the TRO estimator using set (a) has the smallest standard deviation. When all estimators have biases greater than $-0.0464$, the TRO estimator using set (b) has the smallest standard deviation. Additionally, some shape parameters result in TRO estimators with smaller standard deviations than the SAA estimator. For example, sets (a) and (b) produce TRO estimators with smaller standard deviations for certain values of $\theta$. In contrast, set (c) yields TRO estimators with smaller biases but larger standard deviations than the SAA estimator.

\subsection{Conservatism of TRO Solutions} \label{apdx:add_expt_results:conservatism}

In this section, we investigate the conservatism of the spectra of optimal solutions to the TRO model. Similar to prior studies (see, e.g., \cite{Liu_et_al:2022, Yin_et_al:2023}), we employ the measure $T_N(\theta)=\upsilonh_N(\theta)-\E_{\Prob^\star}[f(\xb_N(\theta),\xib)]$ to quantify conservatism,  where $\upsilonh_N(\theta)$ and $\xb_N(\theta)$ are the optimal value and optimal solution to the TRO model with size parameter $\theta$. This measure captures the extent to which the estimated optimal value (cost), $\upsilonh_N(\theta)$, exceeds the actual expected objective function value (cost), $\E_{\Prob^\star}[f(\xb_N(\theta),\xib)]$, associated with implementing the optimal solution $\xb_N(\theta)$. A larger positive value of $T_N(\theta)$ indicates that the estimated cost $\upsilonh_N(\theta)$ exceeds the true expected cost, suggesting that the solution $\xb_N(\theta)$ may be overly conservative. Conversely, a large negative value of $T_N(\theta)$, with a large absolute value, indicates that the actual cost is underestimated, suggesting that $\xb_N(\theta)$ is overly optimistic.
\begin{figure}[t!]
    \centering
    \begin{subfigure}[t]{0.5\textwidth}
        \centering
        \includegraphics[scale=0.75]{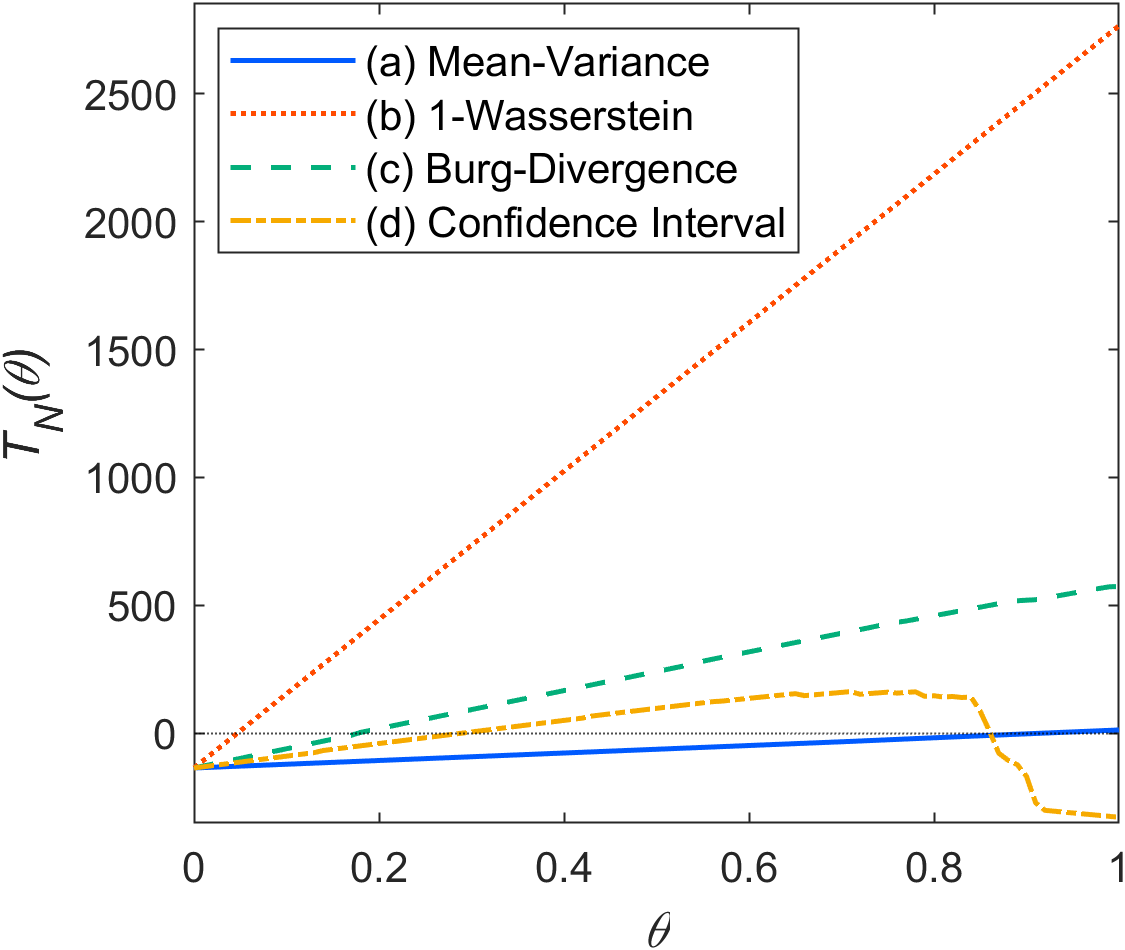}
        \caption{Inventory Control}  \label{fig:conservatism_IC}
    \end{subfigure}%
    ~ 
    \begin{subfigure}[t]{0.5\textwidth}
        \centering
        \includegraphics[scale=0.75]{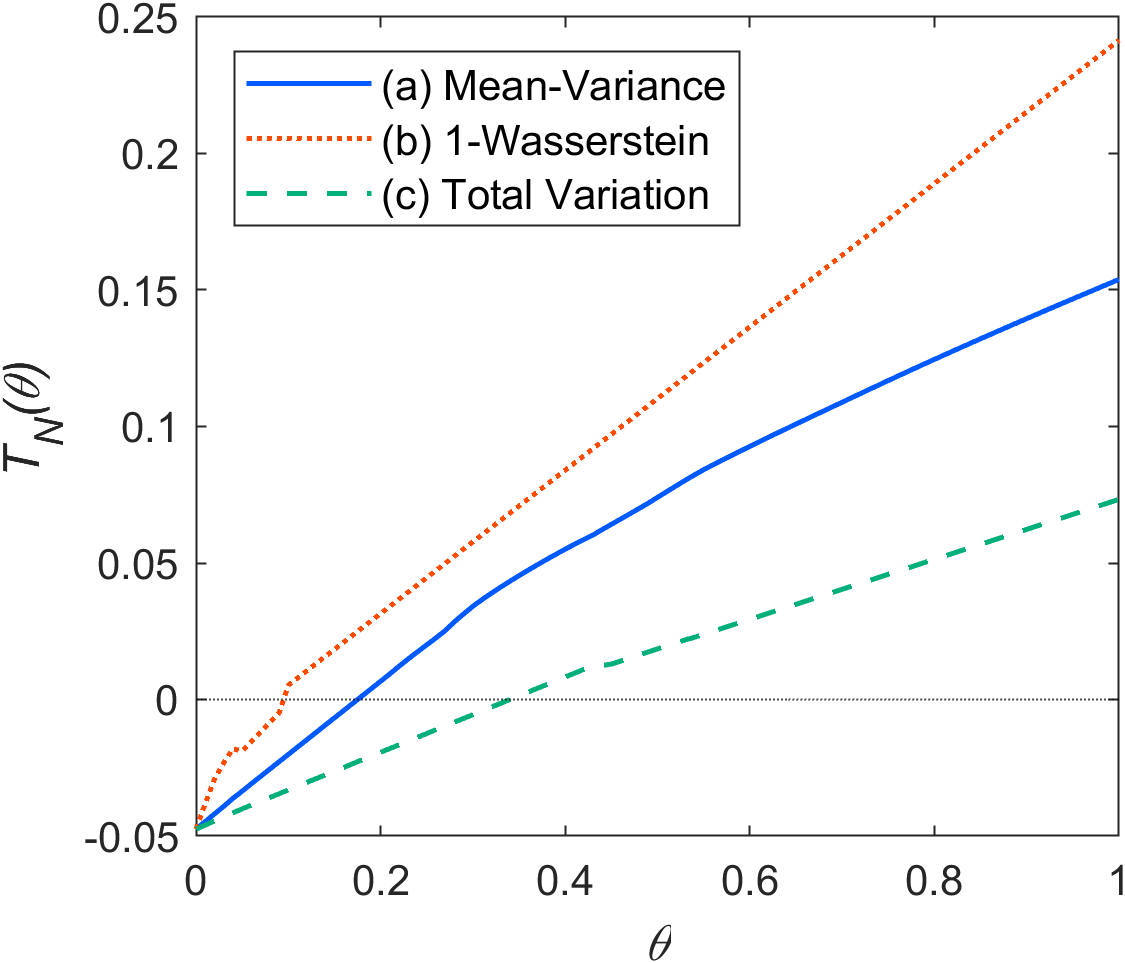}
        \caption{Portfolio Optimization}  \label{fig:conservatism_PO}
    \end{subfigure}
    \caption{$T_N(\theta)$ for different values of $\theta$} \label{fig:conservatism_num_expt}
\end{figure}

We follow the same experimental settings detailed in Sections~\ref{subsec:inventory_control_problem} and~\ref{subsec:portfolio_optimization_problem} to obtain spectra of optimal solutions to our TRO models and compute the associated $T_N(\theta)$.  Figures~\ref{fig:conservatism_IC} and \ref{fig:conservatism_PO} present the values of $T_N(\theta)$ for the inventory control problem and portfolio optimization problem, respectively. Clearly, $T_N(0)$ is negative, suggesting that the SAA solutions $\upsilonh_N(0)$ may be optimistic. Moreover, except for the TRO model of the inventory control problem that employs set (d) as the shape parameter, the value of $T_N(\theta)$ increases with $\theta$. This indicates that the TRO model generates more conservative solutions as $\theta$ grows, which is consistent with the theoretical results and discussions in Section~\ref{sec:conservatism_TRO}. For example, $T_N(1)$ associated with the DRO (TRO with $\theta=1$) solutions to the portfolio optimization problem (Figure~\ref{fig:conservatism_PO}) is fairly large, suggesting that the DRO solutions $\xb_N(1)$ using sets (a)--(c) may be overly conservative. Solutions with an intermediate value of $\theta$ have a larger absolute value of $T_N(\theta)$ than the SAA solutions but a smaller  $T_N(\theta)$ than the DRO solutions. This indicates that these solutions are less optimistic than the SAA solutions and less conservative than the DRO solutions. In the next section, we demonstrate that the out-of-sample performance of TRO solutions obtained with an intermediate value $\theta$ is better than those of the SAA and DRO solutions. Finally, we note that in the inventory control problem, since the optimal value $\upsilonh_N(\theta)$ of TRO model with TRO ambiguity set constructed using set (d) is first increasing and then decreasing (see discussions in Section~\ref{subsec:inventory_control_problem}), the resulting $T_N(\theta)$ follows the same pattern.

\subsection{Out-of-Sample Performance} \label{apdx:add_expt_results:out_of_sample}

In this section, we analyze the performance of the spectrum of optimal solutions to the TRO model via out-of-sample testing. First, we follow the same procedure detailed in Sections~\ref{subsec:inventory_control_problem} and~\ref{subsec:portfolio_optimization_problem} to obtain the spectrum of TRO solutions $\{\xb_N(\theta)\mid \theta\in\Theta\}$, where $\Theta:=\{0,0.01,0.02,\dots,1\}$. Then, we compute  the out-of-sample cost $\upsilonh_N^\text{OS}(\theta)=\E_{\Prob^\text{OS}}[f(\xb_N(\theta),\xib)]$ of the TRO solutions for all $\theta\in\Theta$, where $\Prob^\text{OS}$ denotes the out-of-sample distribution. For the inventory control problem, we choose $\Prob^\text{OS}$ as the exponential distribution with mean $50+\Delta$ and $\Delta\in\{-20,-19.8,-19.6,\dots,10\}$, where a positive (resp. negative) value of $\Delta$ corresponds to an increase (resp. decrease) in mean. For the portfolio optimization problem, we choose $\Prob^\text{OS}$ as the multivatiate normal distribution described in Section~\ref{subsec:portfolio_optimization_problem}, and we change the mean from $\mub$ to $\mub+\Delta\one$ with $\Delta\in\{-0.1,-0.098,-0.096,\dots,0.05\}$. Finally, we search for $\theta_\text{min}:=\argmin\{\theta\in\Theta\mid \upsilonh_N^\text{OS}(\theta)\leq\upsilonh_N^\text{OS}(\theta'),\,\forall\theta'\in\Theta\}$, i.e., the value of $\theta$ that gives the smallest out-of-sample cost.
\begin{figure}[t!]
    \centering
    \begin{subfigure}[t]{0.5\textwidth}
        \centering
        \includegraphics[scale=0.75]{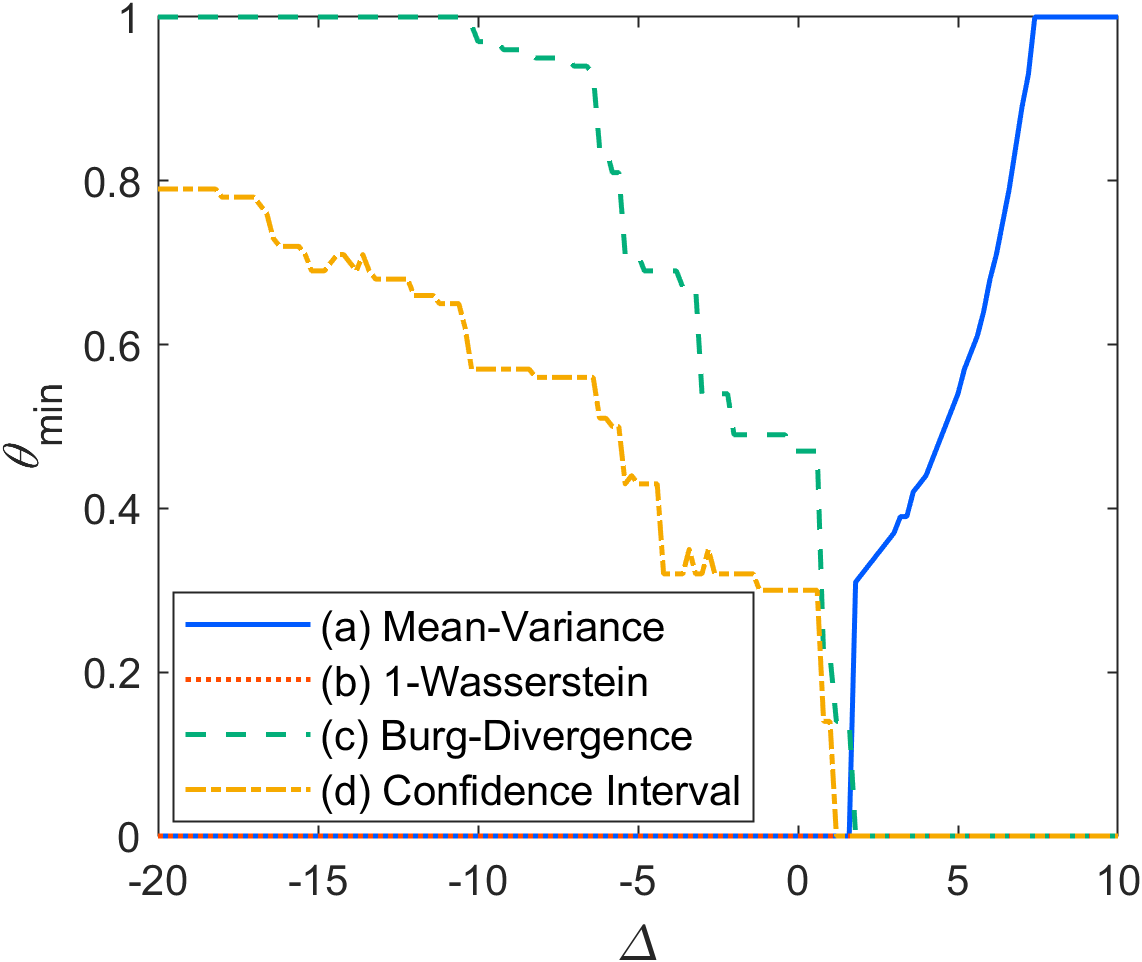}
        \caption{Inventory Control}  \label{fig:OS_IC}
    \end{subfigure}%
    ~ 
    \begin{subfigure}[t]{0.5\textwidth}
        \centering
        \includegraphics[scale=0.75]{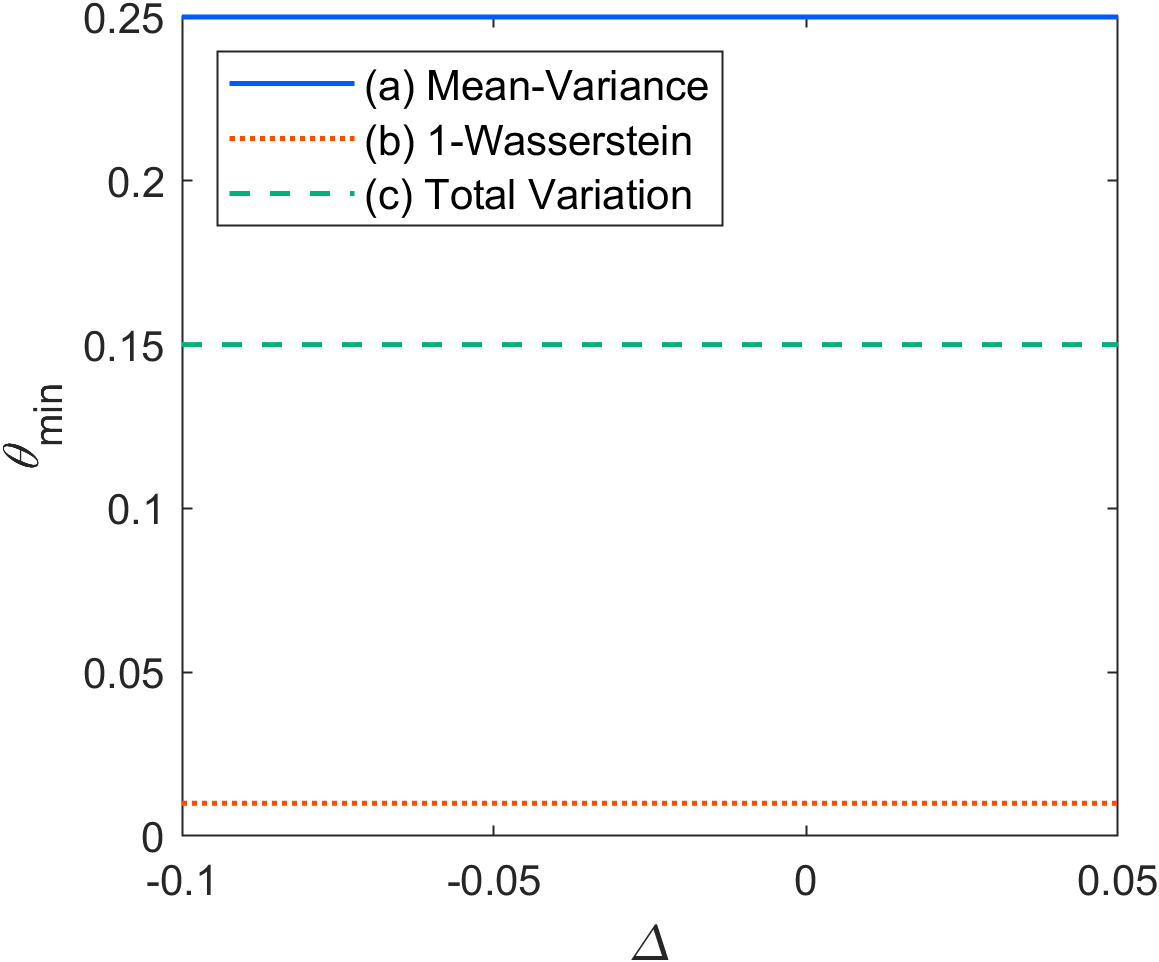}
        \caption{Portfolio Optimization}  \label{fig:OS_PO}
    \end{subfigure}
    \caption{$\theta_\text{min}$ for different values of perturbation $\Delta$} \label{fig:OS}
\end{figure}

Figures~\ref{fig:OS_IC} and \ref{fig:OS_PO} present the values of $\theta_\text{min}$ in the inventory control problem and portfolio optimization problem, respectively. Not surprisingly, the value of $\theta_\text{min}$ varies across different shape parameters employed in the TRO model. Moreover, the value of $\theta$ that yields the lowest out-of-sample costs is generally neither 0 nor 1 but lies strictly between these two extremes. This highlights that adopting solutions from the spectrum of TRO optimal solutions results in smaller out-of-sample costs than the SAA and DRO solutions. Consider the inventory control problem, for example  (see Figure~\ref{fig:OS_IC}). When $\Delta\in[1.8,7.2]$, the TRO model for this problem with set (a) as the shape parameter and intermediate values of $\theta$ yields the smallest out-of-sample costs. The TRO model with sets (c) and (d) and intermediate values of $\theta$ yield the smallest out-of-sample costs when $\Delta\in[-10.2,1.6]$ and $\Delta\in[-20,1]$, respectively. Note that the optimal solution to our TRO model using set (b) is always the same for any $\theta\in[0,1]$, i.e., there is only one solution for this problem.  Similarly, for the portfolio optimization problem (see Figure~\ref{fig:OS_PO}), our TRO model with sets (a), (b), and (c) and  $\theta=0.25$, $\theta=0.01$, and $\theta=0.15$ yields the smallest out-of-sample costs, respectively.  Note that $\theta_\text{min}$ is the same for all $\Delta\in[-0.1,0.05]$ since adding a constant $\Delta$ to each entry of the mean vector $\mub$ shifts the entire objective function of this problem by a constant. These results demonstrate the benefits of adopting solutions on the spectrum of TRO optimal solutions over the SAA and DRO solutions.


\section{Comparison with \protect\cite{Wang_et_al:2023}}  \label{apdx:Wang_et_al}

In this appendix, we provide a detailed comparison between our work and that of \cite{Wang_et_al:2023}.

\begin{itemize}[leftmargin=5mm]
    \item \textit{The Proposed Model.} While our TRO model looks similar to \cite{Wang_et_al:2023}'s model at the outset, our work actually generalizes the idea of \cite{Wang_et_al:2023}. Recall our TRO model in \eqref{model:trade-off_model}, which is equipped with the TRO ambiguity set $\calP'_{N,\theta}$ in \eqref{eqn:trade-off_ambig_set}. The TRO ambiguity set \eqref{eqn:trade-off_ambig_set} is characterized by two parameters: the \textit{shape} parameter $\calP_N$ and the \textit{size} parameter $\theta$. The shape parameter $\calP_N$ represents distributional ambiguity and could be any (data-driven) ambiguity set satisfying some mild assumptions mentioned in the paper. The size parameter $\theta\in[0,1]$ controls the level of optimism, i.e., it controls the trade-off between solving the problem under a distributional belief and solving it under ambiguity. By choosing the shape parameter $\calP_N$ as the distance-based ambiguity set  $B_{\epsilon_N}(\Probh_N)=\{\Prob\in\calP(\Xi)\mid \Delta(\Prob,\Probh_N)\leq \epsilon\}$ (where $\calP(\Xi)$ is the set of probability measures on the support $\Xi$ and $\Delta$ is a statistical distance), our TRO model reduces to the Bayesian distributionally robust (BDR) optimization model proposed in \cite{Wang_et_al:2023} (see equation (11) in \citealp{Wang_et_al:2023}), i.e., model \eqref{model:trade-off_model} reduces to 
    \begin{equation} \label{model:BDR}  \tag{BDR} 
        \underset{\xb\in\calX}{\text{minimize}}\quad \beta_N \max_{\Prob\in B_{\epsilon_N}(\Probh_N)} \E_{\Prob}[f(\xb,\xib)] + (1-\beta_N)\E_{\Probh_N}[f(\xb,\xib)].
    \end{equation}
    Thus, \cite{Wang_et_al:2023}'s \ref{model:BDR} model is a special case of our TRO model. In particular, we emphasize that one can construct the TRO ambiguity set $\calP'_{N,\theta}$ using any shape parameter $\calP_N$, including general moment- and distance-based ambiguity sets. Hence, our theoretical results are valid for various types of the shape parameter. In contrast, \cite{Wang_et_al:2023} analyses are limited to the case where $\calP_N$  is a distance-based ambiguity set. 


    \item \textit{Hierarchical Properties}. In Section~\ref{sec:TRO_ambig_set_property}, we analyze properties of the TRO ambiguity set $\calP'_{N,\theta}$ defined in \eqref{eqn:trade-off_ambig_set} and the sequence of TRO ambiguity sets $\{\calP'_{N,\theta}\mid \theta\in[0,1]\}$. These were not analyzed in \cite{Wang_et_al:2023}.  We first formally introduce the notion of \textit{hierarchical properties} of the sequence of the TRO ambiguity sets $\{\calP'_{N,\theta}\mid\theta\in[0,1]\}$; see Definition~\ref{def:Hierarchical}. The hierarchical properties indicate that the size of the TRO ambiguity set $\calP'_{N,\theta}$ increases with $\theta$, i.e., $\calP'_{N,\theta}$ contains more distributions with a larger $\theta$. This implies that the TRO model is more conservative when we pick a larger $\theta$.  Then, in Theorem~1, we provide necessary and sufficient conditions for the sequence of TRO ambiguity sets $\{\calP'_{N,\theta}\mid \theta\in[0,1]\}$ to satisfy these properties. Specifically, Theorem~1 establishes that constructing the TRO ambiguity set $\calP'_{N,\theta}$ using a star-shaped shape parameter $\calP_N$ with a star center $\Probh_N$ is necessary and sufficient for the sequence of TRO ambiguity sets $\{\calP'_{N,\theta}\mid \theta\in[0,1]\}$ to satisfy the hierarchical property. Part (i) shows that for a general star-shaped shape parameter $\calP_N$, the TRO ambiguity set $\calP'_{N,\theta}$ is non-decreasing in $\theta$, i.e., $\calP'_{N,\theta_1}\subseteq \calP'_{N,\theta_2}$ whenever $\theta_1\leq\theta_2$, indicating that the objective function of the trade-off model \eqref{model:trade-off_model} is non-decreasing in $\theta$. Part (ii) illustrates the relationship between the sets $\{\Probh_N\}$, $\calP'_{N,\theta_1}$, $\calP'_{N,\theta_2}$, and $\calP_N$ with $0<\theta_1<\theta_2<1$. Specifically, part (ii) shows how the TRO ambiguity set  $\calP'_{N,\theta}$ enlarges with $\theta$. This, in turn, implies that the TRO model is more conservative when we pick a larger $\theta$. These important new results establish the connection between the specific choice of $\theta$ and the conservatism of the TRO model. However, they were not analyzed in \cite{Wang_et_al:2023}. In addition, in Proposition~\ref{prop:star_shape_distance}, we derive necessary and sufficient conditions for a general distance-based ambiguity set, i.e., the one adopted in the \ref{model:BDR} model, to be star-shaped with a star center $\Probh_N$. Thus, our results provide necessary and sufficient conditions under which the sequence of ambiguity sets corresponding to the \ref{model:BDR} model satisfies the hierarchical property. This was not studied in \cite{Wang_et_al:2023}. Moreover, as mentioned in the first point, our findings can be applied to TRO models with a general shape parameter. For instance, in Proposition~\ref{prop:star_shape_moment}, we establish the corresponding conditions for a general moment-based ambiguity set to be star-shaped, with a star center represented by $\Probh_N$, a novel contribution not explored in previous studies.

    \item \textit{Analysis of the Conservatism.} In Section~\ref{sec:conservatism}, we investigate the conservatism and properties of the optimal value $\upsilonh_N(\theta)$ and the set of optimal solutions $\calXh_N(\theta)$ of the TRO model through the lens of quantitative stability analysis.  First, in Theorem~\ref{thm:sensitivity_in_theta}, we establish mechanisms to quantify the difference in  $\upsilonh_N(\theta)$ and  $\calXh_N(\theta)$ (and hence conservatism) incurred by perturbation in $\theta$. In particular, it shows that $\upsilonh_N(\theta)$ is Lipschitz continuous in $\theta$ and $\calXh_N(\theta)$ is H\"{o}lder continuous with H\"{o}lder exponent $1/2$ under distance $D$. This shows that both the optimal value and the set of optimal solutions change gradually with $\theta\in[0,1]$. Then, in Theorem~\ref{thm:conservatism}, we show that naively combining SAA and DRO optimal solutions, e.g., via a convex combination after solving each separately, may not yield a feasible solution to the TRO problem \eqref{model:trade-off_model}. This is particularly true in applications where $\calX$ is not convex (e.g., problems involving integer variables such as facility location and scheduling problems). Specifically, part (i) of Theorem~\ref{thm:conservatism} establishes that the optimal value $\upsilonh_N(\theta)$ to our TRO model is not less than the convex combination $(1-\theta)\upsilonh_N(0)+\theta\upsilonh_N(1)$ of the SAA and DRO optimal values. In addition, if $\Probh_N\in\calP_N$, Theorem~\ref{thm:conservatism} implies that $\upsilonh_N$ is non-decreasing in $\theta$ as illustrated in Figure~\ref{fig:conservatism_plot}. Part (ii) of Theorem~\ref{thm:conservatism} indicates that the set of optimal solutions $\calXh_N(\theta)$ to our TRO model can be approximated by $\overline{\calX}_N(\theta):=(1-\theta)\calXh_N(0)+\theta\calXh_N(1)$ only when $\theta$ is close to zero or one; however, the difference could be huge for intermediate values of $\theta \in (0, 1)$. These important investigations and results are new. In particular, \cite{Wang_et_al:2023} only suspected that their \ref{model:BDR} model is likely to be less conservative than the DRO model without providing any theoretical analysis. Indeed, we could apply our results to show that by changing $\beta_N$ in \cite{Wang_et_al:2023}'s \ref{model:BDR} model, one can obtain a spectrum of optimal solutions, ranging from optimistic to conservative solutions.

    \item \textit{Finite-Sample Properties.} As discussed in the first paragraph of Section~\ref{subsec:bias_analysis}, the optimal value of our TRO model $\widehat{\upsilon}_N(\theta)$ represents an estimator of the true optimal value of the stochastic optimization problem $\upsilon^\star = \inf_{\xb\in\calX} \E_{\Prob^\star} \big[f(\xb,\xib)\big]$.  Analyzing the bias of an estimator to the true optimal value $\upsilon^\star$ is common in the related literature; see, e.g.,  \cite{Blanchet_et_al:2019a, Dentcheva_Lin:2022}. Also, it is well known that the SAA estimator $\upsilonh_N(0)$ is a downward biased estimator of $\upsilon^\star$; see~\eqref{eqn:SAA_downward_bias}. Thus, we and \cite{Wang_et_al:2023} analyze the bias of the TRO and BDR estimators, respectively. Specifically, we and \cite{Wang_et_al:2023} show that there exists $\theta^\text{u}_N\in[0,1]$ in our TRO model and $\beta^\text{u}_N\in[0,1]$ in the \ref{model:BDR} model such that the optimal values of the TRO model and the \ref{model:BDR} model are unbiased estimators of $\upsilon^\star$. However, in Section~\ref{subsec:bias_analysis}, we provide a more detailed investigation of the bias of the more general model, the TRO model, as well as new results. First, in Proposition~\ref{prop:bias_non_neg_UB}, we derive an upper bound on the bias of the TRO estimator $\upsilonh_N(\theta)$. It suggests that the bias of $\upsilonh_N(\theta)$ may not be a downward bias as that of the SAA estimator. Second, in Corollary~\ref{cor:debias_theta}, we show that for sufficiently small $\theta$, the TRO estimator $\upsilonh_N(\theta)$ has a smaller bias than the SAA estimator, which was not discussed in \cite{Wang_et_al:2023}. Third, we show that $\E_{\Prob_N}[\upsilonh_N(\theta)]$ can be decomposed as the sum of three terms: (a) the expected value of the SAA estimator $\E_{\Prob^N}[\upsilonh_N(0)]$, (b) the DRO effect $\theta\big\{\E_{\Prob^N}\big[\upsilonh_N(1)\big]-\E_{\Prob^N}\big[\upsilonh_N(0)\big]\big\}$, and (c) the concavity effect $R_N(\theta)$ (see Figure~\ref{fig:bias_reduction}). Finally, as pointed out by  \cite{Wang_et_al:2023}, parameter $\beta^\text{u}_N$ is typically hard to obtain. Similarly, $\theta^\text{u}_N$ is hard to obtain. However,  in Theorems~\ref{thm:rate_of_theta_LIL} and \ref{thm:rate_of_theta_AN}, we analyze the asymptotic behavior of $\theta_N^\text{u}$. In particular, we prove the convergence of $\theta_N^\text{u}$ as $N\rightarrow\infty$ and derive its convergence rate. \cite{Wang_et_al:2023} did not conduct such analyses. In Section~\ref{subsec:generalization}, we derive the generalization bound for our TRO model in Theorem~\ref{thm:generalization_error} based on that of the SAA and DRO models. In particular, we show that the probability \eqref{eqn:generalization_error} (i.e., the generalization error) is upper bounded by the sum of the probabilities $\alpha_{N,1}$ in \eqref{eqn:generalization_error_bound_SAA} from the SAA model and  $\alpha_{N,2}$ in \eqref{eqn:generalization_error_bound_DRO} from the DRO model. This indeed corrects the generalization bound derived in Theorem 3.5 of \cite{Wang_et_al:2023}, where they did not take the sum of the two probability bounds from the SAA and DRO models; see Appendix C.6 of \cite{Wang_et_al:2023}. In addition, we show that for specific choices of the shape parameter $\calP_N$, such as popular distance-based ambiguity sets as the one employed in \cite{Wang_et_al:2023}'s \ref{model:BDR} model, the generalization error exhibits an exponentially decaying tail. This important and attractive finite-sample property was not mentioned in \cite{Wang_et_al:2023}. 
    
    \item \textit{Asymptotic Convergence.} In Section~\ref{sec:asymptotic}, we show the almost sure convergence of the optimal value $\upsilonh_N(\theta_N)$ and the set of optimal solutions $\calXh_N(\theta_N)$ of the TRO model to their true counterparts when $N\rightarrow\infty$, and we derive the asymptotic distribution of $\upsilonh_N(\theta_N)$ when $N\rightarrow\infty$. Our asymptotic convergence results hold for TRO models with TRO ambiguity sets constructed using general shape parameters $\calP_N$, such as moment- and distance-based ambiguity sets. This differs from results in the existing literature focusing on a specific ambiguity set, including \cite{Wang_et_al:2023}. Note that the asymptotic convergence and distribution are two basic asymptotic properties that are commonly analyzed in the existing literature for data-driven optimization models \citep{Blanchet_Shapiro:2023, Kuhn_et_al:2019, Shapiro_et_al:2014}. Hence, \cite{Wang_et_al:2023} analyzed these asymptotic properties of their \ref{model:BDR} model. However, the following are some differences between our analyses and those of \cite{Wang_et_al:2023}.
    
    \begin{itemize}[leftmargin=5mm, topsep=0mm]
        \item  First, \cite{Wang_et_al:2023} assumed the DRO objective $\sup_{\Prob\in B_{\epsilon_N}(\Probh_N)} \E[h(\xb,\xib)]$ is $\Prob^\star$-bounded and attainable for $\xb\in\calX$, where we recall that $\Prob^\star$ is the true distribution; see assumption C1 of Theorem~3.3 in \cite{Wang_et_al:2023}. To justify this assumption, \cite{Wang_et_al:2023} provided one example based on the Wasserstein ambiguity set; see Appendix~C.2 in \cite{Wang_et_al:2023}. Our convergence analyses also require the DRO objective to be upper-bounded for sufficiently large $N$. However, we adopt a more general set of assumptions under which the desired boundedness condition holds. Specifically, we impose assumptions on the objective function or the sequence of the ambiguity sets $\{\calP_N\}_{N\in\N}$; see Assumption~\ref{assumption:ambig_set_regularity}. Examples~\ref{eg:ambig_set_seq_1}--\ref{eg:ambig_set_seq_5} provide a wide range of settings under which Assumption~\ref{assumption:ambig_set_regularity} holds. These examples include the case where $\calP_N$ is constructed based on the Wasserstein ambiguity sets, as well as other distance-based and moment-based ambiguity sets. In Lemma~\ref{lem:DRO_component_asym_boundedness}, we formally prove that under Assumption~\ref{assumption:ambig_set_regularity}, the DRO objective is \textit{asymptotically} bounded, which is weaker than assumption C1 adopted by \cite{Wang_et_al:2023} (which requires the DRO objective to be bounded for all $N\in\N$). 

        \item Second, in Theorem~\ref{thm:asymptotic_convergence}, we prove the almost-sure convergence of our TRO model. Specifically, we show that the optimal value $\upsilonh_N(\theta_N)$ and the set of optimal solutions $\calXh_N(\theta_N)$ of our TRO model converges almost surely to the true optimal value $\upsilon^\star$ and the set of optimal solutions $\calX^\star$ to \eqref{prob:SO}, respectively. In contrast, in Theorem~3.3 of \cite{Wang_et_al:2023}, they only provided the convergence of the optimal value and the set of optimal solutions of the \ref{model:BDR} model in probability, which is weaker than our almost-sure convergence. We would like to highlight that the asymptotic convergence holds for TRO models with TRO ambiguity sets constructed using general shape parameters, such as moment-based ambiguity sets. This differs from the existing convergence results established for data-driven DRO models, which mainly employ distance-based ambiguity sets.

        \item Third, in Theorem~3.3 of \cite{Wang_et_al:2023}, they derived the asymptotic normality of the optimal value of the \ref{model:BDR} model. To show this, they assumed that the optimal solution $\widehat{\xb}_{b,N}$ to the \ref{model:BDR} model converges (in probability) to an optimal solution $\xb_0$ to the stochastic optimization problem under the true distribution $\Prob^\star$. This assumption is not common in relevant literature when deriving the asymptotic distribution of the optimal value of data-driven optimization models \citep{Blanchet_Shapiro:2023, Guigues_et_al:2018, Shapiro_et_al:2014}. In contrast, in Theorem~\ref{thm:asy_dist}, we derive the asymptotic distribution of $\upsilonh_N(\theta_N)$ without imposing such an assumption. Specifically, we apply the Delta's method (see, e.g., \citealp{Carcamo_et_al:2020}) to derive the following asymptotic distribution of the optimal value $\upsilonh_N=\upsilonh_N(\theta_N)$ of the TRO model: $\sqrt{N}(\upsilonh_N-\upsilon^\star) \Rightarrow \inf_{\xb\in \calX^\star} \G(\xb)$, where $\G$ is a tight Gaussian process indexed by $\calX$ with mean zero and covariance function $\Cov(\G(\xb_1)),\G(\xb_2))=\Cov_{\Prob^\star}(f(\xb_1,\xib),f(\xb_2,\xib))$. Thus, our results on the asymptotic distribution of the optimal value are different from those of \cite{Wang_et_al:2023}. 

        \item Finally, for the special case when the shape parameter is chosen as some popular distance-based ambiguity set $\calP_{N,r_N}=\{\Prob\in\calP(\Xi)\mid \sfd(\Prob,\Probh_N)\leq r_N\}$ as in \cite{Wang_et_al:2023}'s \ref{model:BDR} model, we can recover the asymptotics of the optimal value in classical distance-based DRO models. Specifically, in Theorem~\ref{thm:asy_dist_distance_based}, we derive the asymptotic distribution of the optimal value $\upsilonh_N(\theta_N)$ of our TRO model under different convergence rates of the size parameter $\theta_N$ and the radius $r_N$ in the shape parameter $\calP_{N,r_N}$. This generalizes the asymptotic convergence of the optimal solution derived in Theorem~3.3 of \cite{Wang_et_al:2023}. In particular, the asymptotic distribution results in Theorem~3.3 of \cite{Wang_et_al:2023} essentially correspond to the case (i) in Theorem~\ref{thm:asy_dist_distance_based} only. Cases (ii) and (iii) in Theorem~\ref{thm:asy_dist_distance_based} were not investigated in \cite{Wang_et_al:2023}. Moreover, we also investigate the connection between the size parameter $\theta_N$ in our TRO model (with TRO ambiguity set constructed using the shape parameter $\calP_{N,r}$) and the radius $r_N$ in classical distance-based DRO models. This was not analyzed in \cite{Wang_et_al:2023}.
    \end{itemize}
    
\end{itemize}

\newpage
\bibliographystyle{elsarticle-harv}
\bibliography{references}


\end{document}